\pdfoutput=1
\documentclass[11pt,twoside]{article}

\usepackage[margin=1in]{geometry}
\usepackage[T1]{fontenc}
\usepackage[utf8]{inputenc}
\usepackage{amsmath,amssymb,amsfonts,amsthm,amsopn}
\usepackage{bm}
\usepackage{braket}
\usepackage{mathrsfs}
\usepackage{stmaryrd}
\usepackage{array}
\usepackage{multirow}
\usepackage{makecell}
\usepackage{diagbox}
\usepackage{booktabs}
\usepackage{graphicx}
\usepackage{float}
\usepackage{cite}
\usepackage[hidelinks]{hyperref}
\usepackage[nameinlink,capitalize]{cleveref}
\usepackage{fancyhdr}

\numberwithin{equation}{section}
\allowdisplaybreaks

\theoremstyle{plain}
\newtheorem{theorem}{Theorem}[section]
\newtheorem{lemma}[theorem]{Lemma}

\theoremstyle{definition}

\newtheorem{example}[theorem]{Example}

\theoremstyle{remark}
\newtheorem{remark}[theorem]{Remark}

\crefname{hypothesis}{Hypothesis}{Hypotheses}
\Crefname{hypothesis}{Hypothesis}{Hypotheses}

\newcommand{\vertiii}[1]{{\left\vert\kern-0.25ex\left\vert\kern-0.25ex\left\vert #1 
        \right\vert\kern-0.25ex\right\vert\kern-0.25ex\right\vert}}

\newcommand{\email}[1]{\texttt{#1}}

\providecommand{\color}[1]{}
\newenvironment{keywords}{\par\medskip\noindent\textbf{Keywords.}\ }{\par\medskip}

\title{A superconvergent hybridizable discontinuous Galerkin method for the convective Cahn--Hilliard equation\thanks{G. Chen and D. Zuo are supported by the National Natural Science Foundation of China (NSFC) under Grant Nos. 121713413 and 12422115, and the Jiangsu Provincial Scientific Research Center of Applied Mathematics under Grant No. BK20233002. D. Han acknowledges support from National Science Foundation grant DMS-2310340. Y. Zhang is supported by National Science Foundation grant DMS-2111315.}}

\author{Gang Chen\thanks{Department of Mathematics, Sichuan University (\email{cglwdm@scu.cn}, \email{zuodujin@stu.scu.edu.cn}).}
\and Daozhi Han\thanks{Department of Mathematics, The State University of New York at Buffalo, Buffalo, NY (\email{daozhiha@buffalo.edu}).}
\and Jiaxuan Liu\thanks{School of Mathematics and Statistics, Central South University, Changsha, Hunan (\email{jiaxuanliu@csu.edu.cn}).}
\and Yangwen Zhang\thanks{Department of Mathematics, University of Louisiana at Lafayette, Lafayette, LA (\email{yangwen.zhang@louisiana.edu}).}
\and Dujin Zuo\footnotemark[2]}

\date{}

\begin{document}
\maketitle
\begin{abstract}  
		We propose a hybridizable discontinuous Galerkin (HDG) method combined with convex-concave splitting for the temporal discretization of the convective Cahn-Hilliard equation. The convection term is discretized explicitly without stabilization,  yielding three key advantages: (1) unconditional stability,  (2) preservation of the optimal convergence rate for piecewise constant approximations,  and (3) a symmetric system after local elimination,  enabling efficient solver via minimal residual methods. We establish optimal convergence rates in the $L^2$ norm for both the scalar and flux variables for any polynomial degree $k \geq 0$. To achieve optimal $L^2$-norm estimates,  we introduce a specialized HDG elliptic projection operator and analyze its approximation properties. Within the HDG framework,  local elimination is employed to reduce the degrees of freedom associated with the globally coupled unknowns,  and the scalar variables exhibit superconvergence. Finally,  numerical experiments validate the theoretical convergence rates and demonstrate the effectiveness of the proposed method.
	\end{abstract}	
	

	\section{Introduction}
	Let $\Omega\subset \mathbb{R}^d$ ($d=2, 3$) be a {convex} polygonal domain with a Lipschitz boundary $\partial\Omega$,   and let $T$ be a positive constant.
	We consider the following Cahn-Hilliard equation:
	\begin{subequations}\label{ori}
		\begin{align}
			u_t-\frac{1}{Pe}\Delta\phi + \nabla \cdot (\bm{\beta} u)= 0, \;\; -\epsilon^2\Delta u+f(u)= \phi  \;\;\text{in} \; \Omega\times(0, T], \label{o2}\\
			\nabla u\cdot\bm n=\nabla\phi\cdot\bm n= 0 \;\;\text{on} \; \partial\Omega\times(0, T], \quad 
			u(\cdot, 0)= u^0(\cdot)  \;\; \text{in}\; \Omega, \label{o4}
		\end{align}
	\end{subequations}
	where $f(u)=u^3-u$,  $Pe$ is the P\'{e}clet number,  and the known velocity field $\bm{\beta} \in C(\overline{\Omega})$ is solenoidal (divergence-free) and satisfies no-penetration boundary condition $\bm{\beta} \cdot \bm {n}=0$ on $\partial \Omega$.  Here,  $u$ represents  the order parameter,  $\phi$ the chemical potential,  and $\epsilon$   the Cahn number,   a dimensionless measure of interface thickness.

	Originally introduced as a phenomenological model for phase separation in binary alloys,  the Cahn-Hilliard equation has since been widely used in studying interface dynamics,  phase transitions,  multiphase fluids,  and polymer solutions. Due to its significance in material science and multiphase flow,  extensive research has been devoted to establishing and analyzing stable and convergent numerical schemes for solving this equation,  including finite difference methods \cite{ChengKelong2019Aesf, HongJialin2024Dcoa, BomaWilcox2021ANIo, HamSeokjun2022AEAF, FukaoTakeshi2017SFDS, HuangXin2023PSfd},  finite volume methods \cite{NabetFlore2021Aeef, NabetFlore2016Coaf, Cueto-FelguerosoLuis2008Atfv, DargavilleS.2015Alsb},  Fourier spectral methods \cite{LiDong2017OSOS, ZhengNan2020Esac, MR3225505, KrischokA.2024Fius, ZengShilin2023Fdda},  mixed finite element methods,  nonconforming finite element methods \cite{DiegelAmandaE.2016Saco, BrennerSusanneC.2018ARSf, BrennerSusanneC.2020Arsf, WangWansheng2016PMFE, FengXiaobing2020AFDM, ElliottCharlesM.1989ANFM, ZhangShuo2010Anfe},  and discontinuous Galerkin (DG) methods \cite{MR3407258, MR2551193, MR2299767}.  These works introduce hybridization techniques for interior penalty discontinuous Galerkin (IPDG) methods and provide an insightful and technically valuable perspective on the subject \cite{MR4803805, MR4360352}.
	
	More recently,  researchers have explored more complex interface dynamics and multiphase flow problems by incorporating fluid flow into the Cahn-Hilliard system,  leading to models of the Cahn-Hilliard-Navier-Stokes type (cf. \cite{BadalassiV.E.2003Coms, MR2077449, JacqminDavid1999CoTN, KimJunseok2004Cmmf, KimJunseok2005Admf, CaiWentao2023ObEE, Acosta-SobaDaniel2025Pnao, LiuChen2023CoaD, FuGuosheng2020AdHs}). These models introduce a convection term into the Cahn-Hilliard equation. Various numerical methods have been proposed to solve the convective Cahn-Hilliard equation,  including conservative multigrid methods \cite{KimJunseok2004Cmmf},  isogeometric analysis \cite{LiuJu2013Iaot},  finite volume and DG methods \cite{MR3771451},  operator-splitting methods \cite{GideyH.H.2019Omft},  Fourier spectral methods \cite{MR3781562},  hybrid high-order methods \cite{MR3662245},  and others \cite{GruenG.2013OCSF, AppaduA.R.2017Aomf, BadalassiV.E.2003Coms, FengYu2020PSit, MR4544721}.  In convection-dominated cases,  where P\'{e}clet is large,  continuous Galerkin (CG) methods often lead to poor approximations unless numerical stabilization techniques such as streamline-diffusion or least-squares stabilization are applied. A natural alternative is the DG method.  In \cite{KAYDAVID2009DGFE} Kay et al.  proposed a DG scheme for the Cahn-Hilliard equation with convection,  establishing optimal error convergence in $L^{\infty}\left(H^1\right)$ for the order parameter and $L^2\left(H^1\right)$ for the chemical potential.  It should be emphasized that despite the wide availability of optimal $H^1$ error estimates for the Cahn-Hilliard-Navier-Stokes equations,  cf. \cite{CaiYongyong2018EEfa},  optimal error estimates in the $L^2$ norm are only recently  obtained  in \cite{CaiWentao2023ObEE} for CG and in \cite{LiuChen2023CoaD} for DG.
	
	However,  classical DG methods require a significantly larger number of degrees of freedom compared to CG methods. To address this,  hybridizable discontinuous Galerkin (HDG) methods were introduced by Cockburn,  Gopalakrishnan,  and Lazarov \cite{MR2485455}. HDG methods,  based on a mixed formulation,  introduce numerical flux and trace variables,  reducing the globally coupled degrees of freedom. These methods retain the advantages of DG methods while improving computational efficiency,  making them well-suited for convection-diffusion problems
	\cite{MR3463449, MR2991828, MR3120583, MR3342199, MR3440284}. HDG methods also inherit desirable properties from classical mixed methods,  such as optimal convergence rates and superconvergent postprocessing capabilities \cite{MR3044180, MR3813574, MR2772094, MR3831243, MR3833696, MR2911790, MR3000935}.  In our earlier work \cite{CHSZ2023} a super-convergent HDG scheme is designed for solving the Cahn-Hilliard equation with
	the Lehrenfeld-Sch\"{o}berl type stabilization using polynomials of order $k+1$ for the scalar variables,  and polynomials of order $k$ for the other unknowns. The scheme has an equivalent  formulation of a hybrid high order method \cite{MR3662245}.  By leveraging a HDG spectral estimate and constraints on space-time grid sizes,  optimal convergence rates in the $L^2$ norm are obtained for all variables and for polynomials of order $k\geq0$ with error constants depending polynomially on inverse of $\epsilon$.  
	
	In this article,  we extend the super-convergent HDG framework in \cite{CHSZ2023} to the convective Cahn-Hilliard equation. While the emphasis of \cite{CHSZ2023} is on the derivation of error constants depending on $1/\epsilon$ in polynomial order,  this article focuses on the HDG discretization of the convective term in order to obtain optimal convergence rates for all variables. There are three major  difficulties in the design and analysis of the algorithm.  First,  we note that in transport-coupled phase-field computations,  upwind/monotone mechanisms and bound-preserving treatments of the convection are often employed to suppress nonphysical oscillations and to facilitate maximum-principle enforcement,  especially in variable/degenerate mobility and CHNS-type settings; see,  e.g., \cite{MR1742748, MR4796135, MR4754742}. While upwind numerical flux is widely used in HDG-type methods for PDEs with convection,  it can lead to a loss of optimal convergence rates for piecewise constant approximations,  and piecewise constant approximations are particularly advantageous for nonlinear problems with sharp transition layers. 
		See \cite{MR3440284} for convection--diffusion PDEs,  \cite{MR3833696} for an optimal control problem governed by convection--diffusion PDEs,  and \cite{MR3556409} for the Navier--Stokes equations. The same issue arises in the convective Cahn--Hilliard equation when applying upwind numerical flux to the convection term.  Instead,  we discretize the convection term explicitly without the usage of upwind numerical flux. This approach preserves the optimal convergence rate for piecewise constant approximations; see \Cref{error_analysis} for a theoretical investigation and \Cref{table2} for numerical evidence. 
		Second,  the analysis of the convective HDG method necessitates the introduction of new elliptic HDG projections,  since those designed in \cite{CHSZ2023} would lead to a loss of optimal error estimates in the $L^2$ norm.   When there is no convection,  the HDG elliptic approximation $ u_{Ih} $ and $ \phi_{I h} $  can be obtained from solving two decoupled Laplacian equations; while the new  HDG elliptic approximation   needs to  be obtained from a coupled system,  including a convection-diffusion equation and one Laplacian type. Based on the new HDG elliptic projection,   we establish optimal $L^2$ error convergence rates for both scalar and flux variables for polynomials of order $k \geq 0$,  without constraints on spatio-temporal grid sizes.
		Finally,  the convection in the Cahn-Hilliard equation complicates the analysis of stability and the design of efficient solvers. Indeed,  Kay et al. \cite{KAYDAVID2009DGFE} introduced a DG scheme with implicit discretization of the convection term,  which is only  conditionally stable requiring a time step constraint $\Delta t \le C_* \epsilon^2$ \cite[Lemma 3.1]{KAYDAVID2009DGFE}. In contrast,  the explicit treatment of the convection in our HDG method
		guarantees unconditional (finite-time) energy stability,  existence and uniqueness of solution.
		Moreover,  due to the explicit discretization of the convection term,  the global system remains symmetric after element-by-element local elimination. This symmetry allows for the use of efficient iterative solvers,  such as the minimal residual method (MINRES),  to solve the global system.  Since numerical traces serve as the globally coupled degrees of freedom,  superconvergence of order $ k+2 $ is achieved for scalar unknowns. The methods and analytical tools developed in this work lay the foundation for constructing superconvergent HDG methods for phase-field fluid models,  including the Cahn-Hilliard-Navier-Stokes equations.

	Throughout this paper,  we use $ C $ to denote a positive constant independent of mesh sizes $ h $ and $ \Delta t $,  but dependent on $ Pe $,   $ \epsilon $ and $T$.
	The rest of this paper is structured as follows. In \Cref{FDCH},  we present the HDG formulation for the convective Cahn-Hilliard equation. Preliminary tools essential for numerical analysis are introduced in \Cref{HDG-Sobo},  and a special HDG elliptic projection is devised in  \Cref{section 2.1}. The projection error approximation in $L^2$ norm and optimal convergence results are established,  with detailed proofs provided in \Cref{Appendix}. Stability estimates for the nonlinear HDG methods are analyzed in \Cref{E_U_S},  followed by the derivation of optimal convergence rates in \Cref{error_analysis}. Finally,  numerical experiments validating our theoretical results are presented in \Cref{numerics}.

	\section{The HDG scheme}\label{FDCH}
	
	Let $\mathcal{T}_h$ be a shape-regular,  quasi-uniform triangulation of $\Omega$,  though other regular polygonal meshes are also applicable. Denote by $\mathcal{E}_h$ the set of all faces $E$ of the simplices $K$ in $\mathcal{T}_h$,  and let $\mathcal{E}_h^o$ and $\mathcal{E}_h^\partial$ represent the sets of interior and boundary faces,  respectively.
	
	We define the discrete inner products
	\begin{align*}
		(w, v)_{\mathcal{T}_h} := \sum_{K\in\mathcal{T}_h} (w, v)_K  = \sum_{K\in\mathcal{T}_h}\int_K w v\ {\rm d}\bm x,  \quad 	
		\left\langle \zeta, \rho\right\rangle_{\partial\mathcal{T}_h} := \sum_{K\in\mathcal{T}_h} \left\langle \zeta, \rho\right\rangle_{\partial K} = \sum_{K\in\mathcal{T}_h}\int_{\partial K} \zeta\rho\ {\rm d}\bm x, 
	\end{align*}
	 and introduce the corresponding discrete norm:
		$	\|w\|_{\mathcal{T}_h}^2=(w, w)_{\mathcal{T}_h},  \ \|\zeta\|_{\partial\mathcal{T}_h}^2=\langle\zeta, \zeta\rangle_{\partial \mathcal{T}_h}.$
	
	For any integer $k \geq 0$,  let $\mathcal{P}^k(K)$ denote the space of polynomials of degree at most $k$ on element $K$. We introduce the following discontinuous finite element spaces:
	\begin{align*}
		\bm V_h&:=\{\bm v_h\in [L^2(\Omega)]^d:\bm v_h|_K\in [\mathcal{P}^{k}(K)]^{d}, \forall K\in\mathcal{T}_h\}, \\
		W_h&:=\{w_h\in L^2(\Omega):w_h|_K\in \mathcal{P}^{k+1}(K), \forall K\in\mathcal{T}_h\}, \\
		\mathring{W}_h&:=\{w_h\in L^2_0(\Omega):w_h|_K\in \mathcal{P}^{k+1}(K), \forall K\in\mathcal{T}_h\}, \\
		M_h&:=\{\mu_h\in L^2(\mathcal{E}_h):\mu_h|_E\in \mathcal{P}^{k}(E), \forall E\in\mathcal{E}_h\}, 
	\end{align*}
	where $L^2_0(\Omega)$ denotes the subspace of $L^2(\Omega)$ consisting of mean-zero functions.
	
	Introducing auxiliary variables $\bm{p}=-\nabla \phi, \quad \bm{q}=-\nabla u, $
	the original PDE \eqref{ori} can be reformulated as the first-order system:
	\begin{eqnarray}\label{mixed}
		\bm p+ \nabla \phi = \bm 0,  \; u_t+\nabla\cdot (\frac{1}{Pe}\bm p+\bm {\beta}u) = 0, \;
		\bm q+ \nabla u = {\bm 0},    \;	\epsilon^2 \nabla\cdot\bm q+f(u)= \phi .
	\end{eqnarray}
	We now present the fully discrete HDG formulation based on the first-order convex-splitting approach. 	
	For a fixed integer $N$,  let $0=t_0<t_1<\cdots<t_N=T$ be a uniform partition of $[0,  T]$ with time step size $\Delta t = T/N$. The HDG method seeks $(\bm p_h^n, \phi_h^n, \widehat{\phi}_h^n)$ and $(\bm q_h^n,  u_h^n, \widehat u_h^n)$ in $\bm V_h\times W_h\times M_h$,  for all $(\bm r_1, w_1, \mu_1),  (\bm r_2, w_2, \mu_2) \in \bm V_h\times W_h\times M_h$ and $w_3\in W_h$,   satisfying:
	\begin{subequations}\label{HDG formula}
		\begin{align}
			&(\partial_{t}^+ u_h^n, w_1)_{\mathcal{T}_h}+\frac{1}{Pe}\mathcal A(\bm p_h^n,  \phi_h^n,   \widehat{\phi}_h^n;\bm r_1, w_1, \mu_1)+\mathcal{B}(u_h^{n-1},  \widehat{u}_h^{n-1}; w_1)=0, \label{hdg01}\\
			&(F(u^n_h,  u_h^{n-1}), w_2)_{\mathcal{T}_h}+\epsilon^2\mathcal{A}(\bm q_h^n,  u_h^n,  \widehat{u}_h^n;\bm r_2,  w_2, \mu_2)-(\phi_h^n, w_2)_{\mathcal{T}_h}=0\label{hdg02}, \\
			&(u_h^0, w_3)_{\mathcal{T}_h}-(u^0, w_3)_{\mathcal{T}_h}=0 \label{hdg03}, 
		\end{align}
		where   $\partial_{t}^+u_h^n=(u_h^n-u_h^{n-1})/\Delta t$,  
			$F(u^n_h,  u_h^{n-1})=(u_h^n)^3-u_h^{n-1}$
		. The bilinear forms are defined as follows.	
		The operator $\mathcal A: [\bm V_h\times W_h\times M_h]^2 \to \mathbb R$ is given by:
		\begin{align}\label{def_A}
			&\mathcal A(\bm q_h, u_h, \widehat u_h;\bm r_h, w_h, \mu_h)
			=(\bm q_h, \bm r_h)_{\mathcal{T}_h} - (u_h, \nabla\cdot\bm r_h)_{\mathcal{T}_h}
			+\langle \widehat u_h, \bm r_h\cdot\bm n \rangle_{\partial\mathcal{T}_h}  \\
			&\qquad+(\nabla\cdot\bm q_h,  w_h)_{\mathcal{T}_h} -\langle \bm q_h\cdot\bm n,  \mu_h \rangle_{\partial\mathcal{T}_h} +\alpha\langle  h_{\tau}^{-1}({\Pi}_k^{\partial}u_h-\widehat u_h),  {\Pi}_k^{\partial}w_h-\mu_h \rangle_{\partial\mathcal{T}_h}. \nonumber 
		\end{align}
		Here,  $\alpha$ is a positive constant,  which will be specified in the numerical experiments in   \Cref{numerics},  and  $\Pi_k^\partial $ is the element-wise $L^2$ projection onto $\mathcal{P}^k (E)$ such that 
		\begin{align*}
			\langle\Pi_k^\partial u_h,  \mu_h\rangle_{E}=\langle u_h,  \mu_h\rangle_{E},  \quad \forall \mu_h\in \mathcal{P}^k (E) \ \textup{and} \ E\subset \partial K, \ \forall K\in \mathcal{T}_h.
		\end{align*}
		Additionally,  for any $E \in \mathcal{E}_h$,  we have $\left.h_\tau\right|_E=h_E$.
		The operator $\mathcal B: {[ W_h\times M_h]\times W_h} \to \mathbb R$ is given by: 
		\begin{align}\label{def_B}
			&\mathcal B(u_h, \widehat u_h; w_h)
			= -(\bm \beta u_h, \nabla w_h)_{\mathcal{T}_h}
			+\langle  \bm{\beta}\cdot \bm{n}\widehat{u}_h, w_h\rangle_{\partial \mathcal{T}_h}.
		\end{align}
	\end{subequations}	
	
	\begin{remark}\label{firstremarkforB}
		The upwind numerical flux $\langle \tau_c (u_h-\widehat u_h),  w_h- \mu_h\rangle_{\partial\mathcal{T}_h}$ is widely used for PDEs with  convection terms, 
		where $ \tau_c $ is a stabilization piecewise constant with $ \tau_c\mid_E>0 $. In this formulation,   the operator $\mathcal{B}$ would instead take the form:
		\begin{align}\label{def_B_1}
			\mathcal{B}(u_h, \widehat u_h; w_h, \mu_h)
			= -(\bm \beta u_h, \nabla w_h)_{\mathcal{T}_h}
			+\langle  \bm{\beta}\cdot \bm{n}\widehat{u}_h, w_h\rangle_{\partial \mathcal{T}_h}+\langle \tau_c (u_h-\widehat u_h),  w_h- \mu_h\rangle_{\partial\mathcal{T}_h}.
		\end{align}
		In this case,  the optimal error convergence order no longer holds for $k=0$,  although it remains valid for $k\ge 1$. For theoretical analysis,  see \Cref{projection--proof--remark},  and for numerical results,  refer to \Cref{table2} in \Cref{exmaple1}.
		The loss of the optimal convergence rate for $k=0$ is not unique to the convective Cahn-Hilliard equation but also occurs in HDG-type methods for other PDEs. Similar behavior has been observed in \cite{MR3440284} for convection-diffusion PDEs,  \cite{MR3833696} for an optimal control problem governed by convection-diffusion PDEs,  and \cite{MR3556409} for the Navier-Stokes equations.
	\end{remark}

	\subsection{Preliminaries} \label{HDG-Sobo} 
	We recall the standard $L^2$ projection operators $\bm\Pi_{k}^o: [L^2(\Omega)]^d \to \bm V_h$ and $\Pi_{k+1}^o: L^2(\Omega) \to W_h$. The following approximation results are well established; see,  for instance,  {\cite[Lemma 3.3]{ChenGang2019AHMf},  \cite[Lemma 4.5.3]{BrennerSusanne2007TMTo}}:
	\begin{subequations}\label{classical_ine}
		\begin{align}
			&\|{\bm q -\bm\Pi_k^o \bm q}\|_{L^2(K)} \le  C h_K^{k+1} |{\bm q}|_{H^{k+1}(K)},   \quad 
			\| {u -{\Pi_{k+1}^o u}}\|_{L^2(K)} \le  C h_K^{k+2} |{u}|_{H^{k+2}(K)}, \label{classical_ine_1}\\
			&	\| {u - {\Pi_{k+1}^o u}}\|_{L^2(\partial K)} \le  C h_K^{k+ 3/ 2} |{u}|_{H^{k+2}(K)},   \label{classical_ine_2}  \\
			&	  \| w_h\|_{L^2 ( \partial K)} \le  C h_K^{-1/2} \| w_h\|_{L^2(K)},  \quad\forall w_h\in W_h,  \label{classical_ine_5}  \\
			&\| {u - {\Pi_{k+1}^o u}}\|_{L^q ( K)} 	\le  C h_K^{k+2+d(1/q-1/2)} | u |_{H^{k+2}(K)},   \quad q\in [1,  +\infty],  \label{classical_ine_4}\\
			&	\| w_h \|_{L^q ( K)} \le  C h_K^{d(1/q -1/2)} \| w_h \|_{L^2 ( K)},   \quad \forall w_h\in W_h,  \quad q\in [1,  +\infty]. \label{classical_ine_6}
		\end{align}
	\end{subequations}
	
	The following HDG Sobolev inequalities follow directly from \cite[Theorem 6.5]{DiPietroDanieleAntonio2020THHM}:
	\begin{lemma}[HDG Sobolev inequality] \label{discrete-soblev-hdg} 
		Suppose $q\in [1,  \infty)$ for $d=2$,  and $q\in [1,  6]$ for $d=3$. For any $\mu_h \in M_h$,  the following estimates hold:
		\begin{subequations}\label{sobolev}
			\begin{align}
				&\|{w}_{h}\|_{L^q(\Omega)}\le C \left( \|w_h\|_{\mathcal{T}_h}+\|\nabla w_h\|_{\mathcal{T}_h}
				+\alpha^{1/2}\|h_{\tau}^{-1/2}(\Pi_k^{\partial}w_h-\mu_h)\|_{\partial\mathcal{T}_h}\right),  \forall w_h\in W_h,  
				\label{sobolev-01} \\
				&\|w_{h}\|_{L^q(\Omega)}\le C \left(\|\nabla w_h\|_{\mathcal{T}_h}+\alpha^{1/2}\|h_{\tau}^{-1/2}(\Pi_k^{\partial}w_h-\mu_h)\|_{\partial\mathcal{T}_h}\right),  \forall w_h\in \mathring{W}_h. \label{sobolev-02}
			\end{align}
		\end{subequations}
	\end{lemma}
	Next,  we present some basic properties of the operator $ \mathcal{A} $.
	\begin{lemma}{\cite[Lemmas~3.3--3.6]{ChenGang2018Otso}}\label{A-property}
		For all $(\boldsymbol{q}_h,  u_h,  \widehat{u}_h), (\boldsymbol{p}_h,  \phi_h,  \widehat{\phi}_h) \in \boldsymbol{V}_h \times W_h \times M_h$,  the following identities and bounds hold:
		\begin{align}\label{nonsymmetric-A}
			&\mathcal{A}(\boldsymbol{q}_h,  u_h,  \widehat{u}_h ; \boldsymbol{p}_h, -\phi_h, -\widehat{\phi}_h)=\mathcal{A}(\boldsymbol{p}_h,  \phi_h,  \widehat{\phi}_h ; \boldsymbol{q}_h, -u_h, -\widehat{u}_h ), \\
			\label{bounded-A}
			&\mathcal{A}(\boldsymbol{q}_h,  u_h,  \widehat{u}_h ; \boldsymbol{q}_h,  u_h,  \widehat{u}_h)=\|\boldsymbol{q}_h\|_{\mathcal{T}_h}^2+\alpha\|h_{\tau}^{-1 / 2}(\Pi_k^{\partial} u_h-\widehat{u}_h)\|_{\partial \mathcal{T}_h}^2, \\
			&|\mathcal{A}(\bm{q}_h, u_h, \widehat{u}_h;\bm{p}_h, \phi_h, \widehat{\phi}_h)|\le C (\alpha^{1/2}\|h_{\tau}^{-1/2}(\Pi_k^{\partial}u_h-\widehat{u}_h)\|_{\partial\mathcal{T}_h}+\|\bm{q}_h\|_{\mathcal{T}_h}\nonumber\\
			&+\|\nabla u_h\|_{\mathcal{T}_h})\times (\|\bm{p}_h\|_{\mathcal{T}_h}+\|\nabla \phi_h\|_{\mathcal{T}_h}+\alpha^{1/2}\|h_{\tau}^{-1/2}(\Pi_k^{\partial}\phi_h-\widehat{\phi}_h)\|_{\partial\mathcal{T}_h}).\label{A-continuous}
		\end{align}
		In addition,  if $\mathcal{A}(\boldsymbol{q}_h,  u_h,  \widehat{u}_h ; \boldsymbol{r}_h,  0, 0)=0$ for all $\boldsymbol{r}_h \in \boldsymbol{V}_h$,  then the following inequality holds:
		\begin{align}\label{es_u}
			\|\nabla u_h\|_{\mathcal{T}_h}+\|h_{\tau}^{-1 / 2}(u_h-\widehat{u}_h)\|_{\partial \mathcal{T}_h} \leq C(\|\boldsymbol{q}_h\|_{\mathcal{T}_h}+\alpha^{1/2}\|h_{\tau}^{-1 / 2}(\Pi_k^{\partial} u_h-\widehat{u}_h)\|_{\partial \mathcal{T}_h}).
		\end{align}
	\end{lemma}
	Moreover,  the operator $\mathcal{A}$ satisfies the following discrete LBB condition.
		\begin{lemma}\cite[Lemma 3.3]{CHSZ2023}
			\label{A-LBB}
			For all $\left(\boldsymbol{q}_h,  u_h,  \widehat{u}_h\right) \in \boldsymbol{V}_h \times \stackrel{\circ}{W}_h \times M_h$,  we have
			\begin{align*}
				\sup _{0 \neq\left(\boldsymbol{p}_h,  \phi_h,  \widehat{\phi}_h\right) \in \boldsymbol{V}_h \times \mathring{W}_h \times M_h} \frac{\mathcal{A}\left(\boldsymbol{q}_h,  u_h,  \widehat{u}_h ; \boldsymbol{p}_h,  \phi_h,  \widehat{\phi}_h\right)}{\left\|\boldsymbol{p}_h\right\|_{\mathcal{T}_h}+\left\|\nabla \phi_h\right\|_{\mathcal{T}_h}+\left\|h_{\mathcal{T}}^{-1 / 2}\left(\Pi_k^{\partial} \phi_h-\widehat{\phi}_h\right)\right\|_{\partial \mathcal{T}_h}} \\
				\geq C\left(\left\|\boldsymbol{q}_h\right\|_{\mathcal{T}_h}+\left\|\nabla u_h\right\|_{\mathcal{T}_h}+\left\|h_{\mathcal{T}}^{-1 / 2}\left(\Pi_k^{\partial} u_h-\widehat{u}_h\right)\right\|_{\partial \mathcal{T}_h}\right).
			\end{align*}
		\end{lemma}
	
	To simplify notation,  we introduce the following norm:
	\begin{align*}
		\interleave (\bm{p}_h, \phi_h, \widehat{\phi}_h)\interleave^2 :=\|\bm{p}_h\|_{\mathcal{T}_h}^2
		+\alpha\|h_{\tau}^{-1/2}(\Pi_k^{\partial}\phi_h-\widehat{\phi}_h)\|_{\partial\mathcal{T}_h}^2.
	\end{align*}
	Based on the definition \eqref{def_B},  we derive the following property of the operator $ \mathcal{B} $.
	\begin{lemma}
		For all $ (u_h, \widehat{u}_h), (w_h, \mu_h)\in W_h\times M_h $,   the following bound holds:
		\begin{align}\label{B-bound-ori}
				|\mathcal B(u_h, \widehat u_h; w_h)| \leq &C   ( \|u_h\|_{\mathcal{T}_h}
				+h\|h_{\tau}^{-1/2}(\Pi_k^{\partial}u_h-\widehat u_h)\|_{\partial\mathcal{T}_h}) (\|\nabla w_h\|_{\mathcal{T}_h}+\|h_{\tau}^{-1/2}(w_h-\mu_h)\|_{\partial\mathcal{T}_h}).
			\end{align}
			Specifically,  if $	\mathcal A(\bm p_h, w_h, \mu_h; \bm r_h, 0, 0)=0 $ for all $\bm r_h\in \bm V_h$,  then it holds
		
		\begin{align}
			\label{B-es}
			|\mathcal B(u_h, \widehat u_h; w_h)| \leq &C   ( \|u_h\|_{\mathcal{T}_h}
			+h\|h_{\tau}^{-1/2}(\Pi_k^{\partial}u_h-\widehat u_h)\|_{\partial\mathcal{T}_h}) \interleave(\bm p_h,  w_h,  \mu_h) \interleave.
		\end{align}  
	\end{lemma}
	\begin{proof}
		By definition of operator $ \mathcal{B} $,  one has
		\begin{align*}
			\mathcal{B}(u_h, \widehat{u}_h;w_h)&=(-\bm{\beta}u_h, \nabla w_h)_{\mathcal{T}_h}
			+\langle \bm{\beta}\cdot \bm{n}\widehat{u}_h, w_h \rangle_{\partial \mathcal{T}_h}=:R_1+R_2.
		\end{align*}
		For $R_1$,  the Cauchy--Schwarz inequality yields
		\begin{align*}
			|R_1|&\le \|\bm{\beta}\|_{0, \infty}\|u_h\|_{\mathcal{T}_h}\|\nabla w_h\|_{\mathcal{T}_h}\le  C\|u_h\|_{\mathcal{T}_h}
			\|\nabla w_h\|_{\mathcal{T}_h}.
		\end{align*}
		To estimate $R_2$,  note that for any $\mu_h\in M_h$, 
			\[
			\langle \bm{\beta}\cdot\bm{n}\, \widehat{u}_h, \mu_h\rangle_{\partial\mathcal{T}_h}=0, 
			\]
			because $\widehat u_h$ and $\mu_h$ are single-valued on each interior face (so the contributions cancel when faces are
			counted twice with opposite normals),  and $\bm{\beta}\cdot\bm{n}=0$ on $\partial\Omega$.
			Hence, 
		\begin{align*}
			|R_2|&=
				|\langle \bm{\beta}\cdot \bm{n}\widehat{u}_h, w_h\rangle_{\partial \mathcal{T}_h}|
			=
			|\langle \bm{\beta}\cdot \bm{n}\widehat{u}_h, w_h-\mu_h\rangle_{\partial \mathcal{T}_h}|\le C\|\widehat{u}_h\|_{\partial\mathcal{T}_h}\|w_h-\mu_h\|_{\partial\mathcal{T}_h}\\
			&\le C(\|\Pi_k^{\partial}u_h-\widehat{u}_h\|_{\partial\mathcal{T}_h}+\|\Pi_k^{\partial}u_h\|_{\partial\mathcal{T}_h})\|w_h-\mu_h\|_{\partial\mathcal{T}_h}\\
			&\le Ch\|h_{\tau}^{-1/2}(\Pi_k^{\partial}u_h-\widehat{u}_h)\|_{\partial\mathcal{T}_h}\|h_{\tau}^{-1/2}(w_h-\mu_h)\|_{\partial\mathcal{T}_h}
			+C\|u_h\|_{\mathcal{T}_h}\|h_{\tau}^{-1/2}(w_h-\mu_h)\|_{\partial\mathcal{T}_h}.
		\end{align*}
		{Combining the bounds for $R_1$ and $R_2$ proves \eqref{B-bound-ori}.
			
			Finally,  if $\mathcal A(\bm p_h, w_h, \mu_h;\bm r_h, 0, 0)=0$ for all $\bm r_h\in \bm V_h$,  then by \eqref{es_u} we have
			\[
			\|\nabla w_h\|_{\mathcal{T}_h}+\|h_{\tau}^{-1/2}(w_h-\mu_h)\|_{\partial\mathcal{T}_h}
			\le C\, \interleave(\bm p_h, w_h, \mu_h)\interleave.
			\]
			Substituting this into \eqref{B-bound-ori} yields \eqref{B-es}.}
	\end{proof}

	For all $u_h\in\mathring W_h$,   the HDG Laplace inverse of $u_h$ is defined as the unique solution $(\bm{\Pi}_{\bm V}u_h, \Pi_Wu_h, \Pi_{M}u_h)\in \bm V_h\times\mathring{W}_h\times M_h$ satisfying
	\begin{align}\label{def_-1_h}
		\mathcal A(\bm{\Pi}_{\bm V}u_h, \Pi_W u_h, \Pi_{M}u_h;\bm r_h, w_h, \mu_h)=(u_h, w_h)_{\mathcal{T}_h}
	\end{align}
	for all $(\bm r_h, w_h, \mu_h)\in \bm V_h\times{W}_h\times M_h$. Additionally,  we introduce the broken discrete $H^{-1}$,  following  \cite{CHSZ2023},  defined as
	\begin{align*}
		\|u_h\|^2_{-1, h}:=
		\mathcal A(\bm{\Pi}_{\bm V}u_h, \Pi_Wu_h, \Pi_{M}u_h;\bm{\Pi}_{\bm V}u_h, \Pi_Wu_h, \Pi_{M}u_h)=(u_h, \Pi_Wu_h)_{\mathcal{T}_h}.
	\end{align*}

	\begin{lemma}[{\cite[Lemma 3.6,  inequalities (3.12) and (3.13)]{CHSZ2023}}]
		If $u_h\in \mathring{W}_h$ and $(w_h, \mu_h)\in  W_h\times M_h$,   then we have
		\begin{align}
			&(u_h, w_h)_{\mathcal{T}_h}\le C\|u_h\|_{-1, h}(\|\nabla w_h\|_{\mathcal{T}_h}
			+\alpha^{1/2}\|h_{\tau}^{-1/2}(\Pi_k^{\partial}w_h-\mu_h)\|_{\partial\mathcal{T}_h}
			), \label{Cauchy}\\
			&\|\nabla \Pi_W u_h \|_{\mathcal{T}_h}+\alpha^{1/2}\|h_{\tau}(\Pi_k^{\partial} \Pi_W u_h-\Pi_M u_h)\|_{\partial\mathcal{T}_h}  \le C\|u_h\|_{-1, h},   \label{lem-nega} \\
			&\|u_h\|_{-1, h} \leq C\| u_h \|_{\mathcal{T}_h}. \label{nega-imb}
		\end{align}

	\end{lemma}
	\subsection{The HDG elliptic projection}\label{section 2.1}
	Given $ f_1, f_2\in L^2(\Omega) $,  let $ (\psi, \bm{\Psi};\theta, \bm{\Theta}) $ denote the solution of the following equations:
	\begin{subequations}\label{dual equation}
		\begin{align}
			&\bm{\Theta}+\nabla\theta=\bm{0}, \ \frac{1}{Pe}\nabla\cdot\bm{\Theta}+\nabla\cdot(\bm{\beta}\psi)=f_1, \ \textup{in} \ \Omega, \ 
			\bm{\Theta}\cdot\bm{n}=\bm{0}\ \textup{on} \ \partial\Omega, \ \int_{\Omega}\theta\ {\rm d}\bm{x}=0, \label{eq-phi-u-convective-possion}\\
			&\bm{\Psi}+\nabla\psi=\bm{0}, \ \nabla\cdot\bm{\Psi}=f_2, \ \textup{in}\ \Omega, \ \bm{\Psi}\cdot\bm{n}=\bm{0}, \ \textup{on} \ \partial\Omega,  \ \int_{\Omega}\psi \ {\rm d}\bm{x}=0.\label{eq-u-possion}
		\end{align}
	\end{subequations}
	If $ \Omega $ is convex,  we can obtain the following regularity results:
	\begin{subequations}
		\begin{align}
			\|\bm{\Psi}\|_{H^1(\Omega)}+\|\psi\|_{H^2(\Omega)}&\le C\|f_2\|_{L^2(\Omega)}, 
			\label{regularity--condition-u}\\
			\|\bm{\Theta}\|_{H^1(\Omega)}+\|\theta\|_{H^2(\Omega)}&\le C(\|f_1\|_{L^2(\Omega)}+\|\bm{\beta}\|_{0, \infty}\|\psi\|_{H^1(\Omega)})\nonumber \\
			&\le C (\|f_1\|_{L^2(\Omega)}+\|f_2\|_{L^2(\Omega)}).	\label{regularity--condition-phi-u}
		\end{align}
	\end{subequations}
	\begin{remark}Regularity result \eqref{regularity--condition-u} follows directly from the standard analysis of the Poisson equation. For \eqref{regularity--condition-phi-u},  given the solution $\psi$  of \eqref{eq-u-possion},  we have
		\begin{align*}
			\bm{\Theta}+\nabla\theta=\bm{0}, \ \frac{1}{Pe}\nabla\cdot\bm{\Theta}=f_1-\nabla\cdot(\bm{\beta}\psi), 
		\end{align*}
		Applying the same regularity analysis,  we obtain the estimate in \eqref{regularity--condition-phi-u}.
	\end{remark}
	
	Recall that  $(\bm p,  \phi,  \bm q,  u)$  is  the solution of the Cahn-Hilliard equation in mixed form \eqref{mixed}.
	For all $t\in [0, T]$,  we define the HDG elliptic projections as follows:
	Find $(\bm p_{Ih}, \phi_{Ih}, \widehat{\phi}_{Ih})$,  $  (\bm q_{Ih},  u_{Ih}, \widehat u_{Ih})\in \bm V_h\times W_h\times M_h$ such that
	\begin{subequations}\label{HDG-projection}
		\begin{align}
			&\mathcal{B}(u_{Ih}, \widehat{u}_{Ih};w_1)+\frac{1}{Pe}\mathcal{A}(\bm{p}_{Ih}, \phi_{Ih}, \widehat{\phi}_{Ih};\bm{r}_1, w_1, \mu_1)=(\nabla\cdot(\bm{\beta}u)-\frac{1}{Pe}\Delta \phi, w_1)_{\mathcal{T}_h}, \label{HDG-projection--phi--u}\\
			&\mathcal{A}(\bm{q}_{Ih}, u_{Ih}, \widehat{u}_{Ih};\bm{r}_2, w_2, \mu_2)=(-\Delta u, w_2)_{\mathcal{T}_h}\label{HDG-projection--u}, \
			(\phi_{Ih}-\phi, 1)_{\mathcal{T}_h}=(u_{Ih}-u, 1)_{\mathcal{T}_h}=0.
		\end{align}
	\end{subequations}
	\begin{remark}
			The solution of \eqref{HDG-projection--phi--u} and \eqref{HDG-projection--u} exists and is unique. The existence and uniqueness of equation \eqref{HDG-projection--u} is implied by \Cref{A-LBB}. Once we obtain $ (\bm{q}_{Ih}, u_{Ih}, \widehat{u}_{Ih}) $,  substituting it into \eqref{HDG-projection--phi--u} and using the discrete LBB condition in \Cref{A-LBB},  the existence and uniqueness of the solution to \eqref{HDG-projection--phi--u} follow immediately.
	\end{remark}
	Denote the norm  and seminorm on the Hilbert space,  $ H^s(\Omega) $,  by $ \|\cdot\|_s $ and $ |\cdot|_s $. We have the following approximation property for the HDG elliptic projection \eqref{HDG-projection}.
	\begin{lemma}\label{HDG-elliptic-projection-convergence results}
		Assuming that the regularity conditions \eqref{regularity--condition-u} and \eqref{regularity--condition-phi-u} are satisfied,  suppose the solution $(\boldsymbol{p},  \phi,  \boldsymbol{q},  u)$ possesses the maximum regularity required to achieve the best approximation results in \eqref{classical_ine}. Let $(\boldsymbol{p}_{I h},  \phi_{I h},  \widehat{\phi}_{I h})$ and $\left(\boldsymbol{q}_{I h},  u_{I h},  \widehat{u}_{I h}\right)$ be the solutions of  \eqref{HDG-projection}. Then,  the following error estimates hold:
		\begin{subequations}\label{approximation-HDG-projection}
			\begin{align}
				\|u-u_{Ih}\|_{\mathcal{T}_h}&\le Ch^{k+2}|u|_{k+2}, \label{error-u-projection}\\
				\|\bm{q}-\bm{q}_{Ih}\|_{\mathcal{T}_h}+\alpha^{1/2}\|h_{\tau}^{-1/2}(\Pi_k^{\partial}u_{Ih}-\widehat{u}_{Ih})\|_{\partial\mathcal{T}_h}&\le Ch^{k+1}|u|_{k+2}, \label{error-q-projection}\\
				\|\partial_tu-\partial_tu_{Ih}\|_{\mathcal{T}_h}&\le Ch^{k+2}|\partial_tu|_{k+2}, \label{error-ut-projection}\\
				\|\bm{p}-\bm{p}_{Ih}\|_{\mathcal{T}_h}+\alpha^{1/2}\|h_{\tau}^{-1/2}(\Pi_k^{\partial}\phi_{Ih}-\widehat{\phi}_{Ih})\|_{\partial\mathcal{T}_h}&\le Ch^{k+1}(|u|_{k+1}+|u|_{k+2}+|\phi|_{k+2}),  \label{error-p-projection}\\
				\|\phi-\phi_{Ih}\|_{\mathcal{T}_h}&\le Ch^{k+2}(|\phi|_{k+2}+|u|_{k+2}+|u|_{k+1}).\label{error-phi-projection}
			\end{align}
		\end{subequations}
	\end{lemma}
	
	The proof of \eqref{error-u-projection}–\eqref{error-ut-projection} can be found in {\cite[Lemma 3.7]{CHSZ2023}}. The proofs of \eqref{error-p-projection}  and \eqref{error-phi-projection} follow from the third Strang lemma and the Aubin–Nitsche technique,  as presented in \cite{DiPietroDanieleA.2018AtSl}. For completeness,  the detailed derivations are provided in \Cref{Appendix}.
	\begin{remark}\label{approximationprojectionremark}
		As discussed in \Cref{projection--proof--remark},  we observe that if the operator $\mathcal{B}$ is used in the form of \eqref{def_B_1},  then the conclusions of \eqref{error-u-projection}–\eqref{error-p-projection} remain valid. However,  the convergence rate of $\|\phi - \phi_{I h}\|_{\mathcal{T}_h}$ is suboptimal when $k = 0$,  but optimal for $k \geq 1$; see \Cref{projection--proof--remark} for more details. More specifically,  we have the following estimate for $\|\phi - \phi_{Ih}\|_{\mathcal{T}_h}$: 
		\begin{align*}
			\|\phi-\phi_{Ih}\|_{\mathcal{T}_h} \le Ch^{k+1+\min\{k, 1\}}(|\phi|_{k+2}+|u|_{k+2}+|u|_{k+1}).
		\end{align*}
	\end{remark}

	We also recall the definition of the discrete Laplacian from \cite{CHSZ2023}. For any $u_h\in  W_h$,  we define $\Delta_h u_h\in  W_h$ by
	\begin{align}\label{Delta_h}
		(\Delta_h u_h,  w_h)_{\mathcal{T}_h}=-\mathcal{A}(\bm q_h^u, u_h, \widehat u_h^u;\bm r_h, w_h, \mu_h), \quad (\bm r_h, w_h, \mu_h)\in \bm V_h\times W_h\times M_h.
	\end{align}
	Here,     $(\bm q_h^u, \widehat u_h^u)\in \bm V_h\times M_h$ satisfies 
	\begin{align*}
		\mathcal{A}(\bm q_h^u, u_h, \widehat u_h^u;\bm r_h, 0, \mu_h)=0,  \quad \forall (\bm r_h, \mu_h)\in \bm V_h\times M_h.
	\end{align*}
	Then the following stability result holds.
	\begin{lemma}[{\cite[Lemma 3.8]{CHSZ2023}\label{lap_u_h2}\label{Lemma 2.4}}]
		For all $w_h\in  W_h$,  we have 
		\begin{align}\label{HDG-uniform}
			\|w_h\|_{L^\infty(\Omega)}\le C\|\Delta_hw_h\|_{\mathcal T_h}, 
		\end{align}
		where $C$ depends on $\Omega$.
	\end{lemma}

	\section{The stability estimates}\label{E_U_S}
	{ Before presenting the stability estimates,  we first establish the well-posedness of the HDG scheme \eqref{HDG formula}. 
		\begin{theorem}[Well-posedness]\label{thm:wellposedness} For each $n\ge 1$,  given $(u_h^{n-1}, \widehat u_h^{\, n-1})$,  the nonlinear HDG system \eqref{HDG formula}
			admits a unique solution $(\bm p_h^{n}, \phi_h^{n}, \widehat\phi_h^{\, n}, \bm q_h^{n}, u_h^{n}, \widehat u_h^{\, n})$.
		\end{theorem}
		
		\begin{proof}
			The existence and uniqueness follow from the same argument as in
			\cite[Theorems~4.1--4.2]{ChenGang2018Otso}. The only additional ingredient in the convective case is the estimate
			of the explicit convection contribution. Using  \eqref{B-es},  the Cauchy--Schwarz
			inequality,  Young's inequality,  \eqref{sobolev-02},  and \eqref{es_u},  we obtain
			\begin{align*}
				&-(u_h^{n-1}, \overline{\phi}_h^{\, n})_{\mathcal{T}_h}
				+\Delta t\, \mathcal{B}(u_h^{n-1}, \widehat{u}_h^{\, n-1};\overline{\phi}_h^{\, n})\\
				&\le
				\|u_h^{n-1}\|_{\mathcal{T}_h}\, \|\overline{\phi}_h^{\, n}\|_{\mathcal{T}_h}
				+C\Delta t\Bigl(\|u_h^{n-1}\|_{\mathcal{T}_h}
				+h\|h_{\tau}^{-1/2}(\Pi_{k}^{\partial}u_h^{n-1}-\widehat{u}_h^{\, n-1})\|_{\partial\mathcal{T}_h}\Bigr)
				\interleave(\overline{\bm{p}}_h^{\, n}, \overline{\phi}_h^{\, n}, \widehat{\overline{\phi}}_h^{\, n})\interleave\\
				&\le \frac{\Delta t}{2Pe}\interleave(\overline{\bm{p}}_h^{n}, \overline{\phi}_h^{n}, \widehat{\overline{\phi}}_h^{n})\interleave^2
				+\frac{C}{\Delta t}\|u_h^{n-1}\|_{\mathcal{T}_h}^2
				+C\Delta t(\|u_h^{n-1}\|_{\mathcal{T}_h}
				+h\|h_{\tau}^{-1/2}(\Pi_{k}^{\partial}u_h^{n-1}-\widehat{u}_h^{n-1})\|_{\partial\mathcal{T}_h})^2.
			\end{align*}
			With this estimate,  the remainder of the existence/uniqueness argument is identical to
			\cite{ChenGang2018Otso},  and we omit the details.
	\end{proof}}

	The fundamental energy stability bounds for the HDG scheme are established in the following lemma.
	\begin{lemma}\label{corEnergy}
		For any $h,  \Delta t>0$ and $m=1,  2,  \ldots,   N$,   the numerical solution of the HDG scheme satisfies the following stability bounds:
		\begin{subequations}
			\begin{align}
				&\frac{1}{4}\|( u_h^m)^2-1\|^2_{\mathcal{T}_h}+ \frac{\epsilon^2}{2}\interleave(\bm q_h^m,  u_h^m, \widehat{u}_h^m)\interleave^2 +\sum_{n=1}^m \frac{\Delta t}{2Pe}\interleave(\bm p_h^n,  \phi_h^n, \widehat{\phi}_h^n)\interleave^2 \label{corEnergy_eq} \\
				&\qquad+\frac{\epsilon^2}{2} \sum_{n=1}^m \Big(\|\Delta t \partial_t^+  \bm q_h^n\|^2_{\mathcal{T}_h}+ \alpha\|\Delta t h_{\tau}^{-1/2}\partial_{t}^+ (\Pi_k^{\partial}u_h^n-\widehat u_h^n)\|^2_{\partial\mathcal{T}_h}\Big) \leq C,  \nonumber \\
				& \frac{1}{Pe}\|\phi_h^{m}\|_{\mathcal T_h}^2+\frac{1}{Pe}\sum_{n=1}^{m}\|\phi_h^n-\phi_h^{n-1}\|^2_{\mathcal T_h}+\epsilon^2 \Delta t\sum_{n=1}^{m}\|\partial_{t}^+ u_h^n\|_{\mathcal T_h}^2 \leq C,   \label{phi_infty-1} \\
				& \|\Delta_h u_h^n\|_{\mathcal{T}_h}+ \| u_h^n\|_{L^\infty(\Omega)} \le C,  \quad n =1,  2\ldots N. \label{uh_infty_bounded}
			\end{align}
			{We note that the  generic constant $C$ is independent of $h$ and $\Delta t$, but may depend on $\epsilon$ and $Pe$.}  Moreover,  if $ \|u\|_{L^{\infty}(H^2(\Omega))}\le C $,  then the uniform bound of the HDG elliptic projection $ u_{Ih} $ holds:
			\begin{align}
				\|u_{Ih}^n\|_{L^{\infty}(\Omega)}\le C,   \quad n =1,  2\ldots N.\label{u_Ih_infty_bound}
			\end{align}
		\end{subequations}
		
	\end{lemma}
	\begin{remark}
		The inequalities \eqref{corEnergy_eq} and \eqref{phi_infty-1} confirm that our scheme is unconditionally stable,  as it imposes no constraints on the time step $\Delta t$.  In contrast,  Kay et al. \cite{KAYDAVID2009DGFE} proposed a DG scheme with implicit discretization for the convection term,  which guarantees only conditional stability. Their scheme requires the time step $\Delta t$ to satisfy $
		\Delta t \le C_* \epsilon^2$,  see \cite[Lemma 3.1]{KAYDAVID2009DGFE}. 		
	\end{remark}

	\begin{proof}
		The proof follows from a sequence of energy estimates. First,  {choose $(\bm r_2,  w_2, \mu_2)=(\bm 0, \partial_{t}^+u_h^n, \partial_{t}^+ \widehat u_h^n)$ in \eqref{hdg02} and use the definition of $ \mathcal{A} $ in \eqref{def_A} to obtain
			\begin{align}
				&(F(u_h^n, u_h^{n-1}), \partial_{t}^+u_h^n)_{\mathcal{T}_h}+\epsilon^2(\nabla\cdot\bm{q}_h^n, \partial_{t}^+u_h^n)_{\mathcal{T}_h}
				-\epsilon^2\langle\bm{q}_h^n\cdot\bm{n}, \partial_{t}^+\widehat{u}_h^n\rangle_{\partial \mathcal{T}_h}\nonumber\\
				&+\epsilon^2\alpha\langle h_{\tau}^{-1}(\Pi_k^{\partial}u_h^n-\widehat{u}_h^n), \Pi_k^{\partial}\partial_{t}^+u_h^n-\partial_{t}^+\widehat{u}_h^n\rangle_{\partial \mathcal{T}_h}
				-(\phi_h^n, \partial_t^+u_h^n)_{\mathcal{T}_h}=0.\label{temp-1-01}
			\end{align}
			Then apply $ \partial_{t}^+ $ to \eqref{hdg02} to obtain
			\begin{align}\label{temp-1-02}
				(\partial_{t}^+F(u_h^n, u_h^{n-1}), w_2)_{\mathcal{T}_h}+\epsilon^2\mathcal{A}(\partial_{t}^+\bm{q}_h^n, \partial_{t}^+u_h^n, \partial_{t}^+\widehat{u}_h^n;\bm{r}_2, w_2, \mu_2)-(\partial_{t}^+\phi_h^n, w_2)_{\mathcal{T}_h}=0.
			\end{align}	
			Selecting $ (\bm{r}_2, w_2, \mu_2)=(\bm{q}_h^n, 0, 0) $ in \eqref{temp-1-02} and use the definition of $\mathcal{A}$ in \eqref{def_A},  one derives
			\begin{align}\label{temp-1-03}
				\epsilon^2(\partial_{t}^+\bm{q}_h^n, \bm{q}_h^n)_{\mathcal{T}_h}-\epsilon^2(\partial_{t}^+u_h^n, \nabla\cdot\bm{q}_h^n)_{\mathcal{T}_h}
				+\epsilon^2\langle\partial_{t}^+\widehat{u}_h^n, \bm{q}_h^n\cdot\bm{n}\rangle_{\partial \mathcal{T}_h}=0.
			\end{align}
			Add the \eqref{temp-1-01} and \eqref{temp-1-03} together,  one can obtain:
		}
		\begin{align*}
			&(F(u_h^n,  u_h^{n-1}), \partial_t^+ u_h^n)_{\mathcal{T}_h}+\epsilon^2(\partial_{t}^+ \bm q_h^n, \bm q_h^n)_{\mathcal{T}_h} +\langle \alpha\epsilon^2 h_{\tau}^{-1}(\Pi_k^{\partial}u_h^n-\widehat u_h^n), \partial_{t}^+ (\Pi_k^{\partial}u_h^n-\widehat u_h^n) \rangle_{\partial\mathcal T_h} \\
			&\qquad-(\phi_h^n, \partial_{t}^+ u_h^n)_{\mathcal{T}_h}=0.
		\end{align*}	%
		
		Similarly,  selecting  $(\bm r_1,  w_1,  \mu_1)=(\bm p_h^n, \phi_h^n, \widehat\phi_h^n)$ in \eqref{hdg01} leads to
		\begin{align*}
			(\partial_{t}^+ u_h^n, \phi_h^n)_{\mathcal{T}_h}+ \frac{1}{Pe}\interleave(\bm p_h^n,  \phi_h^n, \widehat{\phi}_h^n)\interleave^2
			+\mathcal{B}(u_h^{n-1},  \widehat{u}_h^{n-1}; \phi_h^n)=0\nonumber.
		\end{align*}
		Hence 	
		\begin{align}
			& (F(u_h^n,  u_h^{n-1}), \partial_{t}^+u_h^n)_{\mathcal{T}_h}+\epsilon^2(\partial_{t}^+ \bm q_h^n, \bm q_h^n)_{\mathcal{T}_h}
			+\alpha\langle\epsilon^2h_{\tau}^{-1}(\Pi_k^{\partial}u_h^n-\widehat u_h^n), \partial_{t}^+ (\Pi_k^{\partial}u_h^n-\widehat u_h^n) \rangle_{\partial\mathcal T_h } \nonumber\\
			&\qquad+\frac{1}{Pe}\interleave(\bm p_h^n,  \phi_h^n, \widehat{\phi}_h^n)\interleave^2+\mathcal{B}(u_h^{n-1},  \widehat{u}_h^{n-1}; \phi_h^n)=0. \label{sta-01}
		\end{align}
		{Let $ \gamma^n=\frac{1}{|\Omega|}\int_{\Omega}\phi_h^n{\rm d}x $,  then
			\begin{align*}
				(\phi_h^n-\gamma^n, 1)_{\mathcal{T}_h}=0,  \ \forall n=1, 2, \cdots,  N.
			\end{align*}
			Taking $(w_1, \mu_1)=(0, 0)$ in \eqref{hdg01} yields
			\[
			\mathcal{A}(\bm p_h^n, \phi_h^n, \widehat{\phi}_h^n;\bm r_h, 0, 0)=0
			\qquad \forall\, \bm r_h\in \bm V_h.
			\]
			Since $\gamma^n$ is a constant,  it follows that
			\[
			\mathcal{A}(\bm p_h^n, \phi_h^n-\gamma^n, \widehat{\phi}_h^n-\gamma^n;\bm r_h, 0, 0)=0
			\qquad \forall\, \bm r_h\in \bm V_h, 
			\]
			which is exactly the hypothesis needed to invoke \eqref{es_u} for the triple
			$(\bm p_h^n, \phi_h^n-\gamma^n, \widehat{\phi}_h^n-\gamma^n)$.

			Thus,  by Sobolev inequality \eqref{sobolev-02} and \eqref{es_u},  we obtain
			\begin{align*}
				\|\phi_h^n-\gamma^n\|_{\mathcal{T}_h}&\le C(\|\nabla(\phi_h^n-\gamma^n)\|_{\mathcal{T}_h}
				+\alpha^{1/2}\|h_{\tau}^{-1/2}(\Pi_k^{\partial}(\phi_h^n-\gamma^n)-(\widehat{\phi}_h^n-\gamma^n))\|_{\partial\mathcal{T}_h})\\
				&=C(\|\nabla\phi_h^n\|_{\mathcal{T}_h}
				+\alpha^{1/2}\|h_{\tau}^{-1/2}(\Pi_k^{\partial}\phi_h^n-\widehat{\phi}_h^n)\|_{\partial\mathcal{T}_h})\\
				&\le C\interleave(\bm{p}_h^n, \phi_h^n, \widehat{\phi}_h^n)\interleave.
			\end{align*}
			By the definition of $ \mathcal{B} $ and integration by parts,  noticing the fact that $ \langle\bm{\beta}\cdot\bm{n}\mu_h, 1\rangle_{\partial \mathcal{T}_h}=0,  \forall \mu_h\in M_h  $,   one derives
			\begin{align}
				\mathcal{B}(u_h^{n-1}, \widehat{u}_h^{n-1};\phi_h^n)&=
				-(\bm{\beta}u_h^{n-1}, \nabla\phi_h^n)_{\mathcal{T}_h}+\langle\bm{\beta}\cdot\bm{n}\widehat{u}_h^{n-1}, \phi_h^n\rangle_{\partial\mathcal{T}_h}\nonumber\\
				&=-(\bm{\beta}u_h^{n-1}, \nabla(\phi_h^n-\gamma^n))_{\mathcal{T}_h}+\langle\bm{\beta}\cdot\bm{n}\widehat{u}_h^{n-1}, \phi_h^n-\gamma^n\rangle_{\partial\mathcal{T}_h}\nonumber\\
				&=(\bm{\beta}\cdot\nabla u_h^{n-1}, \phi_h^n-\gamma^n)_{\mathcal{T}_h}
				-\langle\bm{\beta}\cdot\bm{n}(u_h^{n-1}-\widehat{u}_h^{n-1}), \phi_h^n-\gamma^n\rangle_{\partial\mathcal{T}_h}\nonumber\\
				&\le C\|\nabla u_h^{n-1}\|\|\phi_h^n-\gamma^n\|_{\mathcal{T}_h}
				+C\|h_{\tau}^{-1/2}(u_h^{n-1}-\widehat{u}_h^{n-1})\|_{\partial\mathcal{T}_h}\|\phi_h^n-\gamma^n\|_{\mathcal{T}_h}\\
				&\le C(\|\nabla u_h^{n-1}\|_{\mathcal{T}_h}+\|h_{\tau}^{-1/2}(u_h^{n-1}-\widehat{u}_h^{n-1})\|_{\partial\mathcal{T}_h})\interleave(\bm{p}_h^n, \phi_h^n, \widehat{\phi}_h^n)\interleave\nonumber\\
				&\le C\interleave(\bm{q}_h^{n-1}, u_h^{n-1}, \widehat{u}_h^{n-1})\interleave
				\interleave(\bm{p}_h^n, \phi_h^n, \widehat{\phi}_h^n)\interleave.\label{sta-02}
			\end{align}
			where the last inequality we use the inequality \eqref{es_u}.}
		Recall the elementary inequality  $(a^3-b)(a-b) \geq \frac{1}{4}[(a^2-1)^2-(b^2-1)^2]$ and the identity $(a-b)a=\frac{a^2-b^2+(a-b)^2}{2}$.  Apply these and insert \eqref{sta-02} into \eqref{sta-01},  one obtains
		\begin{align*}
			& \frac{1}{4}\|( u_h^n)^2-1\|^2_{\mathcal{T}_h}+ \frac{\epsilon^2}{2}\interleave(\bm q_h^n,  u_h^n, \widehat{u}_h^n)\interleave^2 +\frac{\Delta t}{Pe}\interleave(\bm p_h^n,  \phi_h^n, \widehat{\phi}_h^n)\interleave^2
			\\
			&\qquad+\frac{\epsilon^2}{2} \|\Delta t \partial_t^+  \bm q_h^n\|^2_{\mathcal{T}_h}+\frac{\alpha\epsilon^2}{2} \|\Delta t h_{\tau}^{-1/2}\partial_{t}^+ (\Pi_k^{\partial}u_h^n-\widehat u_h^n)\|^2_{\partial\mathcal{T}_h}\\
			&\quad\leq	\frac{1}{4}\|( u_h^{n-1})^2-1\|^2_{\mathcal{T}_h}
			+\frac{\epsilon^2}{2}\interleave(\bm q_h^{n-1},  u_h^{n-1}, \widehat{u}_h^{n-1})\interleave^2-\Delta t  \mathcal B(u_h^{n-1}, \widehat u_h^{n-1}; \phi_h^n) \\
			&\quad\leq \frac{1}{4}\|( u_h^{n-1})^2-1\|^2_{\mathcal{T}_h}+\frac{\epsilon^2}{2}\interleave(\bm q_h^{n-1},  u_h^{n-1}, \widehat{u}_h^{n-1})\interleave^2+\frac{\Delta t}{2Pe}\interleave(\bm{p}_h^n, \phi_h^n, \widehat{\phi}_h^n)\interleave^2\\
			&\qquad+ C \Delta t\interleave(\bm{q}_h^{n-1}, u_h^{n-1}, \widehat{u}_h^{n-1})\interleave^2, 
		\end{align*}
		where,  in the last step,  we have applied the estimate \eqref{B-es},  the HDG-Sobolev embedding \eqref{sobolev-02},  and the inequality \eqref{es_u}. Summing from $n=1$ to $n=m$ yields:
		\begin{align*}
			&\frac{1}{4}\|( u_h^m)^2-1\|^2_{\mathcal{T}_h}+ \frac{\epsilon^2}{2}\interleave(\bm q_h^m,  u_h^m, \widehat{u}_h^m)\interleave^2 +\sum_{n=1}^m \frac{\Delta t}{2 Pe}\interleave(\bm p_h^n,  \phi_h^n, \widehat{\phi}_h^n)\interleave^2 \\
			&\qquad+\frac{\epsilon^2}{2} \sum_{n=1}^m \Big(\|\Delta t \partial_t^+  \bm q_h^n\|^2_{\mathcal{T}_h}+ \alpha\|\Delta t h_{\tau}^{-1/2}\partial_{t}^+ (\Pi_k^{\partial}u_h^n-\widehat u_h^n)\|^2_{\partial\mathcal{T}_h}\Big)\\
			&\quad\leq C+ C\sum_{n=1}^{m}  \Delta t \interleave(\bm q_h^{n-1},  u_h^{n-1}, \widehat{u}_h^{n-1})\interleave^2.
		\end{align*}
		The inequality \eqref{corEnergy_eq} follows directly from Gronwall’s inequality. 
		
		Applying 
		$\partial_{t}^+$ to \eqref{hdg02} yields
		\begin{subequations}
			\begin{align}
				&\epsilon^2	\mathcal{A}( \partial_{t}^+ \bm q_h^n, \partial_{t}^+ u_h^n, \partial_{t}^+ \widehat{u}_h^n;\bm r_2, w_2, \mu_2)
				+(\partial_{t}^+F(u_h^n,  u_h^{n-1}), w_2)_{\mathcal{T}_h}\label{hdg002}-(\partial_{t}^+\phi_h^n, w_2)_{\mathcal{T}_h}=0.
			\end{align}
			\end{subequations}
		where setting  $(\bm{ r}_2, w_2, \mu_2)=(\bm {p}_h^n,  -\phi_h^n, -\widehat{\phi}_h^n)$ yields
		\begin{align}
			&	\epsilon^2{\mathcal{A}( \partial_{t}^+ \bm q_h^n, \partial_{t}^+ u_h^n, \partial_{t}^+ \widehat{u}_h^n;\bm  p_h^n,  -\phi_h^n, -\widehat{\phi}_h^n)}
			-(\partial_{t}^+F(u_h^n,  u_h^{n-1}),  \phi_h^n)_{\mathcal{T}_h}\label{eq_C2} -(\partial_{t}^+\phi_h^n, -\phi_h^n)_{\mathcal{T}_h}=0. 
		\end{align}
		Taking $(\bm r_1,  w_1,  \mu_1)=(-\partial_{t}^+ \bm q_h^n, \partial_{t}^+ u_h^n, \partial_{t}^+ \widehat u_h^n)$ in \eqref{hdg01},   we obtain
		\begin{align}
			&	(\partial_{t}^+  u_h^n, \partial_{t}^+ u_h^n)_{\mathcal{T}_h}+\frac{1}{Pe}{\mathcal A(\bm p_h^n,  \phi_h^n,   \widehat{\phi}_h^n;- \partial_{t}^+\bm q_h^n,  \partial_{t}^+ u_h^n,  \partial_{t}^+ \widehat{u}_h^n)}\label{eq_C1}+{\mathcal{B}(u_h^{n-1},  \widehat{u}_h^{n-1}; { \partial_{t}^+ u_h^n}})=0. 
		\end{align}
		Multiplying \Cref{eq_C2} by $\frac{\Delta t}{Pe}$ and \Cref{eq_C1} by $\Delta t\epsilon^2$,  {and then summing the two equations,  we use the symmetry identity for $\mathcal{A}$ in \eqref{nonsymmetric-A} to obtain}
		\begin{align}\label{add:0}
			&\frac{1}{Pe}(\phi_h^{n}-\phi_h^{n-1}, \phi_h^n)_{\mathcal{T}_h}+\Delta t\epsilon^2\|\partial_{t}^+u_h^n\|^2_{\mathcal{T}_h}+\epsilon^2 \Delta t\mathcal{B}(u_h^{n-1},  \widehat{u}_h^{n-1}; \partial_{t}^+ u_h^n)\\
			&\qquad-\frac{\Delta t}{Pe}(\partial_{t}^+F(u_h^n,  u_h^{n-1}),  \phi_h^n)_{\mathcal{T}_h}=0.\nonumber
		\end{align}
		By recalling the definition of $\mathcal{B}$ in \eqref{def_B} and applying integration by parts,  we obtain
		\begin{align}\label{B-es2}
			&\mathcal{B}(u_h^{n-1},  \widehat{u}_h^{n-1}; \partial_{t}^+ u_h^n)=
			\big(\bm \beta \cdot \nabla u_h^{n-1},  \partial_t^+ u_h^n\big)_{\mathcal{T}_h}
			+ \big\langle \bm \beta \cdot \bm n(\widehat{u}_h^{n-1}-u_h^{n-1}),  \partial_t^+u_h^n\big\rangle_{\partial\mathcal T_h }. 
		\end{align}
		Substituting \eqref{B-es2} into \eqref{add:0} and using the identity 
		\begin{align*}
			(a-b, a)=\frac{1}{2}(a^2-b^2+(a-b)^2), 
		\end{align*}
		then summing over $n$ from $1$ to $m$,  we obtain
		\begin{align}
			&\frac{1}{2Pe}\|\phi_h^{m}\|^2_{\mathcal{T}_h}+\Delta t\epsilon^2\sum_{n=1}^m \|\partial_{t}^+u_h^n\|^2_{\mathcal{T}_h} 
			+\frac{1}{2Pe}\sum_{n=1}^{m}\|\Delta t\partial_t^+\phi_h^n\|_{\mathcal{T}_h}^2
			\nonumber\\
			&\leq\frac{\Delta t}{Pe}\sum_{n=1}^m\Big(\partial_{t}^+F(u_h^n,  u_h^{n-1}),  \phi_h^n\Big)_{\mathcal{T}_h}
			+\frac{1}{2Pe}\|\phi_h^{0}\|^2_{\mathcal{T}_h}\nonumber \\
			&\quad -\Delta t \epsilon^2 \sum_{n=1}^m \Big[\big(\bm \beta \cdot \nabla u_h^{n-1},  \partial_t^+ u_h^n\big)_{\mathcal{T}_h}+\big\langle \bm \beta \cdot \bm n(\widehat{u}_h^{n-1}-u_h^{n-1}),  \partial_t^+u_h^n\big\rangle_{\partial\mathcal T_h } \Big] \nonumber 
		\end{align}
		{ Using the Cauchy--Schwarz inequality,  the discrete trace inequality \eqref{classical_ine_5},  and \eqref{es_u},  we estimate}
		\begin{align*}
			&\Delta t \epsilon^2 \sum_{n=1}^m \big|\big(\bm \beta \cdot \nabla u_h^{n-1},  \partial_t^+ u_h^n\big)_{\mathcal{T}_h}+\big\langle \bm \beta \cdot \bm n(\widehat{u}_h^{n-1}-u_h^{n-1}),  \partial_t^+u_h^n\big\rangle_{\partial\mathcal T_h }\big|\\
			&\qquad\le C\Delta t\epsilon^2\sum_{n=1}^{m}(\|\partial_{t}^+u_h^n\|_{\mathcal{T}_h}\|\nabla u_h^{n-1}\|_{\mathcal{T}_h}+\|h_{\tau}^{-1/2}(\widehat{u}_h^{n-1}-u_h^{n-1})\|_{\partial\mathcal{T}_h}\|\partial_t^+u_h^n\|_{\mathcal{T}_h})\\
			&\qquad\leq \frac{\Delta t \epsilon^2}{4} \sum_{n=1}^m\|\partial_{t}^+u_h^n\|^2_{\mathcal{T}_h} 
			+C\epsilon^2 \Delta t \sum_{n=1}^m\big( \|\nabla u_h^{n-1} \|_{\mathcal{T}_h}^2+ \|h_{\tau} ^{-1/2}(u_h^{n-1}-\widehat{u}_h^{n-1}) \|_{\partial \mathcal{T}_h}^2\big) \\
			&\qquad\leq \frac{\Delta t \epsilon^2}{4} \sum_{n=1}^m\|\partial_{t}^+u_h^n\|^2_{\mathcal{T}_h}+ C \epsilon^2\Delta t \sum_{n=1}^m\interleave(\bm q_h^{n-1},  u_h^{n-1},  \widehat{u}_h^{n-1})\interleave^2.
		\end{align*}
		For $n\ge 2$,  notice that
			\begin{align*}
				\partial_{t}^+F(u_h^n, u_h^{n-1})
				&=\frac{F(u_h^n, u_h^{n-1})-F(u_h^{n-1}, u_h^{n-2})}{\Delta t}
				=\frac{(u_h^n)^3-u_h^{n-1}-(u_h^{n-1})^3+u_h^{n-2}}{\Delta t}\\
				&=\frac{(u_h^n-u_h^{n-1})\bigl((u_h^n)^2+u_h^n u_h^{n-1}+(u_h^{n-1})^2\bigr)-(u_h^{n-1}-u_h^{n-2})}{\Delta t}\\
				&= \partial_{t}^+u_h^n\bigl((u_h^n)^2+u_h^n u_h^{n-1}+(u_h^{n-1})^2\bigr)-\partial_{t}^+u_h^{n-1}.
			\end{align*}
			In addition,  the term with $n=1$ can be estimated directly and absorbed into the generic constant.
			Using H\"older's inequality,  the HDG Sobolev inequalities \eqref{sobolev-01}--\eqref{sobolev-02}, 
			\eqref{es_u},  and the stability bound \eqref{corEnergy_eq},  we obtain
			\begin{align*}
				&\frac{\Delta t}{Pe}\sum_{n=1}^{m}(\partial_{t}^+F(u_h^n, u_h^{n-1}), \phi_h^n)_{\mathcal{T}_h
				}=	\frac{\Delta t}{Pe}\sum_{n=1}^{m}(\partial_{t}^+u_h^n[(u_h^n)^2+u_h^nu_h^{n-1}+(u_h^{n-1})^2]-\partial_{t}^+u_h^{n-1}, \phi_h^n)_{\mathcal{T}_h}\\
				&\le C\Delta t\sum_{n=1}^{m}\|\partial_{t}^+u_h^n\|_{\mathcal{T}_h}(\|(u_h^n)^2\|_{L^3(\Omega)}+\|(u_h^{n-1})^2\|_{L^3(\Omega)})\|\phi_h^n\|_{L^6(\Omega)} +C\Delta t\sum_{n=1}^{m}\|\partial_{t}^+u_h^{n-1}\|_{\mathcal{T}_h}\|\phi_{h}^n\|_{\mathcal{T}_h}\\
				&=C\Delta t\sum_{n=1}^{m}\|\partial_{t}^+u_h^n\|_{\mathcal{T}_h}(\|u_h^n\|_{L^6(\Omega)}^2+\|u_h^{n-1}\|_{L^6(\Omega)}^2)\|\phi_h^n\|_{L^6(\Omega)} +C\Delta t\sum_{n=1}^{m}\|\partial_{t}^+u_h^{n-1}\|_{\mathcal{T}_h}\|\phi_{h}^n\|_{\mathcal{T}_h}\\
				&\le C\Delta t\sum_{n=1}^{m}\|\partial_{t}^+u_h^n\|_{\mathcal{T}_h}(
				\|u_h^n\|_{\mathcal{T}_h}^2+
				\|\nabla u_h^n\|_{\mathcal{T}_h}^2+\alpha\|h_{\tau}^{-1/2}(\Pi_k^{\partial}u_h^n-\widehat{u}_h^n)\|_{\partial\mathcal{T}_h}^2
				+\| u_h^{n-1}\|_{\mathcal{T}_h}^2
				+\|\nabla u_h^{n-1}\|_{\mathcal{T}_h}^2\\
				&\quad
				+\alpha\|h_{\tau}^{-1/2}(\Pi_k^{\partial}u_h^{n-1}-\widehat{u}_h^{n-1})\|_{\partial\mathcal{T}_h}^2)
				(\|\phi_h^n\|_{\mathcal{T}_h}+\|\nabla\phi_h^n\|_{\mathcal{T}_h}+\alpha^{1/2}\|h_{\tau}^{-1/2}(
				\Pi_k^{\partial}\phi_h^n-\widehat{\phi}_{h}^n)\|_{\partial\mathcal{T}_h})\\
				&\quad +C\Delta t\sum_{n=1}^{m}\|\partial_{t}^+u_h^{n-1}\|_{\mathcal{T}_h}
				\|\phi_h^n\|_{\mathcal{T}_h}\\
				&\le C\Delta t\sum_{n=1}^{m}\|\partial_{t}^+u_h^n\|_{\mathcal{T}_h}
				(\|\phi_h^n\|_{\mathcal{T}_h}+\interleave(\bm{p}_h^n, \phi_h^n, \widehat{\phi}_h^n)\interleave)
				+C\Delta t\sum_{n=1}^{m}\|\partial_{t}^+u_h^{n-1}\|_{\mathcal{T}_h}\|\phi_h^n\|_{\mathcal{T}_h}
				\\
				&\le C+\frac{\epsilon^2\Delta t}{4} \sum_{n=1}^{m}\|\partial_{t}^+u_h^n\|_{\mathcal{T}_h}^2
				+C\Delta t\sum_{n=1}^{m}\interleave(\bm{p}_h^n, \phi_h^n, \widehat{\phi}_h^n)\interleave^2,  
			\end{align*}
			where we use the inequality
			\begin{align*}
				\|\phi_h^n\|_{\mathcal{T}_h}&\le \|\phi_h^n-\gamma^n\|_{\mathcal{T}_h}+\|\gamma^n\|_{\mathcal{T}_h}
				\le \interleave(\bm{p}_h^n, \phi_h^n, \widehat{\phi}_h^n)\interleave+C.
			\end{align*}		
			
		Thus,  invoking the stability estimate \eqref{corEnergy_eq},  we conclude 
		\begin{align*}
			\frac{1}{2Pe}\|\phi_h^{m}\|^2_{\mathcal{T}_h}+\frac{\Delta t\epsilon^2}{2}\sum_{n=1}^m \|\partial_{t}^+u_h^n\|^2_{\mathcal{T}_h} +\frac{1}{2Pe}\sum_{n=1}^{m}\|\Delta t\partial_{t}^+\phi_h^n\|_{\mathcal{T}_h}^2 \leq C.
		\end{align*}
		{ We note that the  generic constant $C$ is independent of $h$ and $\Delta t$, but may depend on $\epsilon$ and ${Pe}$.} This establishes the estimate \eqref{phi_infty-1}.

		The remaining proofs for \eqref{uh_infty_bounded} and the uniform estimate of the HDG elliptic projection follow similarly,  using the triangle inequality,  inverse inequality,  and stability results to bound $u_{Ih}$ in $L^\infty$. We provide the details below.
		
		Recalling the definition of the discrete Laplacian in \eqref{Delta_h},  we choose $(\bm{r}_2, w_2, \mu_2)=(\bm{0}, \Delta_hu_h^n, 0)$ in \eqref{hdg02} to obtain 
		\begin{align*}
			(F(u_h^n, u_h^{n-1}), \Delta_hu_h^n)_{\mathcal{T}_h}-\epsilon^2\|\Delta_hu_h^n\|_{\mathcal{T}_h}^2-(\phi_h^n, \Delta_hu_h^n)_{\mathcal{T}_h}=0.
		\end{align*}
		Applying the Cauchy–Schwarz inequality and the HDG Sobolev inequality from \cref{discrete-soblev-hdg},  we derive
		\begin{align*}
			\epsilon^2\|\Delta_hu_h^n\|_{\mathcal{T}_h}^2& =(F(u_h^n, u_h^{n-1}), \Delta _hu_h^n)_{\mathcal{T}_h}-(\phi_h^n, \Delta_hu_h^n)_{\mathcal{T}_h}\\
			&\le (\|u_h^n\|_{L^6(\Omega)}^3+\|u_h^{n-1}\|_{\mathcal{T}_h})\|\Delta _hu_h^n\|_{\mathcal{T}_h}+
			\|\phi_h^n\|_{\mathcal{T}_h}\|\Delta_hu_h^n\|_{\mathcal{T}_h}\\
			&\le C\|\Delta_hu_h^n\|_{\mathcal{T}_h}(\interleave(\bm{q}_h^{n}, u_h^{n}, \widehat{u}_h^{n})\interleave^3+\interleave(\bm{q}_h^{n-1}, u_h^{n-1}, \widehat{u}_h^{n-1})\interleave+\|\phi_h^n\|_{\mathcal{T}_h}).
		\end{align*}
		Applying \cref{lap_u_h2} and \eqref{corEnergy_eq},  we complete the proof of the first term in \eqref{uh_infty_bounded}.
		
		Next,  we establish the uniform bound for the HDG elliptic projection $ u_{Ih} $. By the triangle and inverse inequalities,  we obtain
		\begin{align*}
			\|u_{Ih}^n-u^n\|_{L^{\infty}(\Omega)}&\le\|u_{Ih}^n-\Pi_{k+1}^ou^n\|_{L^{\infty}(\Omega)}+\|u^n-\Pi_{k+1}^ou^n\|_{L^{\infty}(\Omega)}\\
			& \le Ch^{-d/2}\|u_{Ih}^n-\Pi_{k+1}^ou^n\|_{\mathcal{T}_h}+Ch^{2-d/2}|u^n|_2\\
			& \le Ch^{2-d/2}|u^n|_2
		\end{align*}
		for all $ n=1, 2, \cdots,  N $. Hence with $ \|u\|_{L^{\infty}(H^2(\Omega))}\le C $,  one obtains uniform bound for $ u_{Ih} $.
	\end{proof}

	\section{Error analysis}
	\label{error_analysis}
	The main convergence result is summarized in the following theorem.
	\begin{theorem}\label{main_res}
		Suppose the solution $(\bm p,  \phi,  \bm q,  u)$ attains the maximum regularity required for the best approximation results in \eqref{classical_ine}. Then,  for any $h,  \Delta t$, 
		the following optimal error estimates hold for polynomials of degree $k\geq 0$:
		\begin{align*}
			\max_{1\le n\le N} \|u^n-u_h^n\|_{L^2(\Omega)}^2+\Delta t \sum_{n=1}^N \|\phi^n -\phi_h^n\|_{L^2(\Omega)}^2  \le  C (h^{k+2}+\Delta t)^2,  \\
			\max_{1\le n\le N} \|\bm q^n -\bm q_h^n\|_{L^2(\Omega)}^2+\Delta t \sum_{n=1}^N \|\bm p^n -\bm p_h^n\|_{L^2(\Omega)}^2  \le  C (h^{k+1}+\Delta t)^2.
		\end{align*}
	\end{theorem}
	\begin{remark}\label{proof_main_res}
		The proof of \Cref{main_res} is presented in \Cref{scalar-variables-error-result,lem-err-flux}. In \Cref{scalar-variables-error-resultremark},  we explain why,  if the operator $\mathcal{B}$ is chosen in the form given by \eqref{def_B_1},  the scalar variables $u$ and $\phi$ do not achieve optimal convergence rates when $k = 0$,  whereas optimal rates are retained for $k \ge 1$. More specifically,  the error estimate takes the following form:
		\begin{align*}
			\max_{1\le n\le N} \|u^n-u_h^n\|_{L^2(\Omega)}^2+\Delta t \sum_{n=1}^N \|\phi^n -\phi_h^n\|_{L^2(\Omega)}^2  \le  C (h^{k+1+\min\{k, 1\}}+\Delta t)^2.
		\end{align*}
	\end{remark}

	We introduce the following notations for the errors:
	\begin{subequations}
		\begin{align}
			e_h^{\bm p^n}:=\bm p_{Ih}^n-\bm p_h^n, \ \ \ e_h^{\phi^n}:=\phi_{Ih}^n-\phi_h^n, \ \ \ e_h^{\widehat \phi^n}:=\widehat \phi_{Ih}^n-\widehat \phi_{h}^n, \label{def_e1}\\
			e_h^{\bm q^n}:=\bm q_{Ih}^n-\bm q_h^n, \ \ \ e_h^{u^n}:=u_{Ih}^n-u_h^n, \ \ \ e_h^{\widehat u^n}:=\widehat u_{Ih}^n-\widehat u_{h}^n.\label{def_e2}
		\end{align}
	\end{subequations}
	
	Using  the definition of $\mathcal A$ in \eqref{def_A} and the HDG elliptic projection \eqref{HDG-projection},  we derive the following error equations for all $(\bm r_1,  w_1,  \mu_1)$,  $(\bm r_2,  w_2,  \mu_2)$ in $\bm V_h\times W_h\times M_h$:
	\begin{subequations}
		\begin{align}
			& (\partial_{t}^+ e_h^{u^n}, w_1)_{\mathcal{T}_h}+\frac{1}{Pe}
			\mathcal A(e_h^{\bm p^n}, e_h^{\phi^n}, e_h^{\widehat \phi^n};\bm r_1,  w_1,  \mu_1)
			+\mathcal{B}(u_{Ih}^n, \widehat{u}_{Ih}^n;w_1)
			\nonumber\\
			&\qquad-\mathcal{B}(u_h^{n-1}, \widehat{u}_h^{n-1};w_1)=(\partial_{t}^+ u_{I h}^n-\partial_t u^n, w_1)_{\mathcal{T}_h}\label{error_01}
			\\
			&{\epsilon^2\mathcal{A}(e_h^{\bm q^n}, e_h^{u^n}, e_h^{\widehat u^n};\bm r_2,  w_2,  \mu_2)-(e_h^{\phi^n},  w_2)_{\mathcal{T}_h}}=(\phi^n-\phi^n_{Ih}, w_2)_{\mathcal{T}_h}\label{error_02}\\
			&\qquad+(F(u_h^n,  u_h^{n-1})-f(u^n), w_2)_{\mathcal{T}_h}. \nonumber 
		\end{align}
	\end{subequations}
	\begin{lemma}\label{error-term-B}
		suppose $ u_h^{n}, \widehat{u}_h^{n} $ to be the solutions of HDG scheme and $ u_{Ih}^{n},  \widehat{u}_{Ih}^{n} $ to be the HDG elliptic projection,  respectively. For any $ (w_h, \mu_h)\in W_h\times M_h $,  the following bound holds:
		\begin{align}\label{B-error-bounded}
			\begin{split}
				&|\mathcal{B}(u_{Ih}^n, \widehat{u}_{Ih}^n;w_h)-\mathcal{B}(u_h^{n-1}, \widehat{u}_h^{n-1};w_h)| \leq C(h^{k+2}+\|\Delta t\partial_{t}^+\Pi_{k+1}^ou^n\|_{\mathcal{T}_h}\\
				&\quad+\|e_h^{u^{n-1}}\|_{\mathcal{T}_h}+h\|h_{\tau}^{-1/2}(\Pi_k^{\partial}e_h^{u^{n-1}}-e_h^{\widehat{u}^{n-1}})\|_{\partial\mathcal{T}_h}) (\|\nabla w_h\|_{\mathcal{T}_h}+\|h_{\tau}^{-1/2}(w_h-\mu_h)\|_{\partial\mathcal{T}_h}).
			\end{split}		
		\end{align}
		Moreover,  if $	\mathcal A(\bm p_h, w_h, \mu_h; \bm r_h, 0, 0)=0 $ for all $\bm r_h\in \bm V_h$,  it holds
			\begin{align}\label{B-error-bounded-2}
				\begin{split}
					&|\mathcal{B}(u_{Ih}^n, \widehat{u}_{Ih}^n;w_h)-\mathcal{B}(u_h^{n-1}, \widehat{u}_h^{n-1};w_h)| \leq C(h^{k+2}+\|\Delta t\partial_{t}^+\Pi_{k+1}^ou^n\|_{\mathcal{T}_h}\\
					&\quad+\|e_h^{u^{n-1}}\|_{\mathcal{T}_h}+h\|h_{\tau}^{-1/2}(\Pi_k^{\partial}e_h^{u^{n-1}}-e_h^{\widehat{u}^{n-1}})\|_{\partial\mathcal{T}_h}) \interleave(\bm{p}_h, w_h, \mu_h)\interleave.	
				\end{split}			
			\end{align}
		
	\end{lemma}
	\begin{proof}
		We recall the error notations defined in \eqref{error-notation-varepsilon}:
		\begin{align*}
			\varepsilon_h^{u^n}=\Pi_{k+1}^ou^n-u_{Ih}^n, \ 
			\varepsilon_h^{\widehat{u}^n}=\Pi_k^{\partial}u^n-\widehat{u}_{Ih}^n, 
		\end{align*}
		hence together with
		\begin{align*}
			u_{Ih}^n-u_h^{n-1}&=-\varepsilon_h^{u^n}+\Delta t\partial_t^+\Pi_{k+1}^ou^n+\varepsilon_h^{u^{n-1}}+e_h^{u^{n-1}}, \\
			\widehat{u}_{Ih}^n-\widehat{u}_h^{n-1}&=-\varepsilon_h^{\widehat{u}^n}+\Delta t\partial_t^+\Pi_{k}^{\partial}u^n+\varepsilon_h^{\widehat{u}^{n-1}}+e_h^{\widehat{u}^{n-1}}, 
		\end{align*}
		we can rewrite the error term as:
		\begin{align*}
			&\mathcal{B}(u_{Ih}^n, \widehat{u}_{Ih}^n;w_h)-\mathcal{B}(u_h^{n-1}, \widehat{u}_h^{n-1};w_h)\\
			&=\mathcal{B}(-\varepsilon_h^{u^n}, -\varepsilon_h^{\widehat{u}^n};w_h)
			+\mathcal{B}(\Delta t\partial_{t}^+\Pi_{k+1}^ou^n, \Delta t\partial_{t}^+\Pi_k^{\partial}u^n;w_h)\\
			&\quad +\mathcal{B}(\varepsilon_h^{u^{n-1}}, \varepsilon_h^{\widehat{u}^{n-1}};w_h)
			+\mathcal{B}(e_h^{u^{n-1}}, e_h^{\widehat{u}^{n-1}};w_h)\\
			&=\sum_{i=1}^{4}R_i.
		\end{align*}
		Next,  we perform the error estimate for the right hand terms $ \{R_i\}_{i=1}^4$.  By \eqref{B-bound-ori} and approximation of HDG elliptic projection,  one obtains
		\begin{align*}
			R_1+R_3&\le C(\|\varepsilon_h^{u^n}\|_{\mathcal{T}_h}+h\|h_{\tau}^{-1/2}(\Pi_k^{\partial}\varepsilon_h^{u^n}-\varepsilon_h^{\widehat{u}^n})\|_{\partial\mathcal{T}_h})
			(\|\nabla w_h\|_{\mathcal{T}_h}+\|h_{\tau}^{-1/2}(w_h-\mu_h)\|_{\partial\mathcal{T}_h})
			\\
			& +C(\|\varepsilon_h^{u^{n-1}}\|_{\mathcal{T}_h}+h\|h_{\tau}^{-1/2}(\Pi_k^{\partial}\varepsilon_h^{u^{n-1}}-\varepsilon_h^{\widehat{u}^{n-1}})\|_{\partial\mathcal{T}_h})(\|\nabla w_h\|_{\mathcal{T}_h}+\|h_{\tau}^{-1/2}(w_h-\mu_h)\|_{\partial\mathcal{T}_h})\\
			&\le Ch^{k+2}(|u^n|_{k+2}+|u^{n-1}|_{k+2})(\|\nabla w_h\|_{\mathcal{T}_h}+\|h_{\tau}^{-1/2}(w_h-\mu_h)\|_{\partial\mathcal{T}_h}), \\
			R_4&\le C(\|e_h^{u^{n-1}}\|_{\mathcal T_h}+h\|h_{\tau}^{-1/2}(\Pi_k^{\partial}e_h^{u^{n-1}}-e_h^{\widehat{u}^{n-1}})\|_{\partial\mathcal{T}_h})(\|\nabla w_h\|_{\mathcal{T}_h}+\|h_{\tau}^{-1/2}(w_h-\mu_h)\|_{\partial\mathcal{T}_h}).
		\end{align*}
		Using $ L^2 $ stability of $ \Pi_k^{\partial} $ and the approximation property of $ \Pi_{k+1}^o $,  one deduces
		\begin{align*}
			R_2&\le C(\|\Delta t\partial_{t}^+\Pi_{k+1}^ou^n\|_{\mathcal{T}_h}+h\|h_{\tau}^{-1/2}\Delta t\partial_{t}^+(\Pi_k^{\partial}\Pi_{k+1}^ou^n-\Pi_k^{\partial}u^n)\|_{\partial\mathcal{T}_h})\\
			&\quad \times(\|\nabla w_h\|_{\mathcal{T}_h}+\|h_{\tau}^{-1/2}(w_h-\mu_h)\|_{\partial\mathcal{T}_h})\\
			&\le C(\|\Delta t\partial_{t}^+\Pi_{k+1}^ou^n\|_{\mathcal{T}_h}+Ch^{k+2}|\Delta t\partial_{t}^+u^n|_{k+2})(\|\nabla w_h\|_{\mathcal{T}_h}+\|h_{\tau}^{-1/2}(w_h-\mu_h)\|_{\partial\mathcal{T}_h}).
		\end{align*}
		Combining all above estimates for $ R_1-R_4$,  we can conclude the  result \eqref{B-error-bounded}. Using \eqref{B-error-bounded} and \eqref{B-es},  we derive the bound \eqref{B-error-bounded-2}.
	\end{proof}

	The error analysis is divided into three lemmas. We first establish an error estimate in the negative norm.
	\begin{lemma}[Error estimates in the negative norm]\label{lemma-negative-norm}
		Suppose the solution $(\bm p,  \phi,  \bm q,  u)$ attains the maximum regularity required for the best approximation results in \eqref{classical_ine}.  Then 
		for any $ h $,  $ \Delta t $ and $ m=1, 2, \cdots, N $,  the following negative norm error bound holds:
		\begin{align}\label{negative-norm-error-estimate}
			\max_{1\le n\le m}\|e_h^{u^n}\|_{-1, h}^2+C\Delta t\sum_{n=1}^{m}\interleave(e_h^{\bm{q}^n}, e_h^{u^n}, e_h^{\widehat{u}^n})\interleave^2\le C(\Delta t^2+h^{2k+4}).
		\end{align}
	\end{lemma}
	\begin{proof}
		Substituting  $ (\bm{r}_1, w_1, \mu_1)=Pe(-\bm{\Pi_V}e_h^{u^n}, \Pi_We_h^{u^n}, \Pi_Me_h^{u^n}) $ into  \eqref{error_01} and  $ (\bm{r}_2, w_2, $ $\mu_2)=(e_h^{\bm{q}^n},   e_h^{u^n}, e_h^{\widehat{u}^n}) $ into \eqref{error_02},  then summing the resulting equations,  we obtain
		\begin{align}\label{negative-es-eq}
			&Pe(\partial_{t}^+e_h^{u^n}, \Pi_We_h^{u^n})_{\mathcal{T}_h}+\epsilon^2\interleave(e_h^{\bm{q}^n}, e_h^{u^n}, e_h^{\widehat{u}^n}) \interleave^2=Pe(\partial_{t}^+u_{Ih}^n-\partial_tu^n, \Pi_We_h^{u^n})_{\mathcal{T}_h}
			\\
			&+(\phi^n-\phi_{Ih}^n, e_h^{u^n})_{\mathcal{T}_h}+(F(u_h^n, u_{h}^{n-1})-f(u^n), e_h^{u^n})_{\mathcal{T}_h}\\
			&+Pe[\mathcal{B}(u_h^{n-1}, \widehat{u}_h^{n-1};\Pi_We_h^{u^n})-\mathcal{B}(u_{Ih}^n, \widehat{u}_{Ih}^n;\Pi_We_h^{u^n})]
			.
		\end{align}
		Utilizing the definitions of  $ \bm{\Pi_V}, \Pi_W, \Pi_M $ and \eqref{def_A},  we derive
		\begin{align}\label{temp-2-01}
			(\partial_t^+e_h^{u^n}, \Pi_We_h^{u^n})_{\mathcal{T}_h}&=\mathcal{A}(\bm{\Pi_V}\partial_{t}^+e_h^{u^n}, \Pi_W\partial_t^+e_h^{u^n}, \Pi_M\partial_t^+e_h^{u^n};\bm{0}, \Pi_We_h^{u^n}, \Pi_Me_h^{u^n})\\
			&=(\nabla\cdot\bm{\Pi_V}\partial_t^+e_h^{u^n}, \Pi_We_h^{u^n})_{\mathcal{T}_h}-\langle \bm{n}\cdot\bm{\Pi_V}\partial_{t}^+e_h^{u^n}, \Pi_Me_h^{u^n}\rangle_{\partial \mathcal{T}_h}\nonumber\\
			&\quad +\alpha\langle h_{\tau}^{-1}(\Pi_k^{\partial}\Pi_We_h^{u^n}-\Pi_Me_h^{u^n}), \partial_{t}^+(\Pi_k^{\partial}\Pi_We_h^{u^n}-\Pi_Me_h^{u^n})\rangle_{\partial \mathcal{T}_h}\nonumber.
		\end{align}
		{ On the other hand, by the definition of the HDG Laplace inverse in \eqref{def_-1_h}, we have 
			\begin{align}\label{temp-2-02}
				\mathcal{A}(\bm{\Pi_V}e_h^{u^n},\Pi_We_h^{u^n},\Pi_Me_h^{u^n};\bm{r}_h,w_h,\mu_h)
				=(e_h^{u^n},w_h)_{\mathcal{T}_h},
			\end{align}
			for all $(\bm{r}_h,w_h,\mu_h)$ in the corresponding test spaces. 
			Choosing $(\bm{r}_h,w_h,\mu_h)=(\bm{\Pi_V}\partial_t^+e_h^{u^n},0,0)$ in \eqref{temp-2-02}, and using the definition of $\mathcal{A}$ in \eqref{def_A},  we obtain} 
		\begin{align}\label{temp-2-03}
			(\nabla\cdot\bm{\Pi_V}\partial_t^+e_h^{u^n}, \Pi_We_h^{u^n})_{\mathcal{T}_h}-\langle \bm{n}\cdot\bm{\Pi_V}\partial_{t}^+e_h^{u^n}, \Pi_Me_h^{u^n}\rangle_{\partial \mathcal{T}_h}=(\bm{\Pi_V}e_h^{u^n}, \partial_{t}^+\bm{\Pi_V}e_h^{u^n})_{\mathcal{T}_h}.
		\end{align}
		Recall the definition of DG Laplace inverse in \eqref{def_-1_h},  one has
			\begin{align}\label{-1--norm--1}
				\|e_h^{u^n}\|_{-1, h}^2&=\mathcal{A}(\bm{\Pi}_{\bm{V}}e_h^{u^n}, \Pi_We_h^{u^n}, \Pi_{M}e_h^{u^n};\bm{\Pi}_{\bm{V}}e_h^{u^n}, \Pi_We_h^{u^n}, \Pi_{M}e_h^{u^n})\\
				&=\|\Pi_{\bm{V}}e_h^{u^n}\|_{\mathcal{T}_h}^2+\alpha\|h_{\tau}^{-1/2}(\Pi_We_h^{u^n}-\Pi_{M}e_h^{u^n})\|_{\partial\mathcal{T}_h}^2, \nonumber\\
				\label{-1--norm--2}
				\|\partial_{t}^+e_h^{u^n}\|_{-1, h}^2&=\mathcal{A}(\partial_{t}^+\bm{\Pi}_{\bm{V}}e_h^{u^n}, \partial_{t}^+\Pi_We_h^{u^n}, \partial_{t}^+\Pi_{M}e_h^{u^n};\partial_{t}^+\bm{\Pi}_{\bm{V}}e_h^{u^n}, \partial_{t}^+\Pi_We_h^{u^n}, \partial_{t}^+\Pi_{M}e_h^{u^n})\\
				&=\|\partial_{t}^+\Pi_{\bm{V}}e_h^{u^n}\|_{\mathcal{T}_h}^2+\alpha\|h_{\tau}^{-1/2}\partial_{t}^+(\Pi_We_h^{u^n}-\Pi_{M}e_h^{u^n})\|_{\partial\mathcal{T}_h}^2. \nonumber
		\end{align}
		Thus, substituting \eqref{temp-2-03} into \eqref{temp-2-01} and applying the   identity $(a-b)a=\frac{a^2-b^2+(a-b)^2}{2}$, together with \eqref{-1--norm--1} and \eqref{-1--norm--2}, we obtain
			\begin{align}\label{temp-2-05}
				&(\partial_{t}^+e_h^{u^n}, \Pi_We_h^{u^n})_{\mathcal{T}_h}\\
				&=(\Pi_{\bm{V}}e_h^{u^n}, \partial_{t}^+\Pi_{\bm{V}}e_h^{u^n})_{\mathcal{T}_h}
				+\alpha\langle h_{\tau}^{-1/2}(\Pi_k^{\partial}\Pi_We_h^{u^n}-\Pi_{M}e_h^{u^n}), \partial_{t}^+(\Pi_k^{\partial}\Pi_We_h^{u^n}-\Pi_{M}e_h^{u^n})\rangle_{\partial \mathcal{T}_h}\nonumber\\
				&=\frac{1}{2\Delta t}(\|\Pi_{\bm{V}}e_h^{u^n}\|_{\mathcal{T}_h}^2-\|\Pi_{\bm{V}}e_h^{u^{n-1}}\|_{\mathcal{T}_h}^2+\|\Delta t\partial_{t}^+(\Pi_{\bm{V}}e_h^{u^n})\|_{\mathcal{T}_h}^2)\nonumber\\
				&\quad +\frac{\alpha}{2\Delta t}(\|h_{\tau}^{-1/2}(\Pi_k^{\partial}\Pi_We_h^{u^n}-\Pi_{M}e_h^{u^n})\|_{\partial\mathcal{T}_h}^2
				-\|h_{\tau}^{-1/2}(\Pi_k^{\partial}\Pi_We_h^{u^{n-1}}-\Pi_{M}e_h^{u^{n-1}})\|_{\partial\mathcal{T}_h}^2)\nonumber\\
				&\quad +\frac{\alpha}{2\Delta t}\|\Delta t\partial_{t}^+(\Pi_k^{\partial}\Pi_We_h^{u^n}-\Pi_Me_h^{u^n})\|_{\partial\mathcal{T}_h}^2\nonumber\\
				&=\frac{1}{2}\partial_{t}^+\|e_h^{u^n}\|_{-1, h}^2+\frac{\Delta t}{2}\|\partial_{t}^+e_h^{u^n}\|_{-1, h}^2.\nonumber
			\end{align}

		Consequently,  substituting \eqref{temp-2-05} into \eqref{negative-es-eq},  the following error equation holds:
		\begin{gather}\label{error-eq-negative norm}
			\begin{split}
				&\frac{Pe}{2}\partial_{t}^+\|e_h^{u^n}\|_{-1, h}^2+\frac{Pe\Delta t}{2}\|\partial_{t}^+e_h^{u^n}\|_{-1, h}^2
				+\epsilon^2\interleave(e_h^{\bm{q}^n}, e_h^{u^n}, e_h^{\widehat{u}^n}) \interleave^2\\
				&=Pe(\partial_{t}^+u_{Ih}^n-\partial_tu^n, \Pi_We_h^{u^n})_{\mathcal{T}_h}
				+(\phi^n-\phi_{Ih}^n, e_h^{u^n})_{\mathcal{T}_h}+(F(u_h^n, u_{h}^{n-1})-f(u^n), e_h^{u^n})_{\mathcal{T}_h}\\
				&\quad+Pe[\mathcal{B}(u_h^{n-1}, \widehat{u}_h^{n-1};\Pi_We_h^{u^n})-\mathcal{B}(u_{Ih}^n, \widehat{u}_{Ih}^n;\Pi_We_h^{u^n})]\\
				&=:\sum_{i=1}^{4}R_i.
			\end{split}
		\end{gather}
		
		{ Next, we estimate the right-hand terms \(\{R_i\}_{i=1}^4\). 
			By the definition of the HDG Laplace inverse in \eqref{def_-1_h}, we have
			\[
			\mathcal{A}(\Pi_{\bm V}e_h^{u^n}, \Pi_W e_h^{u^n}, \Pi_M e_h^{u^n};\, \bm r_h, 0, 0)=0,
			\qquad \forall\, \bm r_h\in \bm V_h .
			\]
			Therefore, applying the HDG Sobolev inequality \eqref{sobolev-02} with 
			\((w_h,\mu_h)=(\Pi_W e_h^{u^n},\Pi_M e_h^{u^n})\), and using \eqref{es_u}, we obtain
			\begin{align}
				\|\Pi_W e_h^{u^n}\|_{\mathcal{T}_h}
				&\le C\Big(\|\nabla \Pi_W e_h^{u^n}\|_{\mathcal{T}_h}
				+\alpha^{1/2}\|h_{\tau}^{-1/2}(\Pi_k^{\partial}\Pi_W e_h^{u^n}-\Pi_M e_h^{u^n})\|_{\partial\mathcal{T}_h}\Big)\nonumber\\
				&\le C\Big(\|\Pi_{\bm V} e_h^{u^n}\|_{\mathcal{T}_h}
				+\alpha^{1/2}\|h_{\tau}^{-1/2}(\Pi_k^{\partial}\Pi_W e_h^{u^n}-\Pi_M e_h^{u^n})\|_{\partial\mathcal{T}_h}\Big)\nonumber\\
				&\le C\,\|e_h^{u^n}\|_{-1,h}.\label{temp-2-04}
			\end{align}
			Consequently, by the Cauchy--Schwarz inequality, Young's inequality, and \eqref{temp-2-04}, we derive
			\begin{align*}
				R_1\le  C\|\partial_t^+u_{Ih}^n-\partial_tu^n\|_{\mathcal{T}_h}
				\|\Pi_We_h^{u^n}\|_{\mathcal{T}_h}
				\le C\|\partial_t^+u_{Ih}^n-\partial_tu^n\|_{\mathcal{T}_h}^2+\|e_h^{u^n}\|_{-1, h}^2.
			\end{align*}
		}
		 Applying Young’s inequality  and \eqref{sobolev-02},  \eqref{es_u} to  $R_2$,  we obtain
			\begin{align*}
				R_2&\le \|\phi^n-\phi_{Ih}^n\|_{\mathcal{T}_h}\|e_h^{u^n}\|_{\mathcal{T}_h}
				\le {\|\phi^n-\phi_{Ih}^n\|_{\mathcal{T}_h}}\interleave (e_h^{\bm{q}^n}, e_h^{u^n}, e_h^{\widehat{u}^n})\interleave\\
				&\le {C\|\phi^n-\phi_{Ih}^n\|_{\mathcal{T}_h}^2}+\delta\epsilon^2\interleave (e_h^{\bm{q}^n}, e_h^{u^n}, e_h^{\widehat{u}^n})\interleave^2, 
			\end{align*}
		
		where $ \delta $ is an arbitrary positive constant to be determined later.

		For the nonlinear term,  we have 
		\begin{align*}
			F(u_h^n,  u_h^{n-1})-f(u^n)&=(u_h^n)^3-(u_{Ih}^n)^3+ (u_{Ih}^n)^3-(u^n)^3 +u^n-u_{Ih}^n+e_h^{u^n}+\Delta t\partial_t^+u_h^n.
		\end{align*}
		Introducing {
			\begin{align}\label{xi_hn}
				\xi_h^n
				:=(u_h^n)^2+u_h^nu_{Ih}^n+(u_{Ih}^n)^2
				=\Big(u_h^n+\frac{1}{2}u_{Ih}^n\Big)^2+\frac{3}{4}(u_{Ih}^n)^2 \ge 0,
			\end{align}
			we have \((u_h^n)^3-(u_{Ih}^n)^3=(u_h^n-u_{Ih}^n)\xi_h^n\). Hence, we can rewrite \(R_3\) as
			\begin{align*}
				R_3
				&=((u_h^n)^3-(u_{Ih}^n)^3,e_h^{u^n})_{\mathcal T_h}
				+((u_{Ih}^n)^3-(u^n)^3,e_h^{u^n})_{\mathcal T_h}
				+(u^n-u_{Ih}^n,e_h^{u^n})_{\mathcal T_h}\\
				&\quad +(e_h^{u^n},e_h^{u^n})_{\mathcal T_h}
				+(\Delta t\,\partial_t^+u_h^n,e_h^{u^n})_{\mathcal T_h}
				=:R_{31}+R_{32}+R_{33}+R_{34}+R_{35}.
			\end{align*}
			For the term \(R_{31}\), a direct computation yields
			\begin{align*}
				R_{31}=((u_h^n-u_{Ih}^n)\xi_h^n,e_h^{u^n})_{\mathcal T_h}
				=-((e_h^{u^n})^2,\xi_h^n)_{\mathcal T_h}\le 0.
			\end{align*}
			Using the Cauchy--Schwarz inequality, the triangle inequality, the uniform bound in \eqref{u_Ih_infty_bound},
			and the approximation property in \eqref{error-u-projection}, we deduce
			\begin{align*}
				R_{32}
				&=((u_{Ih}^n-u^n)\big((u_{Ih}^n)^2+u_{Ih}^nu^n+(u^n)^2\big),e_h^{u^n})_{\mathcal T_h}\\
				&\le C\|u_{Ih}^n-u^n\|_{\mathcal T_h}
				\big(\|u_{Ih}^n\|_{L^\infty(\Omega)}^2+\|u^n\|_{L^\infty(\Omega)}^2\big)\|e_h^{u^n}\|_{\mathcal T_h}
				\le Ch^{k+2}|u^n|_{k+2}.
			\end{align*}
			Applying the Cauchy--Schwarz inequality and \eqref{error-u-projection} to \(R_{33}\), we have
			\begin{align*}
				R_{33}\le \|u^n-u_{Ih}^n\|_{\mathcal T_h}\|e_h^{u^n}\|_{\mathcal T_h}
				\le Ch^{k+2}|u^n|_{k+2}.
			\end{align*}
			Applying the negative-norm estimate \eqref{Cauchy} to \(R_{34}\) yields
			\begin{align*}
				R_{34}
				&\le C\|e_h^{u^n}\|_{-1,h}
				\Big(\|\nabla e_h^{u^n}\|_{\mathcal T_h}
				+\alpha^{1/2}\|h_{\tau}^{-1/2}(\Pi_k^{\partial}e_h^{u^n}-e_h^{\widehat{u}^n})\|_{\partial\mathcal T_h}\Big)\\
				&\le C\|e_h^{u^n}\|_{-1,h}\,\interleave(e_h^{\bm q^n},e_h^{u^n},e_h^{\widehat{u}^n})\interleave.
			\end{align*}
			Finally, by the Cauchy--Schwarz inequality,
			\begin{align*}
				R_{35}\le \|\Delta t\,\partial_t^+u_h^n\|_{\mathcal T_h}\,\|e_h^{u^n}\|_{\mathcal T_h}.
			\end{align*}
			Collecting the above bounds for \(\{R_{3i}\}_{i=1}^5\) and applying Young's inequality, we obtain
			\begin{align*}
				R_3
				\le Ch^{2k+4}|u^n|_{k+2}^2 + C\|e_h^{u^n}\|_{-1,h}^2
				+\delta\epsilon^2\interleave(e_h^{\bm q^n},e_h^{u^n},e_h^{\widehat{u}^n})\interleave^2
				+C\Delta t^2\|\partial_t^+u_h^n\|_{\mathcal T_h}^2.
		\end{align*}}
		For $R_4$,  apply \Cref{error-term-B} to yield
		\begin{align*}
			R_4& \le C(h^{k+2}+\|\Delta t\partial_{t}^+\Pi_{k+1}^ou^n\|_{\mathcal{T}_h}+\interleave(e_h^{\bm{q}^{n-1}}, e_h^{u^{n-1}}, e_h^{\widehat{u}^{n-1}})\interleave)\|e_h^u\|_{-1, h}\\
			&\le
			Ch^{2k+4}+C\|\Delta t\partial_{t}^+\Pi_{k+1}^ou^n\|_{\mathcal{T}_h}^2+ \delta\epsilon^2\interleave(e_h^{\bm{q}^{n-1}}, e_h^{u^{n-1}}, e_h^{\widehat{u}^{n-1}})\interleave^2+
			C\|e_h^{u^n}\|_{-1, h}^2.
		\end{align*}
		Multiplying \eqref{error-eq-negative norm} by $\Delta t$,  applying the above estimates for $\{R_i\}_{i=1}^4$,  and summing from $n=1$ to $m$,  we obtain
		\begin{gather}\label{negative-norm-summation}
			\begin{split}
				&\frac{Pe}{2}\|e_h^{u^m}\|_{-1, h}^2+\frac{Pe{\alpha}\Delta t^2}{2}\sum_{n=1}^{m}\|\partial_{t}^+e_h^{u^n}\|_{-1, h}^2+\epsilon^2(1-3\delta)\Delta t\sum_{n=1}^{m}\interleave (e_h^{\bm{q}^n}, e_h^{u^n}, e_h^{\widehat{u}^n})\interleave^2\\
				&\le C\Delta t\sum_{n=1}^{m}\|\partial_{t}^+u_{Ih}^n-\partial_tu^n\|_{\mathcal{T}_h}^2
				+C\Delta t\sum_{n=1}^{m}\|e_h^{u^{n}}\|_{-1, h}^2+Ch^{2k+4}\\
				&\quad +C\Delta t^3\sum_{n=1}^{m}(\|\partial_{t}^+u_h^n\|_{\mathcal{T}_h}^2+\|\partial_{t}^+\Pi_{k+1}^ou^n\|_{\mathcal{T}_h}^2)+\frac{Pe}{2}\|e_h^{u^0}\|_{-1, h}^2+ {C\|\phi^n-\phi_{Ih}^n\|_{\mathcal{T}_h}^2}.
			\end{split}
		\end{gather}
		Now we turn to bound the right hand terms as follows:
		\begin{align}
			\label{temp-2-06}
			&\Delta t\sum_{n=1}^{m}\|\partial_t^+u_{Ih}^n-\partial_tu^n\|_{\mathcal{T}_h}^2\le 
			C\Delta t\sum_{n=1}^{m}(\|\partial_t^+(u_{Ih}^n-u^n)\|_{\mathcal{T}_h}^2+\|\partial_{t}^+u^n-\partial_{t}u^n\|_{\mathcal{T}_h}^2)\\
			&\qquad\le \frac{1}{\Delta t}\sum_{n=1}^{m}\int_{\Omega}\{[\int_{t_{n-1}}^{t_n}\partial_t(u_{Ih}-u)\; {\rm d} t]^2+[\int_{t_{n-1}}^{t_n}(t-t_{n-1})\partial_{tt}u\; {\rm d} t]^2\}\; {\rm d} {\bm x}\nonumber\\
			&\qquad\le \int_{0}^{T}\|\partial_{t}(u_{Ih}-u)\|_{L^2(\Omega)}^2\;{\rm d}t+\Delta t^2\int_{0}^{T}\|\partial_{tt}u\|_{L^2(\Omega)}^2\;{\rm d}t, \nonumber\\
			\label{temp-2-07}
			&\Delta t^3\sum_{n=1}^{m}\|\partial_{t}^+\Pi_{k+1}^ou^n\|_{\mathcal{T}_h}^2=
			\Delta t\sum_{n=1}^{m}\int_{\Omega}(\int_{t_{n-1}}^{t_n}\partial_t\Pi_{k+1}^ou{\rm d}t)^2{\rm d}\bm{x}\\
			&\le C\Delta t^2\|\partial_{t}\Pi_{k+1}^ou\|_{L^2(0, T;L^2(\Omega))}^2.\nonumber
		\end{align}
		Finally,  choosing $\delta \in (0, 1/4)$ and applying Gronwall’s inequality along with the stability result \eqref{phi_infty-1},  we obtain the desired bound.
	\end{proof}
	\begin{remark}\label{negative-norm-remark}
		As stated in \Cref{approximationprojectionremark},  when the operator $\mathcal{B}$ is employed in the form given by \eqref{def_B_1},  the convergence rate of $\|\phi^n - \phi_{I h}^n\|_{\mathcal{T}_h}$ becomes suboptimal when $k = 0$. Consequently,  the estimate in \eqref{negative-norm-summation} implies that,  under this choice of $\mathcal{B}$,  the negative norm estimate takes the form
		\begin{align*}
			\max_{1\le n\le m}\|e_h^{u^n}\|_{-1, h}^2+C\Delta t\sum_{n=1}^{m}\interleave(e_h^{\bm{q}^n}, e_h^{u^n}, e_h^{\widehat{u}^n})\interleave^2
			\le C(\Delta t^2+h^{2k+2+2\min\{k, 1\}}), 
		\end{align*}
		which results in the loss of optimal convergence rates for the scalar variables in the $ L^2 $ norm.
	\end{remark}
	
	Next,  we analyze the error of the scalar variables in the $ L^2 $ norm.
	\begin{lemma}\label{scalar-variables-error-result}
		Suppose the solution $(\bm p,  \phi,  \bm q,  u)$ attains the maximum regularity required for the best approximation results in \eqref{classical_ine}. For any $ h $,  $ \Delta t $ and $ m=1, 2, \cdots, N $, 
		the following error estimate holds:
		\begin{align}\label{scalar-error}
			\Delta t\sum_{n=1}^{m}\|\phi^n-\phi_h^n\|_{L^2(\Omega)}^2
			+\max_{1\le n\le m}\|u^n-u_h^n\|_{L^2(\Omega)}^2\le C( h^{k+2}+\Delta t)^2.
		\end{align}
	\end{lemma}
	\begin{proof}
		We set $ (\bm{r}_1, w_1, \mu_1)=(-e_h^{\bm{q}^n}, e_h^{u^n}, e_h^{\widehat{u}^n}) $ in \eqref{error_01} and $ (\bm{r}_2, w_2, \mu_2)=(e_h^{\bm{p}^n}, -e_h^{\phi^n}, -e_h^{\widehat{\phi}^n}) $ in \eqref{error_02} to obtain:
		\begin{gather}\label{error-eq-scalar-variable-01}
			\begin{split}
				&(\partial_t^+e_h^{u^n}, e_h^{u^n})_{\mathcal{T}_h}+\frac{1}{Pe}\mathcal{A}(e_h^{\bm{p}^n}, e_h^{\phi^n}, e_h^{\widehat{\phi}^n};-e_h^{\bm{q}^n}, e_h^{u^n}, e_h^{\widehat{u}^n})
				\\
				&=(\partial_t^+u_{Ih}^n-\partial_tu^n, e_h^{u^n})_{\mathcal{T}_h}
				+\mathcal{B}(u_h^{n-1}, \widehat{u}_h^{n-1};e_h^{u^n})-\mathcal{B}(u_{Ih}^n, \widehat{u}_{Ih}^n;e_h^{u^n})
				, 
			\end{split}
		\end{gather}
		and
		\begin{gather}\label{error-eq-scalar-variable-02}
			\begin{split}
				&\epsilon^2\mathcal{A}(e_h^{\bm{q}^n}, e_h^{u^n}, e_h^{\widehat{u}^n};e_h^{\bm{p}^n}, -e_h^{\phi^n}, -e_h^{\widehat{\phi}^n})
				+\|e_h^{\phi^n}\|_{\mathcal{T}_h}^2=(\phi^n-\phi_{Ih}^n, -e_h^{\phi^n})_{\mathcal{T}_h}
				\\
				&\qquad+(F(u_h^n, u_h^{n-1})-f(u^n), -e_h^{\phi^n})_{\mathcal{T}_h}.
			\end{split}
		\end{gather}
		Multiplying \eqref{error-eq-scalar-variable-01} by $ \epsilon^2 $ and    \eqref{error-eq-scalar-variable-02} by $\displaystyle \frac{1}{Pe} $,  then adding the resulting equations and using the definition of $ \mathcal{B} $,  we derive the following:
		\begin{align}
			&\epsilon^2(\partial_t^+e_h^{u^n}, e_h^{u^n})_{\mathcal{T}_h}+\frac{1}{Pe}\|e_h^{\phi^n}\|_{\mathcal{T}_h}^2\nonumber\\
			&=-\frac{1}{Pe}(\phi^n-\phi_{Ih}^n, e_h^{\phi^n})_{\mathcal{T}_h}+
			\epsilon^2(\partial_{t}^+u_{Ih}^n-\partial_tu^n, e_h^{u^n})_{\mathcal{T}_h}
			-\frac{1}{Pe}(F(u_h^{n}, u_h^{n-1})-f(u^n), e_h^{\phi^n})_{\mathcal{T}_h}\\
			&\quad +\epsilon^2[\mathcal{B}(u_h^{n-1}, \widehat{u}_h^{n-1};e_h^{u^n})-\mathcal{B}(u_{Ih}^n, \widehat{u}_{Ih}^n;e_h^{u^n})]
			.\label{eq409}
		\end{align}
		where we use the equality
		\begin{align*}
			\mathcal{A}(e_h^{\bm{p}^n}, e_h^{\phi^n}, e_h^{\widehat{\phi}^n};-e_h^{\bm{q}^n}, e_h^{u^n}, e_h^{\widehat{u}^n})=\mathcal{A}(-e_h^{\bm{q}^n}, -e_h^{u^n}, -e_h^{\widehat{u}^n};e_h^{\bm{p}^n}, -e_h^{\phi^n}, -e_h^{\widehat{\phi}^n}).
		\end{align*}
		
		We now multiply \eqref{eq409} by $ \Delta t $ and sum from $ n=1 $ to $ n=m $,  yielding:
		\begin{align*}
			&\frac{\epsilon^2}{2}\|e_h^{u^m}\|_{\mathcal{T}_h}^2
			-\frac{\epsilon^2}{2}\|e_h^{u^0}\|_{\mathcal{T}_h}^2
			+\frac{\epsilon^2}{2}\sum_{n=1}^{m}\|e_h^{u^n}-e_h^{u^{n-1}}\|_{\mathcal{T}_h}^2+\frac{\Delta t}{Pe}\sum_{n=1}^{m}\|e_h^{\phi^n}\|_{\mathcal{T}_h}^2
			\\
			&=-\frac{\Delta t}{Pe}\sum_{n=1}^{m}(\phi^n-\phi_{Ih}^n, e_h^{\phi^n})_{\mathcal{T}_h}
			+\epsilon^2\Delta t\sum_{n=1}^{m}(\partial_{t}^+u_{Ih}^n-\partial_tu^n, e_h^{u^n})_{\mathcal{T}_h}\\
			&\quad -\frac{\Delta t}{Pe}\sum_{n=1}^{m}(F(u_h^{n}, u_h^{n-1})-f(u^n), e_h^{\phi^n})_{\mathcal{T}_h}
			\\
			&\quad +\epsilon^2\Delta t\sum_{n=1}^{m}[\mathcal{B}(u_h^{n-1}, \widehat{u}_h^{n-1};e_h^{u^n})-\mathcal{B}(u_{Ih}^n, \widehat{u}_{Ih}^n;e_h^{u^n})]\\
			&=:\sum_{i=1}^{4}R_i.
		\end{align*}
		
		Next,   we perform   error estimates for  the terms $ \{R_i\}_{i=1}^4$.   Applying Young's inequality to $ R_1 $ and $ R_2 $,  we obtain
		\begin{align*}
			R_1&\le \frac{\Delta t}{Pe}\sum_{n=1}^{m}\|\phi^n-\phi_{Ih}^n\|_{\mathcal{T}_h}\|e_h^{\phi^n}\|_{\mathcal{T}_h}\\
			&\le \frac{1}{Pe}\sum_{n=1}^{m}(C\Delta t {\|\phi^n-\phi_{Ih}^n\|_{\mathcal{T}_h}^2}+\delta \Delta t\|e_h^{\phi^n}\|_{\mathcal{T}_h}^2)\\
			&\le {C\Delta t\sum_{n=1}^{m}\|\phi^n-\phi_{Ih}^n\|_{\mathcal{T}_h}^2}+\frac{\delta\Delta t}{Pe}\sum_{n=1}^{m}\|e_h^{\phi^n}\|_{\mathcal{T}_h}^2, \\
			R_2&\le \epsilon^2\sum_{n=1}^{m}\Delta t\|\partial_t^+u_{Ih}^n-\partial_tu^n\|_{\mathcal{T}_h}\|e_h^{u^n}\|_{\mathcal{T}_h}\le C\Delta t\sum_{n=1}^{m}\|\partial_t^+u_{Ih}^n-\partial_tu^n\|_{\mathcal{T}_h}^2+C\Delta t\sum_{n=1}^{m}\|e_h^{u^n}\|_{\mathcal{T}_h}^2, 
		\end{align*}
		where $ \delta\in (0, 1) $ is a parameter to be specified later.  For the nonlinear term,  note that 
			\begin{align*}
				F(u_h^n, u_h^{n-1})-f(u^n)&=(u_h^n)^3-(u^n)^3+u^n-u_h^{n-1}\\
				&=(u_h^n-u^n)[(u_h^n)^2+u_h^nu^n+(u^n)^2]+u^n-u_h^n+u_h^n-u_h^{n-1}\\
				&=(u_h^n-u^n)[(u_h^n)^2+u_h^nu^n+(u^n)^2-1]+\Delta t\partial_{t}^+u_h^n. 
			\end{align*}
		
		By \eqref{uh_infty_bounded} in \cref{corEnergy} and $ \|u^n\|_{L^{\infty}(\Omega)}\le C $,  we have
		\begin{align*}
			|F(u_h^n, u_h^{n-1})-f(u^n)|\le C|u_h^n-u^n|+\Delta t|\partial_{t}^+u_h^n|.
		\end{align*}
		Using this bound,  we estimate $ R_3 $ as follows:
		\begin{align*}
			R_3&\le \frac{\Delta t}{Pe}\sum_{n=1}^{m}[C\|u_h^n-u^n\|_{\mathcal{T}_h}+\Delta t\|\partial_{t}^+u_h^n\|_{\mathcal{T}_h}]\|e_h^{\phi^n}\|_{\mathcal{T}_h}\\
			&\le \frac{\Delta t}{Pe}\sum_{n=1}^{m}[C\|u_{Ih}^n-u^n\|_{\mathcal{T}_h}
			+C\|e_h^{u^n}\|_{\mathcal{T}_h}
			+\Delta t\|\partial_{t}^+u_h^n\|_{\mathcal{T}_h}]\|e_h^{\phi^n}\|_{\mathcal{T}_h}\\ 
			&\le Ch^{2k+4}+C\Delta t\sum_{n=1}^{m}\|e_h^{u^n}\|_{\mathcal{T}_h}^2
			+C\Delta t^3\sum_{n=1}^{m}\|\partial_{t}^+u_h^n\|_{\mathcal{T}_h}^2+\frac{\delta\Delta t}{Pe}\sum_{n=1}^{m}\|e_h^{\phi^n}\|_{\mathcal{T}_h}^2.
		\end{align*}
		Applying the negative norm error estimate from \eqref{negative-norm-error-estimate} and \Cref{error-term-B},  we bound $R_4$ as
		\begin{align*}
			R_4&\le C\epsilon^2\Delta t\sum_{n=1}^{m}(h^{k+2}+\|\Delta t\partial_{t}^+\Pi_{k+1}^ou^n\|_{\mathcal{T}_h}+\interleave(e_h^{\bm{q}^{n-1}}, e_h^{u^{n-1}}, e_h^{\widehat{u}^{n-1}})\interleave)\interleave(e_h^{\bm{q}^{n}}, e_h^{u^{n}}, e_h^{\widehat{u}^{n}})\interleave\\
			&\le Ch^{2k+4}+C\Delta t^3\sum_{n=1}^{m}\|\partial_{t}^+\Pi_{k+1}^ou^n\|_{\mathcal{T}_h}^2
			+C\Delta t\sum_{n=1}^{m}\interleave(e_h^{\bm{q}^n}, e_h^{u^n}, e_h^{\widehat{u}^n})\interleave^2\\
			&\le  Ch^{2k+4}+C\Delta t^2+C\Delta t^3\sum_{n=1}^{m}\|\partial_{t}^+\Pi_{k+1}^ou^n\|_{\mathcal{T}_h}^2.
		\end{align*}
		
		Combining the estimates for $ \{R_i\}_{i=1}^4 $ and utilizing the negative norm estimate in \eqref{negative-norm-error-estimate},  we derive
		\begin{gather}\label{scalar-estimate-summation}
			\begin{split}
				&\frac{\epsilon^2}{2}\|e_h^{u^n}\|_{\mathcal{T}_h}^2+\frac{\epsilon^2}{2}\sum_{n=1}^{m}\|e_h^{u^n}-e_h^{u^{n-1}}\|_{\mathcal{T}_h}^2+\frac{(1-2\delta)\Delta t}{Pe}\sum_{n=1}^{m}\|e_h^{\phi^n}\|_{\mathcal{T}_h}^2\\
				&\qquad\le Ch^{2k+4}+C\Delta t\sum_{n=1}^{m}\|\partial_{t}^+u_{Ih}^n-\partial_{t}u^n\|_{\mathcal{T}_h}^2+C\Delta t\sum_{n=1}^{m}\|e_h^{u^n}\|_{\mathcal{T}_h}^2\\
				&\qquad\quad
				+C\Delta t^3\sum_{n=1}^{m}(\|\partial_{t}^+u_h^n\|_{\mathcal{T}_h}^2+\|\partial_{t}^+\Pi_{k+1}^ou^n\|_{\mathcal{T}_h}^2)
				+\frac{\epsilon^2}{2}\|e_h^{u^0}\|_{\mathcal{T}_h}^2+ {C\Delta t\sum_{n=1}^{m}\|\phi^n-\phi_{Ih}^n\|_{\mathcal{T}_h}^2.}\\
				&\qquad\le Ch^{2k+4}+C\Delta t\sum_{n=1}^{m}\|\partial_{t}^+u_{Ih}^n-\partial_{t}u^n\|_{\mathcal{T}_h}^2+C\Delta t\sum_{n=1}^{m}\|e_h^{u^n}\|_{\mathcal{T}_h}^2\\
				&\qquad\quad {+C\Delta t^3\sum_{n=1}^{m}(\|\partial_{t}^+u_h^n\|_{\mathcal{T}_h}^2+\|\partial_{t}^+\Pi_{k+1}^ou^n\|_{\mathcal{T}_h}^2)}+\frac{\epsilon^2}{2}\|e_h^{u^0}\|_{\mathcal{T}_h}^2, 
			\end{split}
		\end{gather}
		here we use the HDG elliptic projection approximation.
		
		By the HDG Sobolev embedding inequality \eqref{sobolev-02},  the bound \eqref{es_u},  and the negative norm error estimate \eqref{negative-norm-error-estimate},  we obtain
		\begin{align*}
			\Delta t\sum_{n=1}^{m}\|e_h^{u^n}\|_{\mathcal{T}_h}^2&\le C\Delta t\sum_{n=1}^{m}(\|e_h^{\bm{q}^n}\|_{\mathcal{T}_h}^2+\alpha\|h_{\tau}^{-1/2}(\Pi_k^{\partial}e_h^{u^n}-e_h^{\widehat{u}^n})\|_{\partial\mathcal{T}_h}^2)\le C(h^{2k+4}+\Delta t^2).
		\end{align*}
		Applying the estimate in \eqref{temp-2-06},  \eqref{temp-2-07},  we obtain
		\begin{align*}
			&C\Delta t\sum_{n=1}^{m}\|\partial_t^+u_{Ih}^n-\partial_{t}u^n\|_{\mathcal{T}_h}^2
			+C\Delta t^3\sum_{n=1}^{m}(\|\partial_{t}^+u_h^n\|_{\mathcal{T}_h}^2+\|\partial_{t}^+\Pi_{k+1}^ou^n\|_{\mathcal{T}_h}^2)\\
			&\qquad\le C\int_{0}^{T}\|\partial_t(u_{Ih}-u)\|_{L^2(\Omega)}^2\;{\rm d}t
			+C\Delta t^2\int_{0}^{T}\|\partial_{tt}u\|_{L^2(\Omega)}^2\;{\rm d}t\\
			&\qquad\quad+C\Delta t^3\sum_{n=1}^{m}\|\partial_{t}^+u_h^n\|_{\mathcal{T}_h}^2
			+C\Delta t^2\|\partial_t\Pi_{k+1}^ou\|_{L^2(0, T;L^2(\Omega))}^2.
		\end{align*}
		Substituting these estimates into \eqref{scalar-estimate-summation},  choosing $ \delta\in(0, 1/4) $,     applying stability result in \eqref{phi_infty-1}   and  triangle inequality,  we establish the optimal error approximation of  $ u $ and $ \phi $ in the $ L^2 $ norm.
	\end{proof}
	
	\begin{remark}\label{scalar-variables-error-resultremark}
		Similar to \Cref{negative-norm-remark},  the optimal convergence order is lost for $k=0$ when the operator $\mathcal B$ is used in the form of \eqref{def_B_1}. This loss arises from two factors: one due to the negative norm estimate in \Cref{lemma-negative-norm},  and the other from the HDG elliptic projection \eqref{error-phi-projection}.
	\end{remark}
	
	Finally,  we analyze the error estimates for the flux variables. The following lemma establishes the corresponding bounds.
	\begin{lemma}\label{lem-err-flux} 
		Suppose the solution $(\bm p,  \phi,  \bm q,  u)$ attains the maximum regularity required for the best approximation results in \eqref{classical_ine}.
		For any $h,  \Delta t$ and $m=1\ldots N$,   the flux variables satisfy the following error bounds
		\begin{align} \label{Final-flux}
			\max_{1\le n\le m} \|\bm q^n -\bm q_h^n\|_{L^2(\Omega)}^2+\Delta t \sum_{n=1}^m \|\bm p^n -\bm p_h^n\|_{L^2(\Omega)}^2  \le  C (h^{k+1}+\Delta t)^2.
		\end{align}   
	\end{lemma}
	\begin{proof}
		Substituting  $ (\bm{r}_1, w_1, \mu_1)= (e_h^{\bm{p}^n}, e_h^{\phi^n}, e_h^{\widehat{\phi}^n})$ into \eqref{error_01},  we obtain
		\begin{gather}\label{err-temp1}
			\begin{split}
				&(\partial_{t}^+e_h^{u^n}, e_h^{\phi^n})_{\mathcal{T}_h}+\frac{1}{Pe}\interleave(e_h^{\bm{p}^n}, e_h^{\phi^n}, e_h^{\widehat{\phi}^n}) \interleave^2
				=(\partial_{t}^+u_{Ih}^n-\partial_tu^n, e_h^{\phi^n})_{\mathcal{T}_h}\\
				&\quad+\mathcal{B}(u_h^{n-1}, \widehat{u}_h^{n-1};e_h^{\phi^n})
				-\mathcal{B}(u_{Ih}^n, \widehat{u}_{Ih}^n;e_h^{\phi^n})
				.
			\end{split}
		\end{gather}
		Next,  applying  $ \partial_{t}^+ $ to \eqref{error_02} and selecting $ (\bm{r}_2, w_2, \mu_2)= (e_h^{\bm{q}^n}, 0, 0)$,  we leverage the definition of $ \mathcal{A} $ to derive
		\begin{align}\label{err-temp2}
			\epsilon^2(\partial_{t}^+e_h^{\bm q^n}, e_h^{\bm{q}^n})_{\mathcal{T}_h}-\epsilon^2(\partial_t^+e_h^{u^n}, \nabla\cdot e_h^{\bm q^n})_{\mathcal{T}_h}+\epsilon^2\langle \partial_t^+e_h^{\widehat{u}^n}, e_h^{\bm{q}^n}\cdot \bm{n} \rangle_{\partial \mathcal{T}_h}=0.
		\end{align}
		Taking  $ (\bm{r}_2, w_2, \mu_2)= ({\bm 0}, \partial_{t}^+e_h^{u^n}, \partial_{t}^+e_h^{\widehat{u}^n})$ in \eqref{error_02} leads to
		\begin{gather}\label{err-temp3}
			\begin{split}
				&\epsilon^2\mathcal{A}(e_h^{\bm{q}^n}, e_h^{u^n}, e_h^{\widehat{u}^n};{\bm 0}, \partial_t^+e_h^{u^n}, \partial_{t}^+e_h^{\widehat{u}^n})
				-(e_h^{\phi^n}, \partial_{t}^+e_h^{u^n})_{\mathcal{T}_h}
				=(\phi^n-\phi_{Ih}^n, \partial_{t}^+e_h^{u^n})_{\mathcal{T}_h}\\
				&+(F(u_h^n, u_h^{n-1})-f(u^n), \partial_{t}^+e_h^{u^n})_{\mathcal{T}_h}.
			\end{split}
		\end{gather}
		Summing \eqref{err-temp1},  \eqref{err-temp2},  and \eqref{err-temp3},  we obtain
		\begin{gather}\label{error-eq-flux-temp1}
			\begin{split}
				&\frac{1}{Pe}\interleave(e_h^{\bm{p}^n}, e_h^{\phi^n}, e_h^{\widehat{\phi}^n}) \interleave^2
				+\epsilon^2(\partial_{t}^+e_h^{\bm q^n}, e_h^{\bm{q}^n})_{\mathcal{T}_h}\\
				&\quad
				+\alpha\epsilon^2\langle h_{\tau}^{-1}(\Pi_k^{\partial}e_h^{u^n}-e_h^{\widehat{u}^n}), \partial_{t}^+(\Pi_k^{\partial}e_h^{u^n}-e_h^{\widehat u ^n})\rangle_{\partial \mathcal{T}_h}\\
				&=(\partial_{t}^+u_{Ih}^n-\partial_tu^n, e_h^{\phi^n})_{\mathcal{T}_h}
				+\mathcal{B}(u_h^{n-1}, \widehat{u}_h^{n-1};e_h^{\phi^n})
				-\mathcal{B}(u_{Ih}^n, \widehat{u}_{Ih}^n;e_h^{\phi^n})
				\\
				&\quad +(\phi^n-\phi_{Ih}^n, \partial_{t}^+e_h^{u^n})_{\mathcal{T}_h}
				+(F(u_h^n, u_h^{n-1})-f(u^n), \partial_{t}^+e_h^{u^n})_{\mathcal{T}_h}.
			\end{split}
		\end{gather}
		Multiplying \eqref{error-eq-flux-temp1} by $ \Delta t $ and summing over $ n $ from $ 1 $ to $ m $,  we obtain
		\begin{gather}\label{error-eq-flux-summation}
			\begin{split}
				&\frac{\Delta t}{Pe}\sum_{n=1}^{m}\interleave(e_h^{\bm{p}^n}, e_h^{\phi^n}, e_h^{\widehat{\phi}^n}) \interleave^2
				+\frac{\epsilon^2}{2}\|e_h^{\bm{q}^m}\|_{\mathcal{T}_h}^2
				-\frac{\epsilon^2}{2}\|e_h^{\bm{q}^0}\|_{\mathcal{T}_h}^2
				+\frac{\epsilon^2}{2}\sum_{n=1}^{m}\|e_h^{\bm{q}^n}-e_h^{\bm{q}^{n-1}}\|_{\mathcal{T}_h}^2\\
				&+\frac{\alpha\epsilon^2}{2}\| h_{\tau}^{-1/2}(\Pi_k^{\partial}e_h^{u^m}-e_h^{\widehat{u}^m})\|_{\partial\mathcal{T}_h}^2
				-\frac{\alpha\epsilon^2}{2}\| h_{\tau}^{-1/2}(\Pi_k^{\partial}e_h^{u^0}-e_h^{\widehat{u}^0})\|_{\partial\mathcal{T}_h}^2\\
				&+\frac{\alpha\epsilon^2}{2}\sum_{n=1}^{m}\| h_{\tau}^{-1/2}\Delta t\partial_{t}^+(\Pi_k^{\partial}e_h^{u^n}-e_h^{\widehat{u}^n})\|_{\partial\mathcal{T}_h}^2\\
				&=\Delta t\sum_{n=1}^{m}(\partial_{t}^+u_{Ih}^n-\partial_tu^n, e_h^{\phi^n})_{\mathcal{T}_h}
				+\Delta t\sum_{n=1}^{m}(\phi^n-\phi_{Ih}^n, \partial_{t}^+e_h^{u^n})_{\mathcal{T}_h}
				\\
				&\quad 
				-\Delta t\sum_{n=1}^{m}(\mathcal{B}(u_h^{n-1}, \widehat{u}_h^{n-1};e_h^{\phi^n})
				-\mathcal{B}(u_{Ih}^n, \widehat{u}_{Ih}^n;e_h^{\phi^n}))\\
				&\quad+\Delta t\sum_{n=1}^{m}(F(u_h^n, u_h^{n-1})-F(u_{Ih}^n, u_{Ih}^{n-1}), \partial_{t}^+e_h^{u^n})_{\mathcal{T}_h}\\
				&\quad+\Delta t\sum_{n=1}^{m}(F(u_{Ih}^n, u_{Ih}^{n-1})-f(u^n), \partial_{t}^+e_h^{u^n})_{\mathcal{T}_h}
				\\
				&=:\sum_{i=1}^{5}R_i.
			\end{split}
		\end{gather}
		
		Next, we estimate the terms $\{R_i\}_{i=1}^5$. Using the HDG embedding inequality \eqref{sobolev-02} and \eqref{es_u}, and noting that
		
			\[
			\mathcal{A}(e_h^{\bm p^n}, e_h^{\phi^n}, e_h^{\widehat{\phi}^n};\bm r_h,0,0)=0\qquad \forall\,\bm r_h\in \bm V_h,
			\]
		
		which follows from \eqref{error_01} by taking $(w_1,\mu_1)=(0,0)$, we obtain
		\begin{align*}
			R_1&\le C\Delta t\sum_{n=1}^{m}\| \partial_{t}^+ u_{Ih}^n-\partial_{t}u^n\|_{\mathcal{T}_h}\|e_h^{\phi^n}\|_{\mathcal{T}_h}\\
			& \le C\Delta t\sum_{n=1}^{m}\|\partial_{t}^+ u_{Ih}^n-\partial_{t}u^n\|_{\mathcal{T}_h}(\|\nabla e_h^{\phi^n}\|_{\mathcal{T}_h}+\alpha^{1/2}\|h_{\tau}^{-1/2}(\Pi_k^{\partial}e_h^{\phi^n}-e_h^{\widehat{\phi^n}})\|_{\partial\mathcal{T}_h})\\
			& \le C\Delta t\sum_{n=1}^{m}\|\partial_{t}^+ u_{Ih}^n-\partial_{t}u^n\|_{\mathcal{T}_h}(\| e_h^{\bm{p}^n}\|_{\mathcal{T}_h}+\alpha^{1/2}\|h_{\tau}^{-1/2}(\Pi_k^{\partial}e_h^{\phi^n}-e_h^{\widehat{\phi^n}})\|_{\partial\mathcal{T}_h})\\
			&\le C\Delta t\sum_{n=1}^{m}\|\partial_{t}^+ u_{Ih}^n-\partial_{t}u^n\|_{\mathcal{T}_h}^2
			+\frac{\delta\Delta t}{Pe}\sum_{n=1}^{m}\interleave(e_h^{\bm{p}^n}, e_h^{\phi^n}, e_h^{\widehat{\phi}^n}) \interleave^2, 
		\end{align*}
		where $ \delta $ is a positive constant to be determined later. Leveraging the approximation property of $ \mathcal{B} $ in \Cref{error-term-B} and  negative norm estimate in \eqref{negative-norm-error-estimate},  it follows 
		\begin{align*}
			R_3&\le C\Delta t\sum_{n=1}^{m}(h^{k+2}+\|\Delta t\partial_{t}^+\Pi_{k+1}^ou^n\|_{\mathcal{T}_h}+\interleave(e_h^{\bm{q}^{n-1}}, e_h^{u^{n-1}}, e_h^{\widehat{u}^{n-1}})\interleave)
			\interleave(e_h^{\bm{p}^n}, e_h^{\phi^n}, e_h^{\widehat{\phi}^n})\interleave\\
			&\le Ch^{2k+4}+C\Delta t^3\sum_{n=1}^{m}\|\partial_{t}^+\Pi_{k+1}^ou^n\|_{\mathcal{T}_h}^2
			+C\Delta t\sum_{n=1}^{m}\interleave(e_h^{\bm{q}^{n-1}}, e_h^{u^{n-1}}, e_h^{\widehat{u}^{n-1}})\interleave^2\\
			&\quad+\frac{\delta\Delta t}{Pe}\sum_{n=1}^{m}\interleave(e_h^{\bm{p}^n}, e_h^{\phi^n}, e_h^{\widehat{\phi}^n})\interleave^2\\
			&\le Ch^{2k+4}+C\Delta t^2+C\Delta t^3\sum_{n=1}^{m}\|\partial_{t}^+\Pi_{k+1}^ou^n\|_{\mathcal{T}_h}^2
			+\frac{\delta\Delta t}{Pe}\sum_{n=1}^{m}\interleave(e_h^{\bm{p}^n}, e_h^{\phi^n}, e_h^{\widehat{\phi}^n})\interleave^2.
		\end{align*}
		
		The estimates for $ R_2 $ and $ R_4 $ follow similarly,  utilizing error equation \eqref{error_01}.
		Recalling the definition $\varepsilon_h^{\phi^n}:=\varepsilon_h^{\phi}=\Pi_{k+1}^o\phi^n-\phi_{Ih}^n$ (see \eqref{error-notation-varepsilon}), we take $(\bm r_1,w_1,\mu_1)=(\bm 0,\varepsilon_h^{\phi^n},0)$ in \eqref{error_01} to obtain
		\begin{align*}
			R_2&=\Delta t\sum_{n=1}^{m}(\Pi_{k+1}^{o}\phi^n-\phi_{Ih}^n, \partial_{t}^+e_h^{u^n})_{\mathcal{T}_h}=
			\Delta t\sum_{n=1}^{m}(\varepsilon_h^{\phi^n}, \partial_{t}^+e_h^{u^n})_{\mathcal{T}_h}\\
			&=-\frac{\Delta t}{Pe}\sum_{n=1}^{m}\mathcal{A}(e_h^{\bm p^n}, e_h^{\phi^n}, e_h^{\widehat{\phi}^n};\bm{0}, \varepsilon_h^{\phi^n}, 0)
			+\Delta t\sum_{n=1}^{m}(\partial_{t}^+u_{Ih}^n-\partial_{t}u^n, \varepsilon_h^{\phi^n})_{\mathcal{T}_h}\\
			&\quad +\Delta t\sum_{n=1}^{m}(\mathcal{B}(u_h^{n-1}, \widehat{u}_h^{n-1};\varepsilon_h^{\phi^n})
			-\mathcal{B}(u_{Ih}^n, \widehat{u}_{Ih}^n;\varepsilon_h^{\phi^n}))\\
			&=:R_{21}+R_{22}+R_{23}.
		\end{align*}
		
		Using the continuity property of $\mathcal{A}$ in \eqref{A-continuous} and the approximation estimate for $\varepsilon_h^{\phi^n}$ in \eqref{appendix-vare-h-phi}, we obtain
		\begin{align*}
			R_{21}&=-\frac{\Delta t}{Pe}\sum_{n=1}^{m}\mathcal{A}(e_h^{\bm{p}^n}, e_h^{\phi^n}, e_h^{\widehat{\phi^n}};\bm{0}, \varepsilon_h^{\phi^n}, 0)\\
			&\le \frac{\Delta t}{Pe}\sum_{n=1}^{m}(\|e_h^{\bm{p}^n}\|_{\mathcal{T}_h}+\|\nabla e_h^{\phi^n}\|_{\mathcal{T}_h}+\alpha^{1/2}\|h_{\tau}^{-1/2}(\Pi_k^{\partial}e_h^{\phi^n}-e_h^{\widehat{\phi^n}})\|_{\partial\mathcal{T}_h})\\
			&\quad \times(\|\nabla \varepsilon_h^{\phi^n}\|_{\mathcal{T}_h}+\alpha^{1/2}\|h_{\tau}^{-1/2}\Pi_k^{\partial}\varepsilon_h^{\phi^n}\|_{\partial\mathcal{T}_h})\\
			&\le \frac{\Delta t}{Pe}\sum_{n=1}^{m}(\|e_h^{\bm{p}^n}\|_{\mathcal{T}_h}+\|\nabla e_h^{\phi^n}\|_{\mathcal{T}_h}+\alpha^{1/2}\|h_{\tau}^{-1/2}(\Pi_k^{\partial}e_h^{\phi^n}-e_h^{\widehat{\phi^n}})\|_{\partial\mathcal{T}_h})\\
			&\qquad\times {h^{k+1}(|\phi^n|_{k+2}+|u^n|_{k+1}+|u^n|_{k+2})}
			\\
			&\le \frac{\delta\Delta t}{Pe}\sum_{n=1}^{m}\interleave(e_h^{\bm{p}^n}, e_h^{\phi^n}, e_h^{\widehat{\phi}^n}) \interleave^2+Ch^{2k+2}.
		\end{align*}
		Using Young's inequality  and inequality \eqref{appendix-vare-h-phi} for $ R_{22} $,  we obtain
		\begin{align*}
			R_{22}&\le \Delta t\sum_{n=1}^{m}\|\partial_{t}^+u_{Ih}^n-\partial_{t}u^n\|_{\mathcal{T}_h}
			\|\varepsilon_h^{\phi^n}\|_{\mathcal{T}_h}\\
			&\le C \Delta t\sum_{n=1}^{m}\|\partial_{t}^+u_{Ih}^n-\partial_{t}u^n\|_{\mathcal{T}_h}
			{h^{k+2}(|\phi^n|_{k+2}+|u^n|_{k+1}+|u^n|_{k+2})}\\
			&\le C\Delta t\sum_{n=1}^{m}\|\partial_{t}^+u_{Ih}^n-\partial_{t}u^n\|_{\mathcal{T}_h}^2
			+Ch^{2k+4}.
		\end{align*}
		Applying the projection error estimate for $ \varepsilon_h^{\phi^n} $,   the negative norm estimate \eqref{negative-norm-error-estimate} and approximation property of $ \mathcal{B} $ in \eqref{B-error-bounded},  we obtain
		\begin{align*}
			R_{23}&\le C\Delta t\sum_{n=1}^{m}(h^{k+2}+\|\Delta t\partial_{t}^+\Pi_{k+1}^ou^n\|_{\mathcal{T}_h}
			+\interleave(e_h^{\bm{q}^{n-1}}, e_h^{u^{n-1}}, e_h^{\widehat{u}^{n-1}})\interleave)\\
			&\quad\times(\|\nabla\varepsilon_h^{\phi^n}\|_{\mathcal{T}_h}+\|h_{\tau}^{-1/2}\Pi_k^{\partial}\varepsilon_h^{\phi^n}\|_{\partial\mathcal{T}_h})\\
			&\le Ch^{2k+4}+C\Delta t^3\sum_{n=1}^{m}\|\partial_t^+\Pi_{k+1}^ou^n\|_{\mathcal{T}_h}+C\Delta t\sum_{n=1}^{m}\interleave(e_h^{\bm{q}^{n-1}}, e_h^{u^{n-1}}, e_h^{\widehat{u}^{n-1}})\interleave^2+Ch^{2k+2}\\
			&\le Ch^{2k+2}+C\Delta t^2+C\Delta t^3\sum_{n=1}^{m}\|\partial_t^+\Pi_{k+1}^ou^n\|_{\mathcal{T}_h}^2.
		\end{align*}
		Combining the above estimates for $ R_{21} $ to $ R_{23} $,  we obtain
		\begin{align*}
			R_2&\le \frac{\delta\Delta t}{Pe}\sum_{n=1}^{m}\interleave (e_h^{\bm{p}^n}, e_h^{\phi^n}, e_h^{\widehat{\phi}^^n})\interleave^2
			+Ch^{2k+2}+C\Delta t\sum_{n=1}^{m}\|\partial_{t}^+u_{Ih}^n-\partial_{t}u^n\|_{\mathcal{T}_h}^2\\
			&\quad+C\Delta t^2
			+C\Delta t^3\sum_{n=1}^{m}\|\partial_t^+\Pi_{k+1}^ou^n\|_{\mathcal{T}_h}^2.
		\end{align*}
		
		{
			Next, we estimate $R_4$. By direct computation, we have 
			\begin{align*}
				R_{4}
				&=\Delta t\sum_{n=1}^{m}\Bigl((u_h^n)^3-(u_{Ih}^n)^3+u_{Ih}^{n-1}-u_h^{n-1},\,\partial_{t}^+e_h^{u^n}\Bigr)_{\mathcal{T}_h}\\
				&=\Delta t\sum_{n=1}^{m}\Bigl((u_h^n-u_{Ih}^n)\bigl((u_h^n)^2+u_h^n u_{Ih}^n+(u_{Ih}^n)^2\bigr),\,\partial_{t}^+e_h^{u^n}\Bigr)_{\mathcal{T}_h}
				\\
				&\quad +\Delta t\sum_{n=1}^{m}\Bigl(u_{Ih}^{n-1}-u_h^{n-1},\,\partial_{t}^+e_h^{u^n}\Bigr)_{\mathcal{T}_h}\\
				&=:R_{41}+R_{42}.
			\end{align*}			
			Recalling the definition of $\xi_h^n$ in \eqref{xi_hn} and using the uniform bounds \eqref{u_Ih_infty_bound} and
			\eqref{uh_infty_bounded}, together with the elementary identity
			$a(a-b)=\tfrac12\bigl(a^2-b^2+(a-b)^2\bigr)$, we obtain
			\begin{align*}
				|R_{41}|
				&\le C\Delta t\max_{1\le n\le m}\|\xi_h^n\|_{L^{\infty}(\Omega)}
				\left|\sum_{n=1}^{m}\bigl(u_h^n-u_{Ih}^n,\,\partial_{t}^+e_h^{u^n}\bigr)_{\mathcal{T}_h}\right|\\
				&\le C\Delta t\left|\sum_{n=1}^{m}\bigl(e_h^{u^n},\,\partial_{t}^+e_h^{u^n}\bigr)_{\mathcal{T}_h}\right|\\
				&=\frac{C}{2}\left|\sum_{n=1}^{m}\Bigl(\|e_h^{u^{n}}\|_{\mathcal{T}_h}^2-\|e_h^{u^{n-1}}\|_{\mathcal{T}_h}^2
				+\|e_h^{u^{n}}-e_h^{u^{n-1}}\|_{\mathcal{T}_h}^2\Bigr)\right|\\
				&=\frac{C}{2}\left|\|e_h^{u^{m}}\|_{\mathcal{T}_h}^2-\|e_h^{u^{0}}\|_{\mathcal{T}_h}^2
				+\sum_{n=1}^{m}\|e_h^{u^{n}}-e_h^{u^{n-1}}\|_{\mathcal{T}_h}^2\right|\\
				&\le \frac{C}{2}\Bigl(\|e_h^{u^{m}}\|_{\mathcal{T}_h}^2+\|e_h^{u^{0}}\|_{\mathcal{T}_h}^2
				+\sum_{n=1}^{m}\|e_h^{u^{n}}-e_h^{u^{n-1}}\|_{\mathcal{T}_h}^2\Bigr)\\
				&\le C\bigl(h^{2k+4}+\Delta t^2\bigr),
			\end{align*}
			where the last inequality follows from \Cref{scalar-variables-error-result}. Likewise,   we have
			\begin{align*}
				|R_{42}|
				&=\Delta t\left|\sum_{n=1}^{m}\bigl(u_{Ih}^{n-1}-u_h^{n-1},\,\partial_{t}^+e_h^{u^n}\bigr)_{\mathcal{T}_h}\right|
				=\Delta t\left|\sum_{n=1}^{m}\bigl(e_h^{u^{n-1}},\,\partial_{t}^+e_h^{u^n}\bigr)_{\mathcal{T}_h}\right|\\
				&=\frac{1}{2}\left|\sum_{n=1}^{m}\Bigl(\|e_h^{u^n}\|_{\mathcal{T}_h}^2-\|e_h^{u^{n-1}}\|_{\mathcal{T}_h}^2
				-\|e_h^{u^n}-e_h^{u^{n-1}}\|_{\mathcal{T}_h}^2\Bigr)\right|\\
				&=\frac{1}{2}\left|\|e_h^{u^{m}}\|_{\mathcal{T}_h}^2-\|e_h^{u^{0}}\|_{\mathcal{T}_h}^2
				-\sum_{n=1}^{m}\|e_h^{u^{n}}-e_h^{u^{n-1}}\|_{\mathcal{T}_h}^2\right|\\
				&\le C\bigl(h^{2k+4}+\Delta t^2\bigr).
			\end{align*}
			
			Combining the bounds for $R_{41}$ and $R_{42}$ yields
			\[
			|R_4|\le C\bigl(h^{2k+4}+\Delta t^2\bigr).
			\]}
		
		To simply notation,  for $ R_5 $,  let
		\begin{align*}
			Z^n=F(u_{Ih}^n, u_{Ih}^{n-1})-f(u^n), 
		\end{align*}
		and applying the error equation \eqref{error_01},  we rewrite $ R_5 $ as:
		\begin{align*}
			R_5&=\Delta t\sum_{n=1}^{m}(\Pi_{k+1}^oZ^n, \partial_{t}^+e_h^{u^n})\\
			&=-\frac{\Delta t}{Pe}\sum_{n=1}^{m}\mathcal{A}(e_h^{\bm{p}^n}, e_h^{\phi^n}, e_h^{\widehat{\phi}^n};\bm{0}, \Pi_{k+1}^oZ^n, \Pi_k^{\partial}Z^n)
			+\Delta t\sum_{n=1}^{m}(\partial_t^+u_{Ih}^n-\partial_{t}u^n, \Pi_{k+1}^oZ^n)_{\mathcal{T}_h}\\
			&\quad +\Delta t\sum_{n=1}^{m}(\mathcal{B}(u_h^{n-1}, \widehat{u}_h^{n-1};\Pi_{k+1}^oZ^n)
			-\mathcal{B}(u_{Ih}^n, \widehat{u}_{Ih}^n;\Pi_{k+1}^oZ^n))
			\\
			&=:R_{51}+R_{52}+R_{53}.
		\end{align*}
		Using the same techniques applied to estimate $ R_{51} $ and $ R_{52} $,  we derive
		\begin{align*}
			R_{51}&\le \frac{\Delta t}{Pe}\sum_{n=1}^{m}\interleave (e_h^{\bm{p}^n}, e_h^{\phi^n}, e_h^{\widehat{\phi^n}})\interleave(\|\nabla\Pi_{k+1}^oZ^n\|_{\mathcal{T}_h}+\alpha^{1/2}\|h_{\tau}^{-1/2}(\Pi_k^{\partial}\Pi_{k+1}^oZ^n-\Pi_k^{\partial}Z^n)\|_{\partial\mathcal{T}_h})\\
			&\le \frac{\delta\Delta t}{Pe}\sum_{n=1}^{m}\interleave(e_h^{\bm{p}^n}, e_h^{\phi^n}, e_h^{\widehat{\phi^n}}) \interleave^2+C\Delta t\sum_{n=1}^{m}(\|\nabla\Pi_{k+1}^oZ^n\|_{\mathcal{T}_h}^2+\|\nabla Z^n\|_{\mathcal{T}_h}^2)
			, \\
			R_{52}&\le C\Delta t\sum_{n=1}^{m}\|\partial_{t}^+u_{Ih}^n-\partial_{t}u^n\|_{\mathcal{T}_h}
			\|\Pi_{k+1}^oZ^n\|_{\mathcal{T}_h}\\
			&\le C\Delta t\sum_{n=1}^{m}\|\partial_{t}^+u_{Ih}^n-\partial_{t}u^n\|_{\mathcal{T}_h}^2
			+C\Delta t\sum_{n=1}^{m}	\|\Pi_{k+1}^oZ^n\|_{\mathcal{T}_h}^2.
		\end{align*}
		As for the term $ R_{53} $,  by using the inequality \eqref{B-error-bounded} and  estimate \eqref{negative-norm-error-estimate},  one has
			\begin{align*}
				R_{53}&\le C\Delta t\sum_{n=1}^{m}(h^{k+2}+\|\Delta t\partial_{t}^+\Pi_{k+1}^ou^n\|_{\mathcal{T}_h}+\interleave(e_h^{\bm{q}^{n-1}}, e_h^{u^{n-1}}, e_h^{\widehat{u}^{n-1}})\interleave)\\
				&\quad \times (\|\nabla \Pi_{k+1}^oZ^n\|_{\mathcal{T}_h}+\|h_{\tau}^{-1/2}(\Pi_{k+1}^oZ^n-\Pi_{k}^{\partial }Z^n)\|_{\partial\mathcal{T}_h})\\
				&\le C\Delta t\sum_{n=1}^{m}(h^{k+2}
				+\|\Delta t\partial_{t}^+\Pi_{k+1}^ou^n\|_{\mathcal{T}_h}
				+\interleave(e_h^{\bm{q}^{n-1}}, e_h^{u^{n-1}}, e_h^{\widehat{u}^{n-1}})\interleave) \times (\|\nabla \Pi_{k+1}^oZ^n\|_{\mathcal{T}_h}+\|\nabla Z^n\|_{\mathcal{T}_h})\\
				&\le Ch^{2k+4}+C\Delta t^2+C\Delta t^3\sum_{n=1}^{m}\|\partial_{t}^+\Pi_{k+1}^ou^n\|_{\mathcal{T}_h}^2
				+C\Delta t^2(\|\nabla\Pi_{k+1}^oZ^n\|_{\mathcal{T}_h}^2+\|\nabla Z^n\|_{\mathcal{T}_h}^2).
			\end{align*}	
		
		Thus,  we obtain the bound
		\begin{align*}
			R_5&\le \frac{\delta\Delta t}{Pe}\sum_{n=1}^{m}\interleave (e_h^{\bm{p}^n}, e_h^{\phi^n}, e_h^{\widehat{\phi}^n})\interleave^2
			+C\Delta t\sum_{n=1}^{m}(\|\Pi_{k+1}^oZ^n\|_{\mathcal{T}_h}^2
				+\|\nabla\Pi_{k+1}^oZ^n\|_{\mathcal{T}_h}^2+\|\nabla Z^n\|_{\mathcal{T}_h}^2) \\
			&\quad+C\Delta t^2+C\Delta t\sum_{n=1}^{m}\|\partial_{t}^+u_{Ih}^n-\partial_{t}u^n\|_{\mathcal{T}_h}^2
			+Ch^{2k+4}+C\Delta t^3\sum_{n=1}^{m}\|\partial_t^+\Pi_{k+1}^ou^n\|_{\mathcal{T}_h}^2.
		\end{align*}
		Combining all estimates for $ R_1 $ to $ R_5 $,  we obtain the following  error bound:
		\begin{gather}
			\begin{split}\label{error-eq-flux-summation-substitute}
				&\frac{\Delta t(1-4\delta)}{Pe}\sum_{n=1}^{m}\interleave(e_h^{\bm{p}^n}, e_h^{\phi^n}, e_h^{\widehat{\phi}^n}) \interleave^2
				+\frac{\epsilon^2}{2}\|e_h^{\bm{q}^m}\|_{\mathcal{T}_h}^2
				+\frac{\epsilon^2}{2}\sum_{n=1}^{m}\|e_h^{\bm{q}^n}-e_h^{\bm{q}^{n-1}}\|_{\mathcal{T}_h}^2\\
				&\quad+\frac{\alpha\epsilon^2}{2}\| h_{\tau}^{-1/2}(\Pi_k^{\partial}e_h^{u^m}-e_h^{\widehat{u}^m})\|_{\partial\mathcal{T}_h}^2
				+\frac{\alpha\epsilon^2}{2}\sum_{n=1}^{m}\| h_{\tau}^{-1/2}\Delta t\partial_{t}^+(\Pi_k^{\partial}e_h^{u^n}-e_h^{\widehat{u}^n})\|_{\partial\mathcal{T}_h}^2\\
				&\le \frac{\epsilon^2}{2}\|e_h^{\bm{q}^0}\|_{\mathcal{T}_h}^2
				+\frac{\alpha\epsilon^2}{2}\| h_{\tau}^{-1/2}(\Pi_k^{\partial}e_h^{u^0}-e_h^{\widehat{u}^0})\|_{\partial\mathcal{T}_h}^2
				+C\Delta t\sum_{n=1}^{m}\|\partial_{t}^+u_{Ih}^n-\partial_{t}u^n\|_{\mathcal{T}_h}^2+Ch^{2k+2}\\
				&\quad+C\Delta t^2+
				+C\Delta t\sum_{n=1}^{m}(\|\Pi_{k+1}^oZ^n\|_{\mathcal{T}_h}^2
					+\|\nabla\Pi_{k+1}^oZ^n\|_{\mathcal{T}_h}^2+\|\nabla Z^n\|_{\mathcal{T}_h}^2) \\
				&\quad 
				+C\Delta t^3\sum_{n=1}^m\|\partial_{t}^+\Pi_{k+1}^ou^n\|_{\mathcal{T}_h}^2.
			\end{split}
		\end{gather}
		Finally,  we proceed to estimate the right-hand terms. First,  we have
		it follows 
		\begin{align*}
			&C\Delta t\sum_{n=1}^{m}\|\partial_t^+u_{Ih}^n-\partial_{t}u^n\|_{\mathcal{T}_h}^2
			+C\Delta t^3\sum_{n=1}^m\|\partial_{t}^+\Pi_{k+1}^ou^n\|_{\mathcal{T}_h}^2\\
			&\le C\int_{0}^{T}\|\partial_t(u_{Ih}-u)\|_{L^2(\Omega)}^2\;{\rm d}t
			+C\Delta t^2\int_{0}^{T}\|\partial_{tt}u\|_{L^2(\Omega)}^2\;{\rm d}t
			+C\Delta t^2\|\partial_t\Pi_{k+1}^ou\|_{L^2(0, T;L^2(\Omega))}^2.
		\end{align*}
		By direct computation and using  uniform estimate \eqref{u_Ih_infty_bound},  it follows that
		\begin{align}\label{Z^n}
				Z^n=	F(u_{Ih}^n, u_{Ih}^{n-1})-f(u^n)&=(u_{Ih}^n)^3-(u^n)^3+u^n-u^{n-1}+u^{n-1}-u_{Ih}^{n-1}\\
				&\le C|u_{Ih}^n-u^n|+C|u_{Ih}^{n-1}-u^{n-1}|+\Delta t|\partial_{t}^+u^n|\nonumber, 
			\end{align}
			Applying the error estimate for $ u_{Ih} $ in \eqref{error-u-projection},  we obtain
			\begin{align*}
				&C\Delta t\sum_{n=1}^{m}\|\Pi_{k+1}^oZ^n\|_{\mathcal{T}_h}^2\\
				&\le C\Delta t\sum_{n=1}^{m}(\|\Pi_{k+1}^o(u_{Ih}^n-u^n)\|_{\mathcal{T}_h}^2+\|\Pi_{k+1}^o(u_{Ih}^{n-1}-u^{n-1})\|_{\mathcal{T}_h}^2+\Delta t^2\|\partial_{t}^+\Pi_{k+1}^ou^n\|_{\mathcal{T}_h}^2)\\
				&\le Ch^{2k+4}+C\Delta t^3\sum_{n=1}^{m}\|\partial_{t}^+\Pi_{k+1}^ou^n\|_{\mathcal{T}_h}^2, 
			\end{align*}
			By triangle inequality,   approximation property for HDG elliptic operator in \eqref{error-u-projection}  and  \eqref{Z^n},  one gets
			\begin{align*}
				&C\Delta t\sum_{n=1}^{m}(\|\nabla Z^n\|_{\mathcal{T}_h}^2+\|\nabla \Pi_{k+1}^oZ^n\|_{\mathcal{T}_h}^2)\\
				&\le C\Delta t\sum_{n=1}^{m}(\|\nabla(u_{Ih}^n-u^n)\|_{\mathcal{T}_h}^2+\|\nabla(u_{Ih}^{n-1}-u^{n-1})\|_{\mathcal{T}_h}^2+\Delta t^2\|\partial_{t}^+\nabla u^n\|_{\mathcal{T}_h}^2)\\
				&\quad +C\Delta t\sum_{n=1}^{m}(\|\nabla\Pi_{k+1}^o(u_{Ih}^n-u^n)\|_{\mathcal{T}_h}^2+\|\nabla\Pi_{k+1}^o(u_{Ih}^{n-1}-u^{n-1})\|_{\mathcal{T}_h}^2+\Delta t^2\|\partial_{t}^+\nabla\Pi_{k+1}^o u^n\|_{\mathcal{T}_h}^2)\\
				&=C\Delta t\sum_{n=1}^{m}(\|\nabla(u_{Ih}^n-u^n)\|_{\mathcal{T}_h}^2+\|\nabla(u_{Ih}^{n-1}-u^{n-1})\|_{\mathcal{T}_h}^2+\Delta t^2\|\partial_{t}^+\nabla u^n\|_{\mathcal{T}_h}^2)\\
				&\quad +C\Delta t\sum_{n=1}^{m}(\|\nabla(u_{Ih}^n-\Pi_{k+1}^ou^n)\|_{\mathcal{T}_h}^2+\|\nabla(u_{Ih}^{n-1}-\Pi_{k+1}^ou^{n-1})\|_{\mathcal{T}_h}^2+\Delta t^2\|\partial_{t}^+\nabla\Pi_{k+1}^o u^n\|_{\mathcal{T}_h}^2)\\
				&\le Ch^{2k+2}+C\Delta t^3\sum_{n=1}^{m}(\|\partial_{t}^+\nabla u^n\|_{\mathcal{T}_h}^2
				+\|\partial_{t}^+\nabla\Pi_{k+1}^ou^n\|_{\mathcal{T}_h}^2).
			\end{align*}
		
		Since 
		\begin{align*}
			&C\Delta t^3(\|\partial_t^+\Pi_{k+1}^ou^n\|_{\mathcal{T}_h}^2+\|\partial_{t}^+(\nabla\Pi_{k+1}^ou^n)\|_{\mathcal{T}_h}^2+\|\partial_{t}^+\nabla u^n\|_{\mathcal{T}_h}^2)\\
			&\qquad=  C\Delta t\sum_{n=1}^{m}\int_{\Omega}(\int_{t_{n-1}}^{t_n}\partial_{t}\Pi_{k+1}^ou\;{\rm d}t)^2+(\int_{t_{n-1}}^{t_n}\partial_{t}\Pi_{k+1}^o\nabla u\;{\rm d}t)^2
			+(\int_{t_{n-1}}^{t_n}\partial_{t}\nabla u\;{\rm d}t)^2
			\;{\rm d}{\bm x}\\
			&\qquad\le C\Delta t^2\int_{0}^{T}(\|\partial_{t}\Pi_{k+1}^ou\|_{L^2(\Omega)}^2+
			\|\partial_{t}\Pi_{k+1}^o\nabla u\|_{L^2(\Omega)}^2
			+\|\partial_{t}\nabla u\|_{L^2(\Omega)}^2) {\rm d}t, 
		\end{align*}
		choosing $ \theta=\frac{1}{5} $ and substituting the above estimate into \eqref{error-eq-flux-summation-substitute},  we obtain
		\begin{align}\label{flux-error}
			\frac{\epsilon^2}{2} \interleave(e_h^{\bm q^{m}},  e_h^{u^{m}}, e_h^{\widehat{u}^{m}})\interleave^2
			+\frac{\Delta t}{5Pe}\sum_{n=1}^{m}\interleave(e_h^{\bm p^n},  e_h^{\phi^n}, e_h^{\widehat{\phi}^n})\interleave^2 \leq C \big(h^{k+1}+\Delta t\big)^2.
		\end{align}
		Applying the triangle inequality and \eqref{flux-error},  we immediately obtain \eqref{Final-flux}.
	\end{proof}

	\section{Numerical experiments}
	\label{numerics}
	We consider several examples on a square domain in $\mathbb{R}^2$.
	
	{ In the first example, we study error estimates and convergence rates by considering a manufactured solution of the system \eqref{ori}, with suitable source terms added to the right-hand side of \eqref{ori}. We compare two forms of the convection operator $\mathcal{B}$, namely the non-upwind form \eqref{def_B} and the upwind form \eqref{def_B_1}.
		
		In the second and third examples, where no exact solution is available, we investigate spinodal decomposition under circular convection and compare the qualitative behavior of the numerical solutions obtained with and without upwinding. 
		
		Since the scheme is nonlinear, we solve it by Newton's method after linearizing the nonlinear term. The resulting linear systems are solved by the MINRES iteration. The absolute tolerance for the Newton iteration is set to $10^{-11}$, while the absolute and relative tolerances for MINRES are set to $10^{-14}$ and $10^{-12}$, respectively.}
	\begin{example}\label{exmaple1}
		We consider the computational domain  $\Omega=[0, 1]^2$ with final time 
		$T=0.5$,  Péclet number $Pe=3$ and parameter  $\epsilon=2$. The time step size is chosen as $\Delta t =2(\frac{h}{\sqrt{2}})^{k+2}$. The exact solution $(u, \phi)$ is given by
		\begin{align*}
			u(x, y, t)&=50e^{-t}x^2(x-1)^2y^2(y-1)^2, \\
			\phi(x, y, t)&=0.1\sin(t)\cos(2\pi x)\cos(2\pi y).
		\end{align*}
		The velocity field $\boldsymbol{\beta}(x,  y)=\left(\beta_1,  \beta_2 \right)^{\mathrm{T}}$ is defined as
		\begin{align*}
			& \beta_1=\sin^2(\pi x)\sin(\pi y)\cos(\pi y),  \\
			& \beta_2=-\sin^2(\pi y)\sin(\pi x)\cos(\pi x).  
		\end{align*}
		We set $\tau_c=\alpha=10$ and report the $L^2$ norm error between the exact solutions and the numerical solutions obtained using our HDG method in \Cref{table1}. 
		For clarity,  we note that the operator $\mathcal{B}$ is used in the form given by \eqref{def_B}. As stated in the introduction and \Cref{firstremarkforB},  when employing the operator $\mathcal{B}$ in the form of \eqref{def_B_1},  the optimal convergence rate is lost for $k=0$,  while it remains valid for $k \geq 1$. This result is reported in \Cref{table2},  and the corresponding theoretical result can be found in \Cref{proof_main_res}.

	\end{example}

	\begin{table}[H]
		\begin{tabular}{c|c|cc|cc|cc|cc}
			\Xhline{0.1pt}
			\multirow{2}{*}{Degree} & \multirow{2}{*}{$\frac{h}{\sqrt{2}}$} & \multicolumn{2}{c|}{$\| u- u_h\|_{\mathcal T_h}$} & \multicolumn{2}{c|}{$\| \phi - \phi_h\|_{\mathcal T_h}$} & \multicolumn{2}{c|}{$\| \bm q- \bm q_h\|_{\mathcal T_h}$}
			& \multicolumn{2}{c}{$\| \bm p- \bm p_h\|_{\mathcal T_h}$} \\ \cline{3-10}
			& & \multicolumn{1}{c|}{Error} & Rate & \multicolumn{1}{c|}{Error} & Rate & \multicolumn{1}{c|}{Error} & Rate & \multicolumn{1}{c|}{Error} & Rate \\ \hline
			\multirow{7}{*}{$k=0$} 
			& $2^{-3}$ & \multicolumn{1}{c|}{7.05E-03} & - & \multicolumn{1}{c|}{7.47E-03} &-& \multicolumn{1}{c|}{6.52E-02} & -& \multicolumn{1}{c|}{6.79E-02} & - \\
			& $2^{-4}$ & \multicolumn{1}{c|}{1.69E-03} & 2.06 & \multicolumn{1}{c|}{1.75E-03} & 2.09 & \multicolumn{1}{c|}{3.29E-02} & 0.99 & \multicolumn{1}{c|}{3.36E-02} & 1.02 \\
			& $2^{-5}$ & \multicolumn{1}{c|}{4.20E-04} & 2.01 & \multicolumn{1}{c|}{4.28E-04} & 2.03 & \multicolumn{1}{c|}{1.65E-02} & 1.00 & \multicolumn{1}{c|}{1.67E-02} & 1.00 \\
			& $2^{-6}$ & \multicolumn{1}{c|}{1.05E-04} & 2.00 & \multicolumn{1}{c|}{1.07E-04} & 2.00 & \multicolumn{1}{c|}{8.24E-03} & 1.00 & \multicolumn{1}{c|}{8.36E-03} & 1.00 \\           
			& $2^{-7}$ & \multicolumn{1}{c|}{2.62E-05} & 2.00 & \multicolumn{1}{c|}{2.67E-05} & 2.00 & \multicolumn{1}{c|}{4.12E-03} & 1.00 & \multicolumn{1}{c|}{4.18E-03} & 1.00 \\ \hline
			
			\multirow{6}{*}{$k=1$} 
			& $2^{-2}$ & \multicolumn{1}{c|}{6.36E-03} &- & \multicolumn{1}{c|}{8.28E-03} & - & \multicolumn{1}{c|}{3.02E-02} & - & \multicolumn{1}{c|}{3.42E-02} & - \\
			& $2^{-3}$ & \multicolumn{1}{c|}{6.88E-04} & 3.21 & \multicolumn{1}{c|}{9.04E-04} & 3.20 & \multicolumn{1}{c|}{7.97E-03} & 1.92 & \multicolumn{1}{c|}{8.86E-03} & 1.95 \\
			& $2^{-4}$ & \multicolumn{1}{c|}{7.39E-05} & 3.22 & \multicolumn{1}{c|}{5.28E-05} & 3.25 & \multicolumn{1}{c|}{2.02E-03} & 1.98 & \multicolumn{1}{c|}{2.25E-03} & 1.98 \\
			& $2^{-5}$ & \multicolumn{1}{c|}{8.83E-06} & 3.06 & \multicolumn{1}{c|}{1.10E-05} & 3.11 & \multicolumn{1}{c|}{5.08E-04} & 1.99 & \multicolumn{1}{c|}{5.64E-04} & 1.99 \\
			& $2^{-6}$ & \multicolumn{1}{c|}{1.11E-06} & {2.99} & \multicolumn{1}{c|}{1.40E-06} & 2.98 & \multicolumn{1}{c|}{1.27E-04} & 2.00 & \multicolumn{1}{c|}{1.41E-04} & 2.00 \\ \hline
		\end{tabular}
		\caption{\Cref{exmaple1}: The  smooth case,  with the operator $\mathcal B$ defined in \eqref{def_B}}\label{table1}
	\end{table}

	\begin{table}[H]
		\begin{tabular}{c|c|cc|cc|cc|cc}
			\Xhline{0.1pt}
			\multirow{2}{*}{Degree} & \multirow{2}{*}{$\frac{h}{\sqrt{2}}$} & \multicolumn{2}{c|}{$\| u- u_h\|_{\mathcal T_h}$}    & \multicolumn{2}{c|}{$\| \phi - \phi_h\|_{\mathcal T_h}$}    & \multicolumn{2}{c|}{$\| \bm q- \bm q_h\|_{\mathcal T_h}$}
			& \multicolumn{2}{c}{$\| \bm p- \bm p_h\|_{\mathcal T_h}$}     \\ \cline{3-10}
			&                   & \multicolumn{1}{c|}{Error} & Rate & \multicolumn{1}{c|}{Error} & Rate & \multicolumn{1}{c|}{Error} & Rate & \multicolumn{1}{c|}{Error} & Rate                         \\ \hline
			\multirow{7}{*}{$k=0$}                 
			& $2^{-3}$                 & \multicolumn{1}{c|}{5.86E-03}  & -  & \multicolumn{1}{c|}{3.09E-02}  & -  & \multicolumn{1}{c|}{6.48E-02}  & -  & \multicolumn{1}{c|}{1.53E-01}  & -                 \\
			& $2^{-4}$                 & \multicolumn{1}{c|}{2.14E-03}  & 1.46  & \multicolumn{1}{c|}{1.21E-02}  & 1.35  & \multicolumn{1}{c|}{3.36E-02}  & 0.94  & \multicolumn{1}{c|}{7.49E-02}  & 1.03                  \\
			& $2^{-5}$                 & \multicolumn{1}{c|}{9.83E-04}  & 1.12  & \multicolumn{1}{c|}{5.76E-03}  & 1.07  & \multicolumn{1}{c|}{1.71E-02}  & 0.98  & \multicolumn{1}{c|}{3.78E-02}  & 0.99                 \\
			& $2^{-6}$                 & \multicolumn{1}{c|}{5.03E-04}  & 0.97  & \multicolumn{1}{c|}{2.88E-03}  & 1.00  & \multicolumn{1}{c|}{8.59E-03}  & 0.99 & \multicolumn{1}{c|}{1.91E-02}  & 0.98                  \\
			& $2^{-7}$                 & \multicolumn{1}{c|}{2.59E-04}  & 0.96  & \multicolumn{1}{c|}{1.45E-03}  & 0.99  & \multicolumn{1}{c|}{4.31E-03}  & 1.00  & \multicolumn{1}{c|}{9.63E-03}  & 0.99                 \\ \hline
			\multirow{6}{*}{$k=1$} 
			& $2^{-2}$                 & \multicolumn{1}{c|}{4.15E-03}  & -  & \multicolumn{1}{c|}{1.46E-03}  & - & \multicolumn{1}{c|}{2.97E-02}  & -  & \multicolumn{1}{c|}{4.91E-02}  & -                   \\
			& $2^{-3}$                 & \multicolumn{1}{c|}{5.61E-04}  & 2.89  & \multicolumn{1}{c|}{1.26E-03}  & 3.53 & \multicolumn{1}{c|}{7.95E-03}  & 1.90  & \multicolumn{1}{c|}{1.22E-02}  &2.01                  \\
			& $2^{-4}$                 & \multicolumn{1}{c|}{7.05E-05}  & 2.99  & \multicolumn{1}{c|}{1.45E-04}  & 3.12 & \multicolumn{1}{c|}{2.02E-03}  & 1.97  & \multicolumn{1}{c|}{3.04E-03}  & 2.00                  \\
			& $2^{-5}$                 & \multicolumn{1}{c|}{8.82E-06}  & 3.00  & \multicolumn{1}{c|}{1.84E-05}  & 2.98  & \multicolumn{1}{c|}{5.08E-04}  & 1.99  & \multicolumn{1}{c|}{7.62E-04}  & 2.00                  \\
			& $2^{-6}$                 & \multicolumn{1}{c|}{1.10E-06}  & 3.00  & \multicolumn{1}{c|}{2.35E-06}  & 2.97  & \multicolumn{1}{c|}{1.27E-04}  &2.00  & \multicolumn{1}{l|}{1.91E-04}  & 2.00   \\ \hline
		\end{tabular}
		\caption{\Cref{exmaple1}: The  smooth case,  with the operator $\mathcal B$  defined in \eqref{def_B_1}}\label{table2}
	\end{table}
	
	{
		\begin{example}\label{example--2}
			We consider circular convection on the square domain $\Omega=[0,1]^2$ with
			\[
			Pe=200,\qquad \epsilon=\frac{1}{100},\qquad T=60,\qquad h=0.01,\qquad \Delta t=10^{-3}.
			\]
			
			The initial condition $u_0(x,y)$ is chosen as the cross-shaped profile
			\[
			u_0(x,y):=
			\begin{cases}
				1.0, & (x,y)\in [0.25,0.75]\times[0.375,0.625]\ \cup\ [0.375,0.625]\times[0.25,0.75],\\[2mm]
				-1.0, & \text{otherwise}.
			\end{cases}
			\]
			The velocity field is defined by
			\[
			\bm\beta(x,y):=v(r)(2y-1,\,1-2x)^{\mathrm T}, \qquad (x,y)\in \Omega,
			\]
			where
			\[
			v(r):=\frac{1}{2}\left(1+\tanh\left(a\left(\frac{1}{2}-b-r\right)\right)\right),
			\qquad
			r^2:=\left(x-\frac{1}{2}\right)^2+\left(y-\frac{1}{2}\right)^2,
			\]
			with parameters $a=200$ and $b=0.1$. Clearly, $\nabla\cdot\bm\beta=0$, and the no-penetration condition
			$\bm\beta\cdot\bm n=0$ holds to machine precision on $\partial\Omega$.
			
			We first compute the solution using the non-upwind operator $\mathcal{B}$ in the form \eqref{def_B}. The resulting evolution of the cross-shaped profile under circular convection is shown in \Cref{fig:combined-145} at twelve selected times. To assess the effect of upwinding on this more challenging problem, we also perform computations with the upwind operator \eqref{def_B_1} using the same mesh size $h$ and time step $\Delta t$. The corresponding results for $k=0$ and $k=1$ are shown in \Cref{fig:combined-23}. For brevity, in the upwind case we display only four representative snapshots for each polynomial degree.

			\begin{figure}[htbp]
				\caption{Snapshots of the evolution of an initially cross-shaped profile under circular convection at selected times, computed with the operator $\mathcal{B}$ in the form \eqref{def_B} for $k=0$.}
				\centering
				\includegraphics[width=0.23\textwidth]{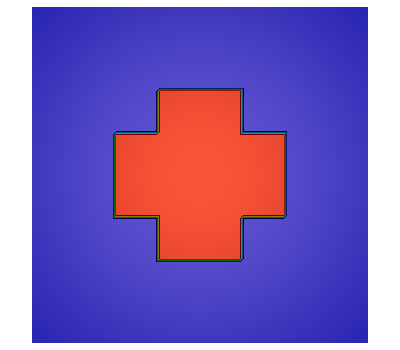} 
				\includegraphics[width=0.23\textwidth]{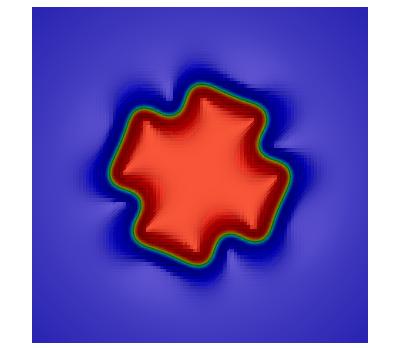} 
				\includegraphics[width=0.23\textwidth]{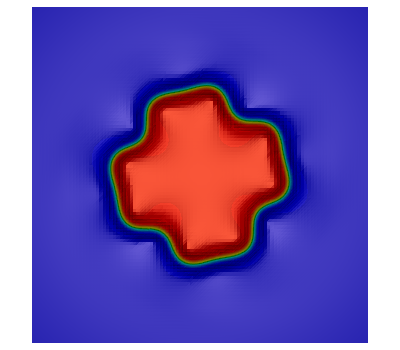} 
				\includegraphics[width=0.23\textwidth]{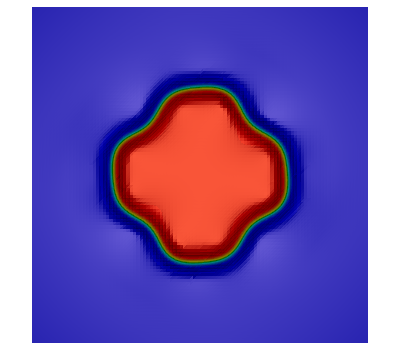}
				\par 
				{\footnotesize From left to right: $t=0$, $0.2$, $0.7$, and $1.6$.}
				
				\vspace{1em}

				\includegraphics[width=0.23\textwidth]{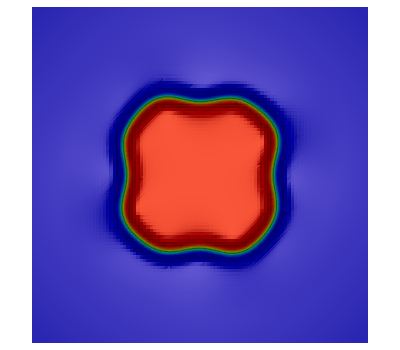} 
				\includegraphics[width=0.23\textwidth]{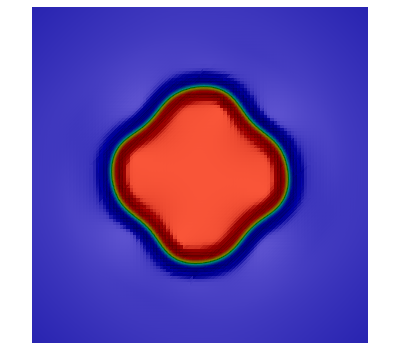} 
				\includegraphics[width=0.23\textwidth]{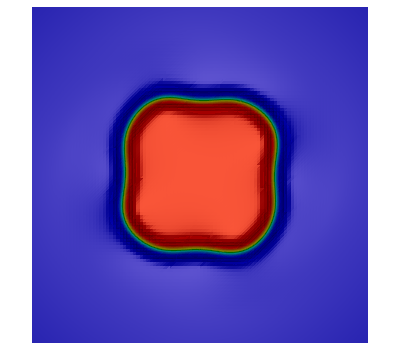} 
				\includegraphics[width=0.23\textwidth]{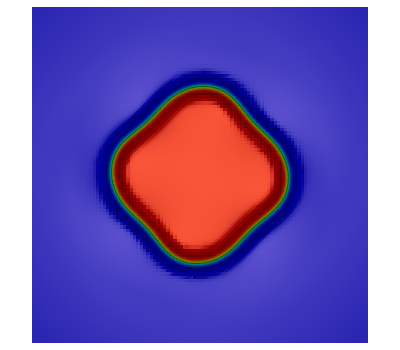}
				
				\par
				
				{\footnotesize From left to right:  $t=2.0$, $2.4$, $2.8$, $3.2$.}
				
				\vspace{1em}

				\includegraphics[width=0.23\textwidth]{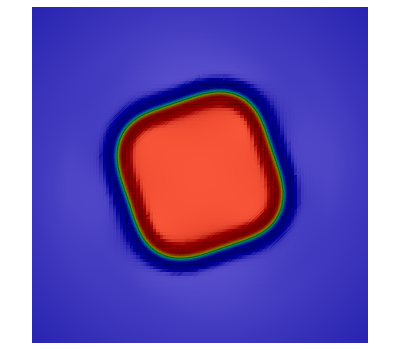} 
				\includegraphics[width=0.23\textwidth]{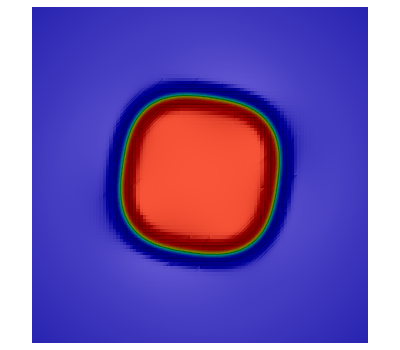} 
				\includegraphics[width=0.23\textwidth]{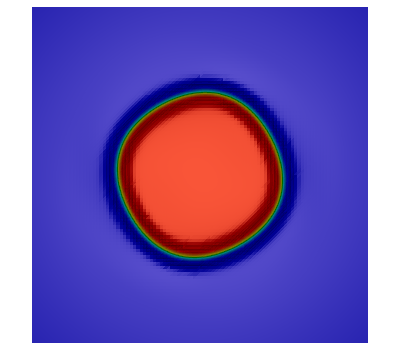} 
				\includegraphics[width=0.23\textwidth]{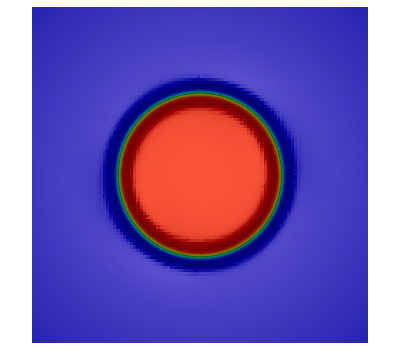}

				\par
				
				{\footnotesize From left to right:   $t=5$, $10$, $20$, and $60$.}
				
				\label{fig:combined-145}
			\end{figure}

			\begin{figure}[htb]
				\caption{Evolution of an initially cross-shaped profile under circular convection, computed with the upwind operator $\mathcal{B}$ in the form \eqref{def_B_1} for different polynomial degrees $k$.}
				
				\centering
				\includegraphics[width=0.23\textwidth]{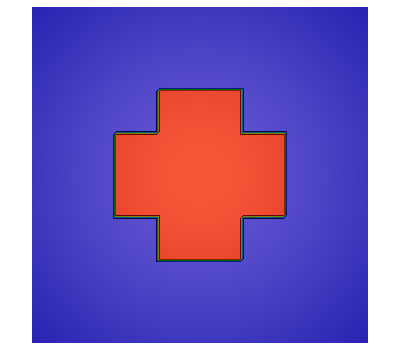}
				\includegraphics[width=0.23\textwidth]{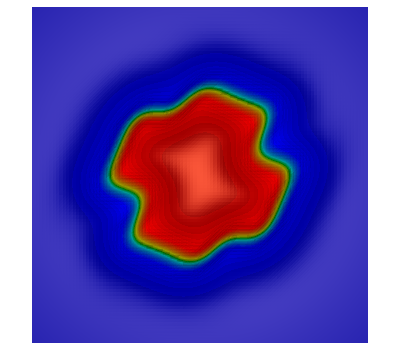}
				\includegraphics[width=0.23\textwidth]{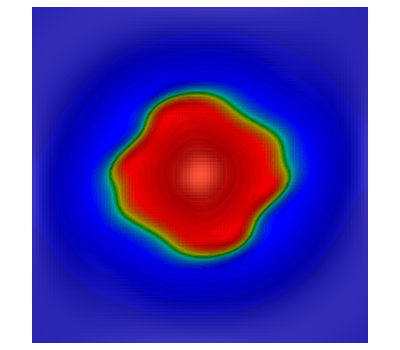}
				\includegraphics[width=0.23\textwidth]{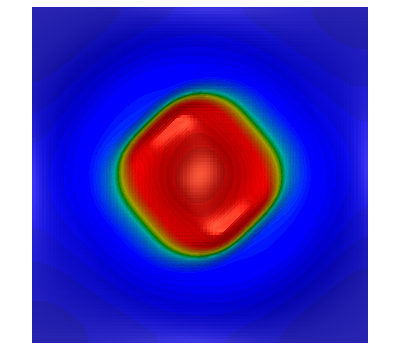}

				\par

				{\footnotesize $k=0$: from left to right, $t=0$, $0.2$, $0.7$, and $1.6$.}

				\vspace{1em}

				\includegraphics[width=0.23\textwidth]{T00--upwind--1--cross.png}
				\includegraphics[width=0.23\textwidth]{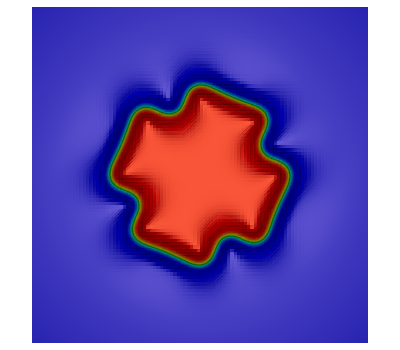}
				\includegraphics[width=0.23\textwidth]{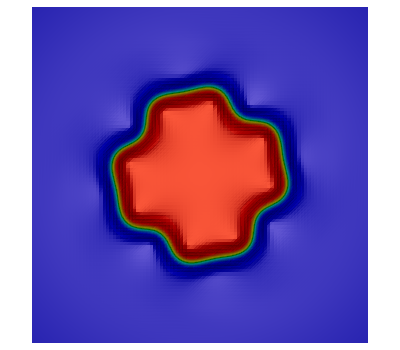}
				\includegraphics[width=0.23\textwidth]{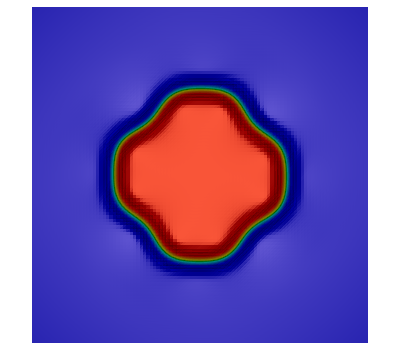}
				\par

				{\footnotesize $k=1$: from left to right, $t=0$, $0.2$, $0.7$, and $1.6$.}
				
				\label{fig:combined-23}
			\end{figure}

			The comparison indicates that, in the lowest-order case $k=0$, the upwind formulation leads to poorer performance. By contrast, for $k=1$ the upwind solution remains qualitatively much closer to the non-upwind computation. These observations are fully consistent with the theoretical analysis and the convergence study in Example~\ref{exmaple1}.
		\end{example}

		\begin{example}
			In this example, we consider a different initial condition from \Cref{example--2}. The parameters are chosen as
			\[
			Pe=200,\quad \epsilon=\frac{1}{200},\quad \alpha=10,\quad \Omega=[0,1]^2,\quad T=350,\quad h=0.01,\quad \Delta t=10^{-3}.
			\]
			
			The initial condition $u_0$ is taken to be uniformly distributed in the interval $[-1,1]$ inside the disk
			\[
			(x-0.5)^2+(y-0.5)^2\le 0.4^2,
			\]
			and $u_0=-1$ outside this disk. The convection field $\bm\beta$ is the same as in \Cref{example--2}.
			
			We first compute the solution using the non-upwind operator $\mathcal{B}$ in the form \eqref{def_B}; the resulting formation and subsequent transport of bulk regions under circular convection are shown in \Cref{figure4}. We then perform the same computation with the upwind operator \eqref{def_B_1}, using the same mesh size $h$ and time step $\Delta t$. The corresponding results are presented in \Cref{figure4-1} for $k=0$ and  $k=1$.
			
			The numerical results again show that the effect of upwinding is most pronounced in the lowest-order case. For $k=0$, the upwind formulation yields a less accurate qualitative evolution, with weaker resolution of the interfaces and less distinct bulk structures. In contrast, for $k=1$, the upwind and non-upwind solutions remain qualitatively close throughout the simulation. This behavior is consistent with the conclusions of Example~\ref{example--2} and with the theoretical and numerical findings in Example~\ref{exmaple1}.
		\end{example}

		\begin{figure}[htbp]
			\caption{Snapshots of bulk regions under circular convection  at selected times, computed with the operator $\mathcal{B}$ in the form \eqref{def_B} for $k=0$.}
			\centering
			
			\includegraphics[width=0.23\textwidth]{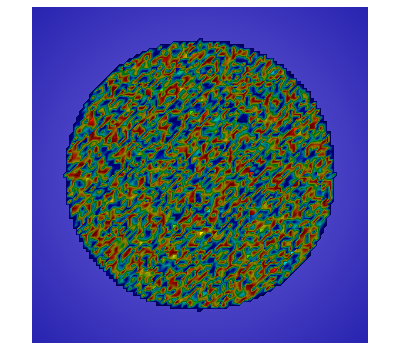}
			\includegraphics[width=0.23\textwidth]{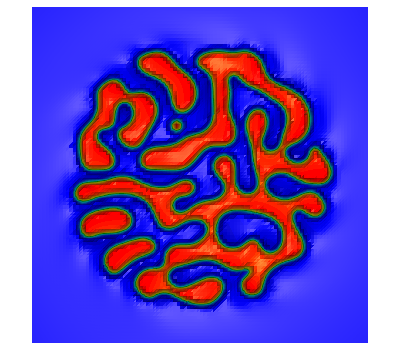}
			\includegraphics[width=0.23\textwidth]{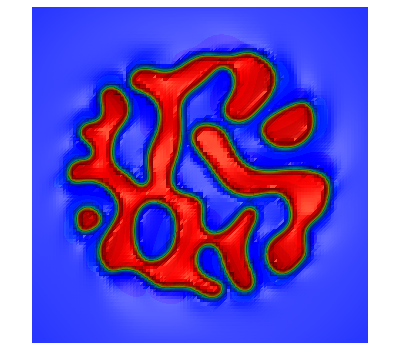}
			\includegraphics[width=0.23\textwidth]{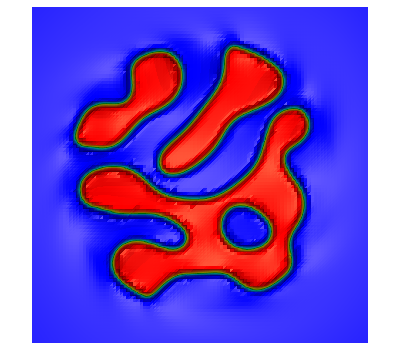}
			\par
			{\footnotesize From left to right:  $t=0$, $2$, $6$, $15$.}
			
			\vspace{1em}

			\includegraphics[width=0.23\textwidth]{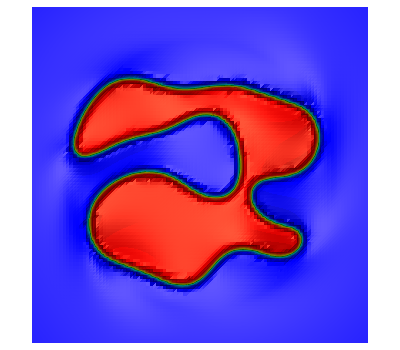}
			\includegraphics[width=0.23\textwidth]{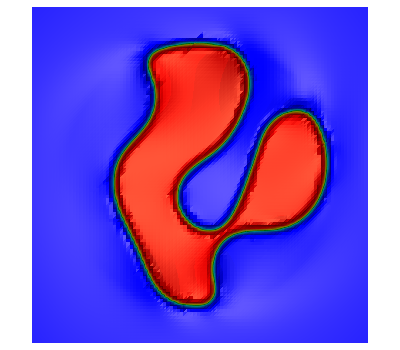}
			\includegraphics[width=0.23\textwidth]{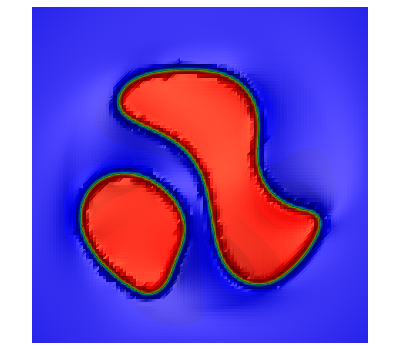}
			\includegraphics[width=0.23\textwidth]{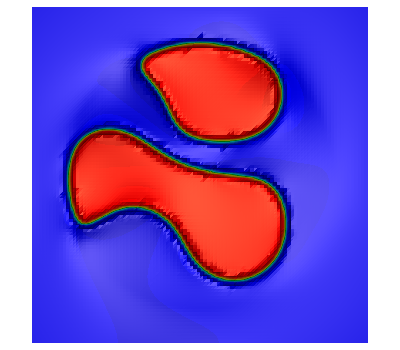}
			
			\par
			{\footnotesize From left to right:  $t=60$, $90$, $120$, $150$.}
			
			\vspace{1em}

			\includegraphics[width=0.23\textwidth]{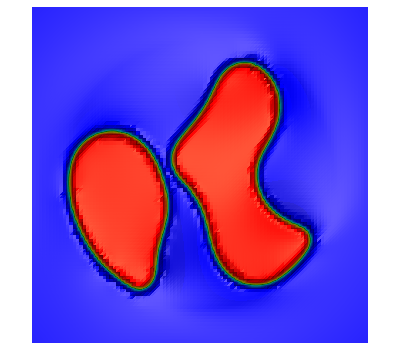}
			\includegraphics[width=0.23\textwidth]{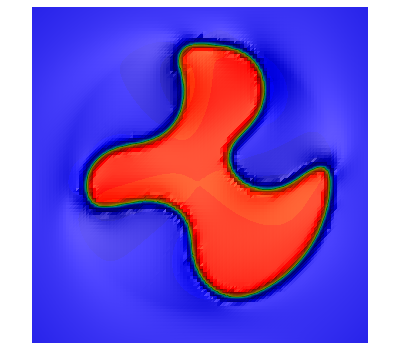}
			\includegraphics[width=0.23\textwidth]{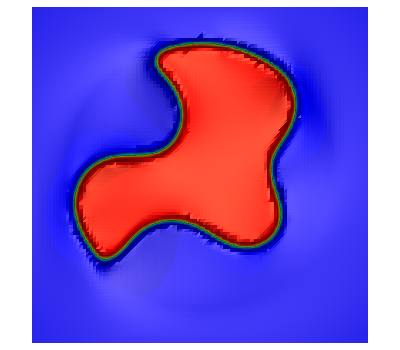}
			\includegraphics[width=0.23\textwidth]{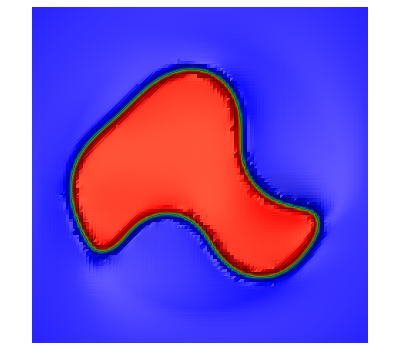}
			
			\par
			{\footnotesize From left to right:  $t=200$, $250$, $300$, $350$.}
			
			\vspace{1em}
			
			\label{figure4}
		\end{figure}
		
		\begin{figure}[htb]
			\caption{Evolution of  bulk regions under circular convection, computed with the upwind operator $\mathcal{B}$ in the form \eqref{def_B_1} for different polynomial degrees $k$.}
			\centering
			
			\includegraphics[width=0.23\textwidth]{T00.png}
			\includegraphics[width=0.23\textwidth]{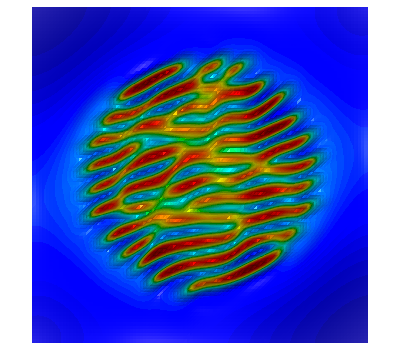}
			\includegraphics[width=0.23\textwidth]{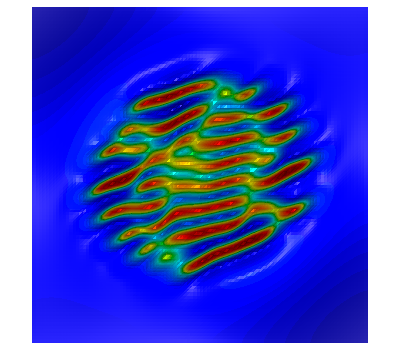}
			\includegraphics[width=0.23\textwidth]{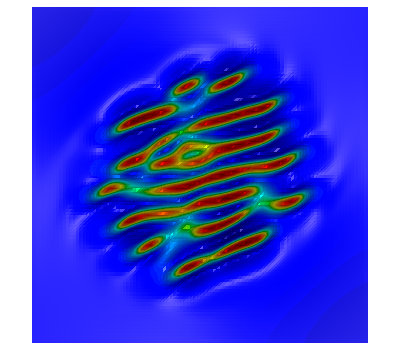}
			
			\par

			{\footnotesize $k=0$: from left to right, $t=0$, $2$, $6$, $15$.}

			\vspace{1em}
			
			\includegraphics[width=0.23\textwidth]{T00.png}
			\includegraphics[width=0.23\textwidth]{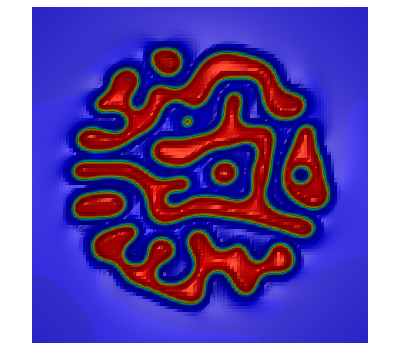}
			\includegraphics[width=0.23\textwidth]{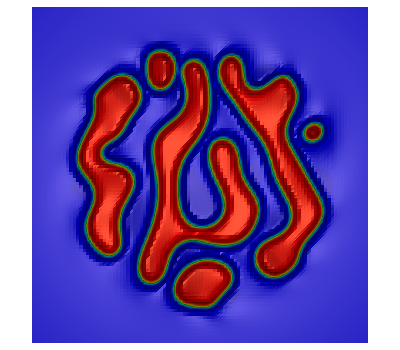}
			\includegraphics[width=0.23\textwidth]{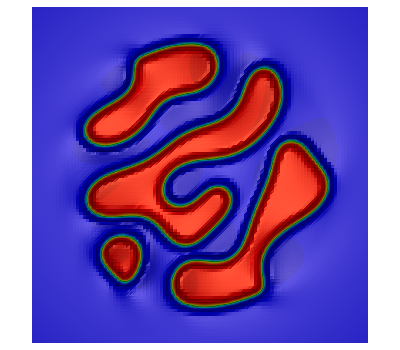}
			
			\par

			{\footnotesize $k=1$: from left to right, $t=0$, $2$, $6$, $15$.} 
			\label{figure4-1}
		\end{figure}
	}

	\section{Appendix}\label{Appendix}
	In this section,  we provide the complete proof of   \eqref{error-phi-projection} and \eqref{error-p-projection},   dividing it into two steps for clarity.
	\subsection*{Step 1: An energy argument}
	First,  we derive the projection equations \eqref{projection--eq--phi-u}. For any $(w_1, \mu_1)\in W_h\times M_h$,  using the definitions of the operators $ \mathcal{A} $ and $ \mathcal{B} $,  we obtain:
	\begin{align*}
		&\frac{1}{Pe}\mathcal{A}(\bm{\Pi}_k^o\bm{p}, \Pi_{k+1}^o\phi, \Pi_k^{\partial}\phi;\bm{r}_1, w_1, \mu_1)+\mathcal{B}(\Pi_{k+1}^{o}u, \Pi_k^{\partial}u;w_1)\\
		&=\frac{1}{Pe}[(\bm{\Pi}_k^o\bm{p}, \bm{r}_1)_{\mathcal{T}_h}-(\Pi_{k+1}^{o}\phi, \nabla\cdot\bm{r}_1)_{\mathcal{T}_h}+\langle \Pi_k^{\partial}\phi, \bm{r}_1\cdot\bm{n}\rangle_{\partial\mathcal{T}_h}]\\
		&\quad+\frac{1}{Pe}[(\nabla\cdot\bm{\Pi}_k^o\bm{p}, w_1)_{\mathcal{T}_h}-\langle \bm{\Pi}_k^o\bm {p}\cdot\bm{n}, \mu_1\rangle_{\partial\mathcal{T}_h}+\alpha\langle h_{\tau}^{-1}(\Pi_k^{\partial}\Pi_{k+1}^o\phi-\Pi_k^{\partial}\phi), \Pi_k^{\partial}w_1-\mu_1\rangle_{\partial\mathcal{T}_h}]\\
		&\quad-(\bm{\beta}\Pi_{k+1}^ou, \nabla w_1)_{\mathcal{T}_h}
		+\langle \bm{\beta}\cdot\bm{n}\Pi_k^{\partial}u, w_1\rangle_{\partial\mathcal{T}_h}
		.
	\end{align*}
	By the definition of the projection operator and applying integration by parts,  we obtain:
	\begin{align*}
		&\frac{1}{Pe}\mathcal{A}(\bm{\Pi}_k^o\bm{p}, \Pi_{k+1}^o\phi, \Pi_k^{\partial}\phi;\bm{r}_1, w_1, \mu_1)+\mathcal{B}(\Pi_{k+1}^{o}u, \Pi_k^{\partial}u;w_1)\\
		&=\frac{1}{Pe}[(\bm{p}, \bm{r}_1)_{\mathcal{T}_h}-(\phi, \nabla\cdot\bm{r}_1)_{\mathcal{T}_h}+\langle \phi, \bm{r}_1\cdot\bm{n}\rangle_{\partial\mathcal{T}_h}
		+\alpha\langle h_{\tau}^{-1}(\Pi_{k+1}^o\phi-\phi), \Pi_k^{\partial}w_1-\mu_1\rangle_{\partial\mathcal{T}_h}]\\
		&\quad+\frac{1}{Pe}[-(\bm{\Pi}_k^o\bm{p}, \nabla w_1)_{\mathcal{T}_h}
		+\langle \bm{\Pi}_k^o\bm{p}\cdot\bm{n}, w_1\rangle_{\partial\mathcal{T}_h}
		-\langle \bm{\Pi}_k^o\bm{p}\cdot\bm{n}, \mu_1\rangle_{\partial\mathcal{T}_h}
		]\\
		&\quad-(\bm{\beta}\Pi_{k+1}^ou, \nabla w_1)_{\mathcal{T}_h}
		+\langle \bm{\beta}\cdot\bm{n}\Pi_k^{\partial}u, w_1\rangle_{\partial\mathcal{T}_h}
		.
	\end{align*}
	Noting the fact that $ (\bm{p}+\nabla \phi, \bm{r}_1)_{\mathcal{T}_h}=0 $,  we obtain:
	\begin{align}\label{appendix-temp-1}
		\begin{split}
			&\frac{1}{Pe}\mathcal{A}(\bm{\Pi}_k^o\bm{p}, \Pi_{k+1}^o\phi, \Pi_k^{\partial}\phi;\bm{r}_1, w_1, \mu_1)+\mathcal{B}(\Pi_{k+1}^{o}u, \Pi_k^{\partial}u;w_1)\\
			&=\frac{1}{Pe}[-(\bm{p}, \nabla w_1)_{\mathcal{T}_h}+\langle\bm{\Pi}_k^o\bm{p}\cdot\bm{n}, w_1-\mu_1\rangle_{\partial\mathcal{T}_h}
			+\alpha\langle h_{\tau}^{-1}(\Pi_{k+1}^o\phi-\phi), \Pi_k^{\partial}w_1-\mu_1\rangle_{\partial\mathcal{T}_h}]\\
			&\quad-(\bm{\beta}\Pi_{k+1}^ou, \nabla w_1)_{\mathcal{T}_h}
			+\langle \bm{\beta}\cdot\bm{n}\Pi_k^{\partial}u, w_1\rangle_{\partial\mathcal{T}_h}.
		\end{split}				
	\end{align}
	
		Since $\bm p\in \bm H(\mathrm{div};\Omega)$ and $\bm p\cdot\bm n=0$ on $\partial\Omega$, we have
		\[
		\langle \bm p\cdot\bm n,\mu_1\rangle_{\partial\mathcal T_h}=0
		\qquad \forall\,\mu_1\in M_h.
		\]
		Moreover, since both $u$ and $\Pi_k^\partial u$ are single-valued on each interior face, the corresponding interior-face contributions cancel when the faces are counted with opposite normals; on boundary faces they vanish because $\bm\beta\cdot\bm n=0$ on $\partial\Omega$. Therefore,
		\[
		\langle \bm\beta\cdot\bm n(\Pi_k^\partial u-u),\mu_1\rangle_{\partial\mathcal T_h}=0
		\qquad \forall\,\mu_1\in M_h.
		\]
		Substituting these identities into \eqref{appendix-temp-1} and integrating by parts, we obtain
	
	\begin{align}\label{projection--eq--phi-u} 
		\begin{split} &\frac{1}{Pe}\mathcal{A}(\bm{\Pi}_k^o\bm{p}, \Pi_{k+1}^o\phi, \Pi_k^{\partial}\phi;\bm{r}_1, w_1, \mu_1)+\mathcal{B}(\Pi_{k+1}^{o}u, \Pi_k^{\partial}u;w_1)\\ &=\frac{1}{Pe}[(\nabla\cdot\bm{p}, w_1)_{\mathcal{T}_h} -\langle\bm{p}\cdot\bm{n}, w_1\rangle_{\partial\mathcal{T}_h} +\langle\bm{p}\cdot\bm{n}, \mu_1\rangle_{\partial \mathcal{T}_h} +\langle\Pi_k^o\bm{p}\cdot\bm{n}, w_1-\mu_1\rangle_{\partial\mathcal{T}_h} ] \\ &\quad +\frac{\alpha}{Pe}\langle h_{\tau}^{-1}(\Pi_{k+1}^o\phi-\phi), \Pi_{k}^{\partial}w_1-\mu_1\rangle_{\partial\mathcal{T}_h} -(\bm{\beta}(\Pi_{k+1}^o u-u), \nabla w_1)_{\mathcal{T}_h}\\ &\quad+\langle\bm{\beta}\cdot\bm{n}(\Pi_{k}^{\partial}u-u), w_1\rangle_{\partial \mathcal{T}_h} -(\bm{\beta}u, \nabla w_1)_{\mathcal{T}_h}+\langle\bm{\beta}\cdot\bm{n}u, w_1\rangle_{\partial \mathcal{T}_h} -\langle\bm{\beta}\cdot\bm{n}(\Pi_{k}^{\partial}u-u), \mu_1\rangle_{\partial \mathcal{T}_h} \\ &=\frac{1}{Pe}[(\nabla\cdot\bm{p}, w_1)_{\mathcal{T}_h} +\langle(\bm{\Pi}_k^o\bm{p}-\bm{p})\cdot\bm{n}, w_1-\mu_1\rangle_{\partial\mathcal{T}_h} +\alpha\langle h_{\tau}^{-1}(\Pi_{k+1}^o\phi-\phi), \Pi_k^{\partial}w_1-\mu_1\rangle_{\partial\mathcal{T}_h}]\\ &\quad -(\bm{\beta}(\Pi_{k+1}^ou-u), \nabla w_1)_{\mathcal{T}_h} +\langle\bm{\beta}\cdot\bm{n}(\Pi_k^{\partial}u-u), w_1-\mu_1\rangle_{\partial\mathcal{T}_h}+(\nabla\cdot(\bm{\beta}u), w_1)_{\mathcal{T}_h}, \end{split} 
	\end{align}

		To simplify the notation,  we introduce the following symbols:
		\begin{align}\label{error-notation-varepsilon}
			\begin{split}
				&\varepsilon_h^{\bm{p}}=\bm{\Pi}_k^o\bm{p}-\bm{p}_{Ih}, \ 
				\varepsilon_h^{\phi}=\Pi_{k+1}^o\phi-\phi_{Ih}, \
				\varepsilon_h^{\widehat{\phi}}=\Pi_k^{\partial}\phi-\widehat{\phi}_{Ih}, \\
				&\varepsilon_h^{\bm{q}}=\bm{\Pi}_k^o\bm{q}-\bm{q}_{Ih}, \ 
				\varepsilon_h^{u}=\Pi_{k+1}^ou-u_{Ih}, \
				\varepsilon_h^{\widehat{u}}=\Pi_k^{\partial}u-\widehat{u}_{Ih}.
			\end{split}			
		\end{align}
		Subtracting \eqref{HDG-projection--phi--u}   from \eqref{projection--eq--phi-u},  one obtains:
		\begin{gather}\label{projection--error--phi--u}
			\begin{split}
				&\frac{1}{Pe}\mathcal{A}(\varepsilon_h^{\bm{p}}, \varepsilon_h^{\phi}, \varepsilon_h^{\widehat{\phi}};\bm{r}_1, w_1, \mu_1)+\mathcal{B}(\varepsilon_h^{u}, \varepsilon_h^{\widehat{u}};w_1)\\
				&=\frac{1}{Pe}[\langle(\bm{\Pi}_k^o\bm{p}-\bm{p})\cdot\bm{n}, w_1-\mu_1\rangle_{\partial\mathcal{T}_h}+\alpha\langle h_{\tau}^{-1}(\Pi_{k+1}^o\phi-\phi), \Pi_k^{\partial}w_1-\mu_1\rangle_{\partial\mathcal{T}_h}]\\
				&\quad -(\bm{\beta}(\Pi_{k+1}^ou-u), \nabla w_1)_{\mathcal{T}_h}+\langle \bm{\beta}\cdot\bm{n}(\Pi_k^{\partial}u-u), w_1-\mu_1\rangle_{\partial\mathcal{T}_h}.
			\end{split}
		\end{gather}
		
		Substitute   $ (\bm{r}_1, w_1, \mu_1)=(\varepsilon_h^{\bm{p}}, \varepsilon_h^{\phi}, \varepsilon_h^{\widehat{\phi}}) $ into \eqref{projection--error--phi--u} and apply integration by parts to derive:
		\begin{align*}
			&\frac{1}{Pe}\|\varepsilon_h^{\bm{p}}\|_{\mathcal{T}_h}^2+\frac{\alpha}{Pe}\|h_{\tau}^{-1/2}(\Pi_k^{\partial}\varepsilon_h^{\phi}-\varepsilon_h^{\widehat{\phi}})\|_{\partial\mathcal{T}_h}^2\\
			&=(\bm{\beta}\varepsilon_h^{u}, \nabla\varepsilon_h^{\phi})_{\mathcal{T}_h}
			-\langle \bm{\beta}\cdot\bm{n}\varepsilon_h^{\widehat{u}}, \varepsilon_h^{\phi}\rangle_{\partial\mathcal{T}_h}
			\\
			&\quad+\frac{1}{Pe}[\langle(\bm{\Pi}_k^o\bm{p}-\bm{p})\cdot\bm{n}, \varepsilon_h^{\phi}-\varepsilon_h^{\widehat{\phi}}\rangle_{\partial\mathcal{T}_h}+\alpha\langle h_{\tau}^{-1}(\Pi_{k+1}^o\phi-\phi), \Pi_k^{\partial}\varepsilon_h^{\phi}-\varepsilon_h^{\widehat{\phi}}\rangle_{\partial\mathcal{T}_h}]\\
			&\quad -(\bm{\beta}(\Pi_{k+1}^o u-u), \nabla \varepsilon_h^{\phi})_{\mathcal{T}_h}+\langle \bm{\beta}\cdot\bm{n}(\Pi_k^{\partial}u-u), \varepsilon_h^{\phi}-\varepsilon_h^{\widehat{\phi}}\rangle_{\partial\mathcal{T}_h}
			\\
			&=-(\bm{\beta}\cdot\nabla\varepsilon_h^u, \varepsilon_h^{\phi})_{\mathcal{T}_h}
			+\langle\bm{\beta}\cdot\bm{n}(\varepsilon_h^u-\varepsilon_h^{\widehat{u}}), \varepsilon_h^{\phi}\rangle_{\partial\mathcal{T}_h}
			\\
			&\quad+\frac{1}{Pe}[\langle(\bm{\Pi}_k^o\bm{p}-\bm{p})\cdot\bm{n}, \varepsilon_h^{\phi}-\varepsilon_h^{\widehat{\phi}}\rangle_{\partial\mathcal{T}_h}+\alpha\langle h_{\tau}^{-1}(\Pi_{k+1}^o\phi-\phi), \Pi_k^{\partial}\varepsilon_h^{\phi}-\varepsilon_h^{\widehat{\phi}}\rangle_{\partial\mathcal{T}_h}]\\
			&\quad -(\bm{\beta}(\Pi_{k+1}^o u-u), \nabla \varepsilon_h^{\phi})_{\mathcal{T}_h}+\langle \bm{\beta}\cdot\bm{n}(\Pi_k^{\partial}u-u), \varepsilon_h^{\phi}-\varepsilon_h^{\widehat{\phi}}\rangle_{\partial\mathcal{T}_h}\\
			&=\sum_{i=1}^{6}R_i.
		\end{align*}
		
		Next, we estimate the terms $R_1$--$R_6$. {We first estimate $R_1$ and $R_2$ together. By the Cauchy--Schwarz inequality, we have 
			\begin{align}\label{esr1r2_1}
				R_1+R_2
				\le
				C\Bigl(\|\nabla\varepsilon_h^u\|_{\mathcal{T}_h}
				+\|h_{\tau}^{-1/2}(\varepsilon_h^u-\varepsilon_h^{\widehat{u}})\|_{\partial\mathcal{T}_h}\Bigr)
				\|\varepsilon_h^{\phi}\|_{\mathcal{T}_h}.
			\end{align}
			
			By \eqref{HDG-projection--u}, we have $(\phi-\phi_{Ih},1)_{\mathcal T_h}=0$, and by the definition of the $L^2$ projection $\Pi_{k+1}^o$, we also have $(\Pi_{k+1}^o\phi-\phi,1)_{\mathcal T_h}=0$. Therefore,
			\[
			(\varepsilon_h^{\phi},1)_{\mathcal T_h}
			=
			(\Pi_{k+1}^o\phi-\phi_{Ih},1)_{\mathcal T_h}
			=0,
			\]
			which shows that $\varepsilon_h^{\phi}\in \mathring{W}_h$. Hence, by \eqref{sobolev-02},
			\begin{align}\label{esr1r2_2}
				\|\varepsilon_h^\phi\|_{\mathcal T_h}
				\le
				C\Bigl(
				\|\nabla \varepsilon_h^\phi\|_{\mathcal T_h}^2
				+\alpha\|h_\tau^{-1/2}(\Pi_k^\partial\varepsilon_h^\phi-\varepsilon_h^{\widehat\phi})\|_{\partial\mathcal T_h}^2
				\Bigr)^{1/2}.
			\end{align}
			
			Next, taking $(w_1,\mu_1)=(0,0)$ in \eqref{projection--error--phi--u}, we obtain
			\[
			\mathcal{A}(\varepsilon_h^{\bm p}, \varepsilon_h^{\phi}, \varepsilon_h^{\widehat{\phi}};\bm r_1,0,0)=0,
			\qquad \forall\,\bm r_1\in \bm V_h.
			\]
			Therefore, by \eqref{es_u},
			\begin{align}\label{esr1r2_3}
				\|\nabla \varepsilon_h^\phi\|_{\mathcal T_h} + \|h_\tau^{-1/2}(\varepsilon_h^\phi-\varepsilon_h^{\widehat\phi})\|_{\partial\mathcal T_h}
				\le
				C\Bigl(
				\|\varepsilon_h^{\bm p}\|_{\mathcal T_h}^2
				+\alpha\|h_\tau^{-1/2}(\Pi_k^\partial\varepsilon_h^\phi-\varepsilon_h^{\widehat\phi})\|_{\partial\mathcal T_h}^2
				\Bigr)^{1/2}.
			\end{align}
			
			Finally, using the approximation result in \cite[Lemma~5.7]{ChenGang2018Otso} together with \eqref{esr1r2_2} and \eqref{esr1r2_3}, and substituting these bounds into \eqref{esr1r2_1}, we obtain
			\begin{align*}
				R_1+R_2
				\le
				Ch^{k+1}|u|_{k+2}
				\Bigl(
				\|\varepsilon_h^{\bm p}\|_{\mathcal T_h}^2
				+\alpha\|h_\tau^{-1/2}(\Pi_k^\partial\varepsilon_h^\phi-\varepsilon_h^{\widehat\phi})\|_{\partial\mathcal T_h}^2
				\Bigr)^{1/2}.
			\end{align*}

			Next, we estimate $R_3+R_4$ together.  By using \eqref{esr1r2_3} we have 
			\begin{align*}
				R_3+R_4&\le 
				Ch^{k+1}|\phi|_{k+2}(\|h_{\tau}^{-1/2}(\varepsilon_h^{\phi}-\varepsilon_h^{\widehat{\phi}}
				)
				\|_{\partial\mathcal{T}_h}+\|h_{\tau}^{-1/2}(\Pi_k^{\partial}\varepsilon_h^{\phi}-\varepsilon_h^{\widehat{\phi}})\|_{\partial\mathcal{T}_h})\\
				&\le Ch^{k+1}|\phi|_{k+2}(\|\varepsilon_h^{\bm{p}}\|_{\mathcal{T}_h}^2+\alpha\|h_{\tau}^{-1/2}(\Pi_k^{\partial}\varepsilon_h^{\phi}-\varepsilon_h^{\widehat{\phi}})\|_{\partial\mathcal{T}_h}^2)^{1/2}.
		\end{align*}	}
		Applying the Cauchy-Schwarz inequality and \eqref{es_u} to the terms $ R_5-R_6 $,  we obtain:
		\begin{align*}
			R_5+R_6&\le Ch^{k+2}|u|_{k+2}\|\nabla\varepsilon_h^{\phi}\|_{\mathcal{T}_h}
			+Ch^{k+1}|u|_{k+1}\|h_{\tau}^{-1/2}(\varepsilon_h^{\phi}-\varepsilon_h^{\widehat{\phi}})\|_{\mathcal{T}_h}\\
			&\le Ch^{k+1}(|u|_{k+1}+h|u|_{k+2})(\|\varepsilon_h^{\bm{p}}\|_{\mathcal{T}_h}^2+\alpha\|h_{\tau}^{-1/2}(\Pi_k^{\partial}\varepsilon_h^{\phi}-\varepsilon_h^{\widehat{\phi}})\|_{\partial\mathcal{T}_h}^2)^{1/2}.
		\end{align*}
		Combining all the above estimates for the terms  $ R_1-R_8 $,  we obtain:
		\begin{align*}
			\|\varepsilon_h^{\bm{p}}\|_{\mathcal{T}_h}+\alpha\|h_{\tau}^{-1/2}(\Pi_k^{\partial}\varepsilon_h^{\phi}-\varepsilon_h^{\widehat{\phi}})\|_{\partial\mathcal{T}_h}\le Ch^{k+1}(|u|_{k+1}+|u|_{k+2}+|\phi|_{k+2}).
		\end{align*}
		It follows from the triangle inequality that:
		\begin{align*}
			\|\bm{p}-\bm{p}_{Ih}\|_{\mathcal{T}_h}+\alpha^{1/2}\|h_{\tau}^{-1/2}(\Pi_k^{\partial}\phi_{Ih}-\widehat{\phi}_{Ih})\|_{\mathcal{T}_h}\le Ch^{k+1}(|u|_{k+1}+|u|_{k+2}+|\phi|_{k+2}).
		\end{align*}
		
		This completes the proof of \eqref{error-p-projection}.

		\subsection*{Step 2: $ L^2 $ norm estimate by a duality argument}
		Let $ f_1\in L^2(\Omega) $ and $ (\bm \Theta, \theta) $ be the solutions of \eqref{eq-phi-u-convective-possion}. For any $ (\bm{r}_1, w_1, \mu_1) \in \bm{V}_h\times W_h\times M_h $,  one obtains:
		
		\begin{align}
			\label{dual-projection-equation--2}
			&\frac{1}{Pe}\mathcal{A}(\bm{\Pi}_k^o\bm{\Theta}, \Pi_{k+1}^o\theta, \Pi_k^{\partial}\theta;\bm{r}_1, w_1, \mu_1)+\mathcal{B}(\Pi_{k+1}^{o}\psi, \Pi_k^{\partial}\psi;w_1)\\
			&=(f_1, w_1)_{\mathcal{T}_h}
			+\frac{1}{Pe}[\langle(\bm{\Pi}_k^o\bm{\Theta}-\bm{\Theta})\cdot\bm{n}, w_1-\mu_1\rangle_{\partial\mathcal{T}_h}\nonumber\\
			&\quad
			+\alpha\langle h_{\tau}^{-1}(\Pi_{k+1}^o\theta-\theta), \Pi_k^{\partial}w_1-\mu_1\rangle_{\partial\mathcal{T}_h}]\nonumber\\
			&\quad -(\bm{\beta}\cdot(\Pi_{k+1}^o\psi-\psi), \nabla w_1)_{\mathcal{T}_h}
			+\langle \bm{\beta}\cdot\bm{n}(\Pi_k^{\partial}\psi-\psi), w_1-\mu_1\rangle_{\partial\mathcal{T}_h}\nonumber.
		\end{align}
		
		Taking $ (\bm{r}_1, w_1, \mu_1)=(\varepsilon_h^{\bm{p}}, -\varepsilon_h^{\phi}, -\varepsilon_h^{\widehat{\phi}}) $ and $ f_1=-\varepsilon_h^{\phi} $ in \eqref{dual-projection-equation--2},  one gets:
		\begin{align*}
			\|\varepsilon_h^{\phi}\|_{\mathcal{T}_h}^2&=
			\frac{1}{Pe}\mathcal{A}(\bm{\Pi}_k^o\bm{\Theta}, \Pi_{k+1}^o\theta, \Pi_k^{\partial}\theta;\varepsilon_h^{\bm{p}}, -\varepsilon_h^{\phi}, -\varepsilon_h^{\widehat{\phi}})
			+\mathcal{B}(\Pi_{k+1}^o\psi, \Pi_k^{\partial}\psi;-\varepsilon_h^{\phi})\\
			&\quad+\frac{1}{Pe}[\langle (\bm{\Pi}_k^o\bm{\Theta}-\bm{\Theta})\cdot\bm{n}, \varepsilon_h^{\phi}-\varepsilon_h^{\widehat{\phi}}\rangle_{\partial\mathcal{T}_h}+\alpha\langle h_{\tau}^{-1}(\Pi_{k+1}^o\theta-\theta), \Pi_k^{\partial}\varepsilon_h^{\phi}-\varepsilon_h^{\widehat{\phi}}\rangle_{\partial\mathcal{T}_h}]\\
			&\quad -(\bm{\beta}(\Pi_{k+1}^o\psi-\psi), \nabla\varepsilon_h^{\phi})_{\mathcal{T}_h}
			+\langle \bm{\beta}\cdot\bm{n}(\Pi_k^{\partial}\psi-\psi), \varepsilon_h^{\phi}-\varepsilon_h^{\widehat{\phi}}\rangle_{\partial\mathcal{T}_h}\\
			&=\frac{1}{Pe}\mathcal{A}(\varepsilon_h^{\bm{p}}, \varepsilon_h^{\phi}, \varepsilon_h^{\widehat{\phi}};\bm{\Pi}_k^o\bm{\Theta}, -\Pi_{k+1}^o\theta, -\Pi_k^{\partial}\theta)
			+\mathcal{B}(\Pi_{k+1}^o\psi, \Pi_k^{\partial}\psi;-\varepsilon_h^{\phi})\\
			&\quad+\frac{1}{Pe}[\langle (\bm{\Pi}_k^o\bm{\Theta}-\bm{\Theta})\cdot\bm{n}, \varepsilon_h^{\phi}-\varepsilon_h^{\widehat{\phi}}\rangle_{\partial\mathcal{T}_h}+\alpha\langle h_{\tau}^{-1}(\Pi_{k+1}^o\theta-\theta), \Pi_k^{\partial}\varepsilon_h^{\phi}-\varepsilon_h^{\widehat{\phi}}\rangle_{\partial\mathcal{T}_h}]\\
			&\quad -(\bm{\beta}(\Pi_{k+1}^o\psi-\psi), \nabla\varepsilon_h^{\phi})_{\mathcal{T}_h}
			+\langle \bm{\beta}\cdot\bm{n}(\Pi_k^{\partial}\psi-\psi), \varepsilon_h^{\phi}-\varepsilon_h^{\widehat{\phi}}\rangle_{\partial\mathcal{T}_h}\\
			&=:\sum_{i=1}^{6}S_i, 
		\end{align*}
		where the second equality we use the fact that
		\begin{align*}
			\mathcal{A}(\bm{\Pi}_k^o\bm{\Theta}, \Pi_{k+1}^o\theta, \Pi_k^{\partial}\theta;\varepsilon_h^{\bm{p}}, -\varepsilon_h^{\phi}, -\varepsilon_h^{\widehat{\phi}})=
			\mathcal{A}(\varepsilon_h^{\bm{p}}, \varepsilon_h^{\phi}, \varepsilon_h^{\widehat{\phi}};\bm{\Pi}_k^o\bm{\Theta}, -\Pi_{k+1}^o\theta, -\Pi_k^{\partial}\theta).
		\end{align*}
		Next,  we turn to estimate the terms $ \sum_{i=1}^{6}S_i $.

		First,  we use the error equation  \eqref{projection--error--phi--u} to obtain
		\begin{align}\label{remarkedadded}
			\begin{split}
				S_1&=\frac{1}{Pe}\mathcal{A}(\varepsilon_h^{\bm{p}}, \varepsilon_h^{\phi}, \varepsilon_h^{\widehat{\phi}};\bm{\Pi}_k^o\bm{\Theta}, -\Pi_{k+1}^o\theta, -\Pi_k^{\partial}\theta)\\
				&=-\mathcal{B}(\varepsilon_h^u, \varepsilon_h^{\widehat{u}};-\Pi_{k+1}^o\theta)
				-\frac{1}{Pe}\langle(\bm{\Pi}_k^o\bm{p}-\bm{p})\cdot\bm{n}, \Pi_{k+1}^o\theta-\Pi_k^{\partial}\theta\rangle_{\partial\mathcal{T}_h}\\
				&\quad-\frac{\alpha}{Pe}\langle h_{\tau}^{-1}(\Pi_{k+1}^o\phi-\phi), \Pi_k^{\partial}\Pi_{k+1}^o\theta-\Pi_k^{\partial}\theta\rangle_{\partial\mathcal{T}_h} 
				+(\bm{\beta}(\Pi_{k+1}^ou-u), \nabla \Pi_{k+1}^o\theta)_{\mathcal{T}_h}
				\\
				&\quad-\langle \bm{\beta}\cdot\bm{n}(\Pi_k^{\partial}u-u), \Pi_{k+1}^o\theta-\Pi_k^{\partial}\theta\rangle_{\partial\mathcal{T}_h}
				\\
				&=\sum_{i=1}^{5}T_i. 
			\end{split}	
		\end{align}
		
		{ For the term $T_1$, using the definition of $\mathcal{B}$, the fact that $\nabla\cdot\bm{\beta}=0$, integration by parts, and since both $\varepsilon_h^{\widehat{u}}$ and $  \theta$ are single-valued on each interior face, the corresponding interior-face contributions cancel when the faces are counted with opposite normals; on boundary faces they vanish because $\bm\beta\cdot\bm n=0$ on $\partial\Omega$. Therefore, 
			\[
			\langle \bm{\beta}\cdot\bm{n}\,\varepsilon_h^{\widehat{u}}, \theta\rangle_{\partial \mathcal{T}_h}=0.
			\]
			Hence, 
			\begin{align*}
				T_1
				&=(\bm{\beta}\cdot\nabla\varepsilon_h^u, \Pi_{k+1}^o\theta)_{\mathcal{T}_h}
				-\langle \bm{\beta}\cdot\bm{n}(\varepsilon_h^u-\varepsilon_h^{\widehat{u}}), \Pi_{k+1}^o\theta\rangle_{\partial \mathcal{T}_h}\\
				&=(\bm{\beta}\cdot\nabla\varepsilon_h^u, \Pi_{k+1}^o\theta-\theta)_{\mathcal{T}_h}
				-\langle \bm{\beta}\cdot\bm{n}(\varepsilon_h^u-\varepsilon_h^{\widehat{u}}), \Pi_{k+1}^o\theta-\theta\rangle_{\partial \mathcal{T}_h}\\
				&\quad
				+\Bigl[(\bm{\beta}\cdot\nabla\varepsilon_h^u,\theta)_{\mathcal{T}_h}
				-\langle \bm{\beta}\cdot\bm{n}(\varepsilon_h^u-\varepsilon_h^{\widehat{u}}),\theta\rangle_{\partial \mathcal{T}_h}\Bigr]\\
				&=(\bm{\beta}\cdot\nabla\varepsilon_h^u, \Pi_{k+1}^o\theta-\theta)_{\mathcal{T}_h}
				-\langle \bm{\beta}\cdot\bm{n}(\varepsilon_h^u-\varepsilon_h^{\widehat{u}}), \Pi_{k+1}^o\theta-\theta\rangle_{\partial \mathcal{T}_h}
				-(\varepsilon_h^u,\bm{\beta}\cdot\nabla\theta)_{\mathcal{T}_h}.
			\end{align*}
			By the approximation property \eqref{classical_ine_1} of $\Pi_{k+1}^o$,
			\begin{align}\label{theta-proj}
				\|\Pi_{k+1}^o\theta-\theta\|_{\mathcal{T}_h}
				+\|h_{\tau}^{1/2}(\Pi_{k+1}^o\theta-\theta)\|_{\partial\mathcal{T}_h}
				\le Ch^2\|\theta\|_{H^2(\Omega)}.
			\end{align}
			Therefore, by \eqref{theta-proj}, the Cauchy--Schwarz inequality, and
			$\|\theta\|_{H^1(\Omega)}\le C\|\theta\|_{H^2(\Omega)}$, we have
			\begin{align}\label{appT_1}
				\begin{split}
					|T_1|
					&\le Ch^2\|\theta\|_{H^2(\Omega)}
					\Bigl(
					\|\nabla\varepsilon_h^u\|_{\mathcal{T}_h}
					+\|h_{\tau}^{-1/2}(\varepsilon_h^u-\varepsilon_h^{\widehat{u}})\|_{\partial\mathcal{T}_h}
					\Bigr)
					+C\|\varepsilon_h^u\|_{\mathcal{T}_h}\|\theta\|_{H^2(\Omega)}.
				\end{split}
			\end{align}
			Taking $f_1=-\varepsilon_h^{\phi}$ and $f_2=-\varepsilon_h^{u}$ in the regularity condition \eqref{regularity--condition-phi-u}, we obtain
			\begin{align}\label{theta}
				\|\theta\|_{H^2(\Omega)}
				\le
				C\bigl(\|\varepsilon_h^{u}\|_{\mathcal{T}_h}+\|\varepsilon_h^{\phi}\|_{\mathcal{T}_h}\bigr).
			\end{align}
			Moreover,
			Using the result in \eqref{error-u-projection} and the  approximation property \eqref{classical_ine_1} of $\Pi_{k+1}^o$ we have
			\begin{align}\label{vare_u}
				\|\varepsilon_h^u\|_{\mathcal T_h} \le Ch^{k+2}|u|_{k+2}.
			\end{align}
			Substituting \eqref{theta} and \eqref{vare_u} into \eqref{appT_1}, and then using the approximation result in \cite[Lemma~5.7]{ChenGang2018Otso} together with Young's inequality, we conclude that
			\begin{align*}
				|T_1|
				\le Ch^{2k+4}|u|_{k+2}^2+\frac{1}{32}\|\varepsilon_h^{\phi}\|_{\mathcal{T}_h}^2.
		\end{align*}}

		{For the term $T_2$, notice that $ \bm{p}\in \bm{H}(\text{div};\Omega)$ 
			and satisfies $\bm{p}\cdot\bm{n}=\bm{0}$ on  $\partial\Omega$,  leading to
			\begin{align*}
				\langle \bm{p}\cdot\bm{n}, \Pi_k^{\partial}\theta\rangle_{\partial\mathcal{T}_h}=0= 
				\langle \bm{p}\cdot\bm{n}, \theta\rangle_{\partial\mathcal{T}_h}.
			\end{align*}
			It follows from \eqref{theta}   that 
			\begin{align*}
				T_2&=\frac{1}{Pe}\langle\bm{p}\cdot\bm{n}, \Pi_{k+1}^o\theta-\Pi_{k}^{\partial}\theta\rangle_{\partial \mathcal{T}_h}
				-\frac{1}{Pe}\langle\bm{\Pi}_k^o\bm{p}\cdot\bm{n}, \Pi_{k+1}^o\theta-\Pi_{k}^{\partial}\theta\rangle_{\partial\mathcal{T}_h}\\
				&=\frac{1}{Pe}\langle\bm{p}\cdot\bm{n}, \Pi_{k+1}^o\theta-\theta\rangle_{\partial \mathcal{T}_h}
				-\frac{1}{Pe}\langle\bm{\Pi}_k^o\bm{p}\cdot\bm{n}, \Pi_{k+1}^o\theta-\theta\rangle_{\partial\mathcal{T}_h}\\
				&=\frac{1}{Pe}\langle(\bm{p}-\bm{\Pi}_k^o\bm{p})\cdot\bm{n}, \Pi_{k+1}^o\theta-\theta\rangle_{\partial \mathcal{T}_h}\le Ch^{k+2}|\phi|_{k+2}\|\theta\|_{H^2(\Omega)}\le Ch^{2k+4}+\frac{1}{32}\|\varepsilon_h^{\phi}\|_{\mathcal{T}_h}^2.
			\end{align*}
			
			For the term $T_3$, using the approximation property of $ L^2 $ projection with \eqref{theta} to drive
			\begin{align*}
				T_3&=-\frac{\alpha}{Pe}\langle h_{\tau}^{-1}(\Pi_{k+1}^o\phi-\phi), \Pi_{k}^{\partial}\Pi_{k+1}^o\theta-\Pi_k^{\partial}\theta\rangle_{\partial \mathcal{T}_h}\\
				&\le Ch^{-1}\|\Pi_{k+1}^o\phi-\phi\|_{\partial\mathcal{T}_h}\|\Pi_{k}^{\partial}(\Pi_{k+1}^o\theta-\theta)\|_{\partial\mathcal{T}_h}\le Ch^{k+2}|\phi|_{k+2}\|\theta\|_{H^2(\Omega)}\\
				&\le Ch^{k+2}|\phi|_{k+2}(\|\varepsilon_h^{\phi}\|_{\mathcal{T}_h}+Ch^{k+2})\le Ch^{2k+4}+\frac{1}{32}\|\varepsilon_h^{\phi}\|_{\mathcal{T}_h}^2.
			\end{align*}
			
			Likewise,  for $ T_4 $,  we have
			\begin{align*}
				T_4&=(\bm{\beta}(\Pi_{k+1}^o u-u), \nabla\Pi_{k+1}^o\theta)_{\mathcal{T}_h}\le C\|\bm{\beta}\|_{\infty}\|\Pi_{k+1}^o u-u\|_{\mathcal{T}_h}\|\nabla\Pi_{k+1}^o\theta\|_{\mathcal{T}_h}\\
				&\le Ch^{k+2}|u|_{k+2}\|\theta\|_{H^2(\Omega)}\le Ch^{2k+4}+\frac{1}{32}\|\varepsilon_h^{\phi}\|_{\mathcal{T}_h}^2.
			\end{align*}
			
			For the term $T_5$,  since both $u$,  $\Pi_k^\partial u$ and $\Pi_k^\partial \theta$ are single-valued on each interior face, the corresponding interior-face contributions cancel when the faces are counted with opposite normals; on boundary faces they vanish because $\bm\beta\cdot\bm n=0$ on $\partial\Omega$. Therefore,
			\begin{align*}
				\langle \bm{\beta}\cdot\bm{n}u, \Pi_k^{\partial}\theta\rangle_{\partial\mathcal{T}_h}&=0= 
				\langle \bm{\beta}\cdot\bm{n}u, \theta\rangle_{\partial\mathcal{T}_h}, \\
				\langle \bm{\beta}\cdot\bm{n}\Pi_{k}^{\partial}u, \Pi_k^{\partial}\theta\rangle_{\partial\mathcal{T}_h}&=0= 
				\langle \bm{\beta}\cdot\bm{n}\Pi_{k}^{\partial}u, \theta\rangle_{\partial\mathcal{T}_h}.
			\end{align*}
			Thus one can estimate $ T_5 $ together with \eqref{theta} as:
			\begin{align*}
				T_5&=\langle\bm{\beta}\cdot\bm{n}u, \Pi_{k+1}^o\theta-\Pi_k^{\partial}\theta\rangle_{\partial \mathcal{T}_h}
				-\langle\bm{\beta}\cdot\bm{n}\Pi_{k}^{\partial}u, \Pi_{k+1}^o\theta-\Pi_k^{\partial}\theta\rangle_{\partial \mathcal{T}_h}\\
				&=\langle\bm{\beta}\cdot\bm{n}u, \Pi_{k+1}^o\theta-\theta\rangle_{\partial \mathcal{T}_h}
				-\langle\bm{\beta}\cdot\bm{n}\Pi_{k}^{\partial}u, \Pi_{k+1}^o\theta-\theta\rangle_{\partial \mathcal{T}_h}\\
				&=\langle\bm{\beta}\cdot\bm{n}(u-\Pi_{k}^{\partial}u), \Pi_{k+1}^o\theta-\theta\rangle_{\partial \mathcal{T}_h}\\
				&\le C\|\bm{\beta}\|_{\infty}\|u-\Pi_{k}^{\partial}u\|_{\partial\mathcal{T}_h}\|\Pi_{k+1}^o\theta-\theta\|_{\partial\mathcal{T}_h}\\
				&\le Ch^{k+2}|u|_{k+1}\|\theta\|_{H^2(\Omega)}\\
				&\le Ch^{2k+4}+\frac{1}{32}\|\varepsilon_h^{\phi}\|_{\mathcal{T}_h}^2.
			\end{align*}
		}
		Combing all estimates for the terms $ \{T_i\}_{i=1}^5$,  one derives
		\begin{align*}
			S_1\le Ch^{2k+4}(|\phi|_{k+2}^2+|u|_{k+1}^2+|u|_{k+2}^2)+\frac{3}{32}\|\varepsilon_h^{\phi}\|_{\mathcal{T}_h}^2.
		\end{align*}	
		Using the regularity conditions \eqref{regularity--condition-phi-u} and \eqref{regularity--condition-u},  we obtain:
		\begin{align*}
			S_3+S_4&\le Ch\|\bm{\Theta}\|_{H^1(\Omega)}\|h_{\tau}^{-1/2}(\varepsilon_h^{\phi}-\varepsilon_h^{\widehat{\phi}})\|_{\partial\mathcal{T}_h}+Ch\|\theta\|_{H^2(\Omega)}\|h_{\tau}^{-1/2}(\Pi_k^{\partial}\varepsilon_h^{\phi}-\varepsilon_h^{\widehat{\phi}})\|_{\partial\mathcal{T}_h}\\
			&\le C(\|\varepsilon_h^u\|_{\mathcal{T}_h}+\|\varepsilon_h^{\phi}\|_{\mathcal{T}_h})h^{k+2}(|u|_{k+1}+|u|_{k+2}+|\phi|_{k+2})\\
			&\le
			{Ch^{2k+4}|u|_{k+1}^2}+ Ch^{2k+4}|u|_{k+2}^2+Ch^{2k+4}|\phi|_{k+2}^2+\frac{1}{32}\|\varepsilon_h^{\phi}\|_{\mathcal{T}_h}^2.
		\end{align*}
		Notice the fact 
			\begin{align*}
				\langle\bm{\beta}\cdot\bm{n}\Pi_{k}^{\partial}\psi, \varepsilon_h^{\widehat{\phi}}\rangle_{\partial\mathcal{T}_h}=0=\langle\bm{\beta}\cdot\bm{n}\psi, \varepsilon_h^{\widehat{\phi}}\rangle_{\partial\mathcal{T}_h}, 
			\end{align*}	
			using the definition of $ \mathcal{B} $ and integration by parts  to drive
			\begin{align*}
				S_2+S_5+S_6&=(\bm{\beta}\cdot\Pi_{k+1}^o\psi, \nabla\varepsilon_h^{\phi})_{\mathcal{T}_h}
				-\langle\bm{\beta}\cdot\bm{n}\Pi_{k}^{\partial}\psi, \varepsilon_h^{\phi}\rangle_{\partial\mathcal{T}_h}\\
				&\quad -(\bm{\beta}(\Pi_{k+1}^o\psi-\psi), \nabla\varepsilon_h^{\phi})_{\mathcal{T}_h}
				+\langle\bm{\beta}\cdot\bm{n}(\Pi_{k}^{\partial}\psi-\psi), \varepsilon_h^{\phi}-\varepsilon_h^{\widehat{\phi}}\rangle_{\partial \mathcal{T}_h}
				\\
				&=(\bm{\beta}\psi, \nabla\varepsilon_h^{\phi})_{\mathcal{T}_h}
				-\langle\bm{\beta}\cdot\bm{n}\psi, \varepsilon_h^{\phi}\rangle_{\partial \mathcal{T}_h}
				-\langle\bm{\beta}\cdot\bm{n}\Pi_k^{\partial}\psi, \varepsilon_h^{\widehat{\phi}}\rangle_{\partial\mathcal{T}_h}
				+\langle\bm{\beta}\cdot\bm{n}\psi, \varepsilon_h^{\widehat{\phi}}\rangle_{\partial\mathcal{T}_h}\\
				&=-(\bm{\beta}\cdot\nabla\psi, \varepsilon_h^{\phi})_{\mathcal{T}_h}\\
				&\le C\|\beta\|_{\infty}\|\psi\|_{H^1(\Omega)}\|\varepsilon_h^{\phi}\|_{\mathcal{T}_h}\\
				&\le Ch^{2k+4}|u|_{k+2}^2+\frac{1}{32}\|\varepsilon_h^{\phi}\|_{\mathcal{T}_h}^2, 
			\end{align*}	
			where the second from last inequality we use the regularity condition 
			\begin{align*}
				\|\psi\|_{H^2(\Omega)}\le C\|\varepsilon_h^{u}\|_{\mathcal{T}_h}\le Ch^{k+2}|u|_{k+2}.
			\end{align*}
		
		Combining all above estimates for $\{S_i\}_{i=1}^6$,  one derives
		\begin{align}\label{appendix-vare-h-phi}
			\|\varepsilon_h^{\phi}\|_{\mathcal T_h}\le Ch^{k+2}(|u|_{k+1}+|u|_{k+2}+|\phi|_{k+2}).
		\end{align}
		Together with the triangle inequality,  \eqref{error-phi-projection} follows.
		
		\begin{remark}\label{projection--proof--remark}
			We note that if the operator $\mathcal{B}$ is used in the form of \eqref{def_B_1},  then the conclusions of \eqref{error-u-projection}–\eqref{error-p-projection} remain valid. However,  the convergence rate of $\|\phi - \phi_{I h}\|_{\mathcal{T}_h}$ is suboptimal when $k = 0$,  but optimal for $k \geq 1$.
			
			The proof proceeds similarly to the argument above. By following the previous proof closely,  we find that an additional term arises at each step. However,  these extra terms do not degrade the overall convergence order—except for one specific term that is added to the quantity $R_1$ in \eqref{remarkedadded}. This additional term is given by: 
			\begin{align*} 
				T_6 = -\langle \tau_c(\Pi_{k+1}^o u - \Pi_k^{\partial} u),  \Pi_{k+1}^o \theta - \Pi_k^{\partial} \theta \rangle_{\partial \mathcal{T}_h}. \end{align*}
			We now estimate the term $T_6$:
			\begin{align*}
				T_6&\le C\|\Pi_{k+1}^ou-\Pi_{k}^{\partial}u\|_{\partial\mathcal{T}_h}\|\Pi_{k+1}^o\theta-\Pi_k^{\partial}\theta\|_{\partial\mathcal{T}_h}\\
				&\le C(\|\Pi_{k+1}^ou-u\|_{\partial\mathcal{T}_h}+\|u-\Pi_{k}^{\partial}u\|_{\partial\mathcal{T}_h})(\|\Pi_{k+1}^o\theta-\theta\|_{\partial\mathcal{T}_h}+\|\theta-\Pi_k^{\partial}\theta\|_{\partial\mathcal{T}_h})\\
				&\le C(h^{k+3/2}|u|_{k+2}+h^{k+1/2}|u|_{k+1})h^{1/2+\min\{1, k\}}\|\theta\|_2\\
				&\le Ch^{k+1+\min\{1, k\}}\|\theta\|_2\\
				&\le Ch^{k+1+\min\{1, k\}}(\|\varepsilon_h^u\|_{\mathcal{T}_h}+\|\varepsilon_h^{\phi}\|_{\mathcal{T}_h})\\
				&\le Ch^{2k+2+\min\{2, 2k\}}+\frac{1}{32}\|\varepsilon_h^{\phi}\|_{\mathcal{T}_h}^2.
			\end{align*}
			Hence,  the optimal convergence order for $\|\varepsilon_h^{\phi}\|_{\mathcal{T}_h}$ fails to hold when $k = 0$,  but is preserved when $k \ge 1$.
		\end{remark}

		\section{Conclusion} In this work,  we propose an HDG method combined with a convex-concave splitting temporal discretization to solve the convective Cahn-Hilliard equation. We perform a stability analysis and prove that our numerical scheme is unconditionally stable. Additionally,  we establish optimal convergence rates in the $L^2$ norm for both the scalar and flux variables for any polynomial degree $k \geq 0$.
		
		There are several promising directions for future research. In the current setting,  the velocity field is prescribed. A natural extension is to couple the velocity with the fluid flow,  governed by the Navier-Stokes equations. Another important direction is the development of efficient solvers for the proposed system. We aim to explore time-parallel techniques to accelerate the computations. Furthermore,  designing fast solvers for the underlying linear systems remains an important challenge.

		
		\bibliographystyle{plain}
		\bibliography{mybib}
	\end{document}